\documentclass[twoside,11pt,reqno]{amsart} 
\usepackage{amsmath,amssymb,amscd,mathrsfs,stmaryrd,wasysym,latexsym}

\makeatletter

\newtheorem{Theorem}{Theorem}[section]
\newtheorem{Corollary}[Theorem]{Corollary}
\newtheorem{Proposition}[Theorem]{Proposition}

\newtheorem{Hypothesis}[Theorem]{Hypothesis}
\newtheorem{Lemma}[Theorem]{Lemma}
\theoremstyle{remark}
\newtheorem{Remark}[Theorem]{Remark}

\numberwithin{equation}{section}

\newbox\squ  
\setbox\squ=\hbox{\vrule width.6pt
   \vbox{\hrule height.6pt width.4em\kern1ex
         \hrule height.6pt}%
   \vrule width.6pt}

\def\lexeq{\leq_{\operatorname{lex}}}
\def\lex{<_{\operatorname{lex}}}

\def\gex{>_{\operatorname{lex}}}

\def\codeg{\operatorname{codeg}}

\def\op{\operatorname{op}}

\def\ad{\operatorname{ad}}

\def\f{f}
\def\e{e}

\def\Rep#1{\operatorname{Rep}(#1)}
\def\mod#1{\underline{\operatorname{Mod}}(#1)}
\def\rep#1{\underline{\operatorname{Re}}\!\operatorname{p}(#1)}
\def\proj#1{\underline{\operatorname{Pro}}\!\operatorname{j}(#1)}
\def\Proj#1{\operatorname{Proj}(#1)}
\def\Mod#1{\operatorname{Mod}(#1)}
\def\id{\operatorname{id}}

\def\db{\delta}
\def\dd{\eps}
\def\Id{\operatorname{Id}}

\def\defect{\operatorname{def}}
\def\A{{\mathbb A}}
\def\C{{\mathbb C}}
\def\Q{{\mathbb Q}}
\def\Z{{\mathbb Z}}

\def\onto{{\twoheadrightarrow}}

\def\0{{\bar 0}}
\def\1{{\bar 1}}
\def\pr{{\operatorname{pr}}}
\def\infl{{\operatorname{infl}}}

\def\St{{\mathscr{T}}}
\def\T{{\mathtt T}}
\def\Laurent{\mathscr A}
\def\Stab{{\mathtt S}}
\def\Par{{\mathscr P}^\La}
\def\Part{{\mathscr P}^{\La^t}}
\def\RPar{{\mathscr{RP}}^\La}
\def\RPart{{\mathscr{RP}}^{\La^t}}
\def\RRPar{\widetilde{\mathscr{RP}}^\La}

\def\qdim{{\operatorname{qdim}}}
\def\hom{\underline{\operatorname{Hom}}}
\def\HOM{{\operatorname{HOM}}}
\def\Hom{{\operatorname{Hom}}}

\def\End{{\operatorname{End}}}
\def\END{{\operatorname{END}}}

\def\Ind{{\operatorname{Ind}}}

\def\res{{\operatorname{res}}}
\def\Res{{\operatorname{Res}}}

\def\cha{{\operatorname{char}\,}}

\def\im{{\operatorname{im}}}
\def\soc{{\operatorname{soc}\:}}

\def\CH{{\operatorname{ch}_q\:}}
\def\wt{{\operatorname{wt}}}
\def\height{{\operatorname{ht}}}

\def\cont{{\operatorname{cont}}}

\def\bid{\hbox{\boldmath{$1$}}}
\def\bi{\text{\boldmath$i$}}
\def\bj{\text{\boldmath$j$}}

\def\kk{{\tilde{k}}}
\def\ddd{{\tilde{d}}}

\def\hla{{\bar{\la}}}
\def\hmu{{\bar{\mu}}}
\def\hnu{{\bar{\nu}}}

\def\eps{{\varepsilon}}
\def\phi{{\varphi}}
\def\emptyset{{\varnothing}}
\def\ga{{\gamma}}
\def\Ga{{\Gamma}}
\def\la{{\lambda}}
\def\La{{\Lambda}}
\def\de{{\delta}}
\def\De{{\Delta}}
\def\al{{\alpha}}

\def\g{{\mathfrak g}}
\def\h{{\mathfrak h}}

\def\nslash{\:\notslash\:}

\def\underbar{\mathpalette\@underbar}
\def\@underbar#1#2{\settowidth{\@tempdimb}{$#1#2$}\@tempdimb=0.8\@tempdimb
                   \ooalign{$#1#2$\crcr%
                         \hfil\rule[-.5mm]{\@tempdimb}{.4pt}\hfil}}

\newdimen\hoogte    \hoogte=14pt    
\newdimen\breedte   \breedte=14pt   
\newdimen\dikte     \dikte=0.5pt    

\newenvironment{young}{\begingroup
       \def\vr{\vrule height0.8\hoogte width\dikte depth 0.2\hoogte}
       \def\fbox##1{\vbox{\offinterlineskip
                    \hrule height\dikte
                    \hbox to \breedte{\vr\hfill##1\hfill\vr}
                    \hrule height\dikte}}
       \vbox\bgroup \offinterlineskip \tabskip=-\dikte \lineskip=-\dikte
            \halign\bgroup &\fbox{##\unskip}\unskip  \crcr }
       {\egroup\egroup\endgroup}
\def\diagram#1{\relax\ifmmode\vcenter{\,\begin{young}#1\end{young}\,}\else%
              $\vcenter{\,\begin{young}#1\end{young}\,}$\fi}

\begin{document}

\title[Graded decomposition numbers]{Graded decomposition numbers for cyclotomic Hecke algebras}
\author{Jonathan Brundan and Alexander Kleshchev}

\begin{abstract}
In recent joint work with Wang, we have constructed graded Specht modules
for cyclotomic Hecke algebras. 
In this article, we prove a
graded version of the Lascoux-Leclerc-Thibon conjecture,
describing the
decomposition numbers of graded Specht modules
over a field of characteristic zero.
\end{abstract}
\thanks{{\em 2000 Mathematics Subject Classification:} 20C08.}
\thanks{Supported in part by NSF grant 
DMS-0654147.}
\address{Department of Mathematics, University of Oregon, Eugene, USA.}
\email{brundan@uoregon.edu, klesh@uoregon.edu}
\maketitle

\vspace{-2mm}

\section{Introduction}\label{SIntro}

Since the classic work of Bernstein and Zelevinsky \cite{BZ},
the representation theory of the affine Hecke algebra $H_d$ associated
to the symmetric group $\Sigma_d$ 
has been a fundamental topic in representation theory from many points of view.
For brevity in this introduction,
we discuss only the situation when $H_d$ is defined 
over the ground field $\C$ at parameter
$1 \neq \xi \in \C^\times$ 
that is a primitive $e$th root of unity.
In \cite{Ariki},
building on powerful geometric results of Kazhdan and Lusztig \cite{KL}
and Ginzburg \cite[Chapter 8]{CG}, 
Ariki established a remarkable connection between the
representation theory of certain finite dimensional 
quotients of $H_d$, known as {cyclotomic Hecke algebras},
and the canonical bases of integrable
highest weight modules over the affine Lie algebra
$\widehat{\mathfrak{sl}}_e(\C)$.
In a special case, this connection had been suggested earlier
by Lascoux, Leclerc and Thibon \cite{LLT}.
Similar results were announced by Grojnowski following \cite{G1},
but the proofs were never published.

To recall some of these results in a little more detail,
let $\La$ be a dominant integral weight of level $l$ for
$\widehat{\mathfrak{sl}}_e(\C)$. Let
$V(\La)_\C$ be the corresponding integrable highest weight module
and fix a highest weight vector $v_\La \in V(\La)_\C$.
To the weight $\La$ we associate cyclotomic Hecke algebras
$H^\La_d$ for each $d \geq 0$; see $\S$\ref{sah}.
Letting $\proj{H_d^\La}$ denote the category of
finitely generated projective $H_d^\La$-modules and writing
$[\proj{H_d^\La}]_\C$ for its complexified Grothendieck group,
Ariki showed that there is a unique $\C$-linear isomorphism
$$
\underline{\delta}:
V(\La)_\C \stackrel{\sim}{\rightarrow} \bigoplus_{d \geq 0} [\proj{H_d^\La}]_\C
$$
such that $v_\La$ maps to the class
of the regular $H_0^\La$-module, and the actions of the Chevalley generators
$e_i, f_i \in \widehat{\mathfrak{sl}}_e(\C)$ correspond to certain exact
{\em $i$-restriction} and {\em $i$-induction functors} on the
 Hecke algebra side.

Now the key result obtained by Ariki in \cite[Theorem 4.4]{Ariki}
can be formulated as follows:
the isomorphism $\underline{\delta}$ maps the 
Kashiwara-Lusztig canonical basis for $V(\La)_\C$
to the basis of the Grothendieck group arising from the
isomorphism classes of projective indecomposable modules. 
Ariki then applied this theorem to compute the
decomposition numbers
of {\em Specht modules}, for which some foundational results
in levels $l > 1$
were developed subsequently by Dipper, James and Mathas \cite{DJM}.
In level one this gave a proof of the Lascoux-Leclerc-Thibon
conjecture from \cite{LLT}
concerning decomposition numbers of the Iwahori-Hecke algebra of 
type $A$ at an $e$th root of unity over $\C$;
moreover it generalized the conjecture to higher levels.

Recently, there have been some exciting new developments
thanks to works of Khovanov and Lauda \cite{KL1, KL2}
and Rouquier \cite{Ro}, who have independently introduced a new family
of algebras attached to Cartan matrices.
For the rest of the introduction,
we let $R_d$ denote the Khovanov-Lauda-Rouquier algebra
of degree $d$ attached to the Cartan matrix
of type $\widehat{\mathfrak{sl}}_e$; 
we mean the direct sum over all $\alpha \in Q_+$ of height $d$
of the algebras $R_\alpha$ defined by generators and relations
in $\S$\ref{sqnh} below.
Unlike the affine Hecke algebra $H_d$,
the algebra $R_d$ is $\Z$-graded in a canonical way.

In \cite[$\S$3.4]{KL1}, Khovanov and Lauda also introduced certain  ``cyclotomic''
finite dimensional graded quotients $R_d^\La$ of $R_d$
(see $\S$\ref{slump}),
and conjectured a result which can be viewed as a 
graded version of Ariki's categorification theorem as
formulated above. Remarkably, 
the 
Khovanov-Lauda categorification conjecture
makes equally good sense in any type.
One of the main results of this article proves 
the conjecture in the
$\widehat{\mathfrak{sl}}_e$-case. To do this, we 
exploit
an explicit algebra isomorphism
$\rho:R^\La_d\stackrel{\sim}{\rightarrow} H^\La_d$ constructed 
in \cite{BKyoung}, which allows us to lift
existing results about $H^\La_d$ to $R^\La_d$, incorporating
additional information about gradings as we go.

To give a little more detail, 
the algebra $R_d^\La$ is graded, so it makes sense to consider
the category $\Proj{R_d^\La}$ of finitely generated projective {\em graded} 
$R_d^\La$-modules. The Grothendieck group
$[\Proj{R_d^\La}]$ is a $\Z[q,q^{-1}]$-module, with multiplication
by $q$ corresponding to shifting the grading on a module up by one.
Let $[\Proj{R_d^\La}]_{\Q(q)} := \Q(q) \otimes_{\Z[q,q^{-1}]}
[\Proj{R_d^\La}]$.
Let $V(\La)$ be the integrable highest weight module
for the quantized enveloping algebra $U_q(\widehat{\mathfrak{sl}}_e)$ 
over the field $\Q(q)$, with highest weight vector $v_\La$.
Combining results from \cite{G, KL1, CR, Ro},
we show to start with that there is a unique $\Q(q)$-linear isomorphism
$$
\delta:V(\La) \stackrel{\sim}{\rightarrow}
\bigoplus_{d \geq 0} [\Proj{R^\La_d}]_{\Q(q)}
$$
such that $v_\La$ maps to the class of the regular $R^\La_0$-module
and the actions of the Chevalley
generators $E_i, F_i \in U_q(\widehat{\mathfrak{sl}}_e)$
correspond to 
graded analogues of the $i$-restriction and $i$-induction
functors from before; see $\S$\ref{sir}.

Moreover, we show that the
isomorphism $\delta$ maps the canonical basis
for $V(\La)$ to the basis of the Grothendieck
group arising from the isomorphism classes of 
indecomposable projective graded modules 
that are self-dual 
with respect to a certain duality $\circledast$; see $\S$\ref{scd}.
Our proof of this 
relies ultimately on
Ariki's original categorification theorem from \cite{Ariki}.

In joint work with Wang \cite{BKW}, we have also defined
graded versions of Specht modules for the algebras $R^\La_d$.
Another of our main results gives an explicit formula
for the decomposition numbers of graded Specht modules.
This should be regarded as a graded version of the 
Lascoux-Leclerc-Thibon conjecture
(generalized to higher levels).
It shows that the decomposition numbers of graded Specht
modules are obtained by expanding the ``standard monomials''
in $V(\La)$ in terms of the dual-canonical basis;
see $\S$\ref{sgdn} and $\S$\ref{dcqc} for details.

The results of this article fit naturally into the general
framework of $2$-representations of $2$-Kac-Moody algebras
 developed by Rouquier in
\cite{Ro}; see also \cite{KL3}.
While writing up this work, 
we have learnt of an announcement by Rouquier
indicating that he has found
a direct geometric proof of the Khovanov-Lauda categorification conjecture
that is valid in arbitrary type, although
details are not yet available.
More recently still, Varagnolo and Vasserot have released a preprint
in which they prove the Khovanov-Lauda categorification conjecture
at the affine level in arbitrary simply-laced type; see \cite{VV3}.
We point out however that these results do not immediately imply the graded version of the Lascoux-Leclerc-Thibon conjecture proved here, since 
for that one needs to deal with 
{Specht modules} over the cyclotomic quotients.

\vspace{2mm}

We end the introduction with a brief guide to the rest of the article,
indicating some of the other things to be found here.
Section 2 is primarily devoted to recalling the definition of
the algebras $R_d$ in type $\widehat{\mathfrak{sl}}_e$, 
and then reviewing some of the foundational results
proved about them in \cite{KL1}.

In section 3 we review the construction of the irreducible
highest weight module $V(\La)$ over $U_q(\widehat{\mathfrak{sl}}_e)$
as a summand of Fock space.
At the same time, we construct various bases for these modules, 
paralleling the
setup of \cite[$\S$2]{BKariki} closely. This part of the story is
surprisingly lengthy as 
there are some subtle 
combinatorial issues surrounding the triangularity
of the standard monomials in $V(\La)$; see $\S$\ref{tzero}.
Unlike almost all of the literature in the subject,
our approach emphasizes the dual-canonical basis 
rather than the canonical basis.

In section 4 we consider the cyclotomic quotients
$R^\La_d$ of $R_d$ introduced originally in
\cite[$\S$3.4]{KL1}.
We use the isomorphism between $R^\La_d$ and $H^\La_d$
from \cite{BKyoung}
to quickly deduce the classification
of irreducible graded $R^\La_d$-modules from Grojnowski's classification
of irreducible $H^\La_d$ in terms of crystal graphs from  \cite{G};
see $\S$\ref{sb}.
At the same time we lift various
branching rules to the graded setting.
Then we prove the first key categorification theorem,
which identifies $V(\La)$ with the direct sum 
$\bigoplus_{d \geq 0} [\Proj{R^\La_d}]_{\Q(q)}$ as above;
see $\S$\ref{sc}.
As an application, we compute the graded dimension
of $R^\La_d$; see $\S$\ref{sgdf}.
We stress that this part of the development 
makes sense over any ground field, 
and does not depend on any results from geometric representation theory.

In section 5 we lift Ariki's results to the graded setting
to prove simultaneously
the graded version of the Lascoux-Leclerc-Thibon conjecture
and the Khovanov-Lauda conjecture; see
$\S$\ref{sgdn}.
In the course of this we encounter some non-trivial issues related to the parametrization
of irreducible modules: there are two relevant 
parametrizations,
one arising from the crystal graph and the other arising from Specht module
theory; see $\S$\ref{sanother} for the latter.
The identification of the two parametrizations is addressed in
Ariki's work, but we give a self-contained treatment here in order to
keep track of gradings.
We also discuss 
the situation over fields of positive characteristic, introducing
graded analogues of James' adjustment matrices; see $\S$\ref{sgam}.

\vspace{2mm}
\noindent
{\em Acknowledgements.}
We thank Rapha\"el Rouquier for sending
us a preliminary version of \cite{Ro} and an extremely helpful remark about Uglov's construction.
We also thank Mikhail Khovanov for his comments on a previous version, 
as well as Bernard Leclerc, Andrew Mathas and Weiqiang Wang 
for some valuable discussions and e-mail correspondence.

\section{Review of results of Khovanov and Lauda}

Fix an algebraically closed 
field $F$ and an integer $e$ such that either $e = 0$ or $e \geq 2$.
Always $q$ denotes an indeterminate.

\subsection{Cartan integers, weights and roots}\label{SSDynkin}
Let $\Ga$ be the quiver with vertex set  $I := \Z / e\Z$,
and a directed edge from $i$ to $j$ if 
$i \neq j  = i+1$ in $I$.
Thus $\Gamma$ is the quiver of type $A_\infty$ if $e=0$
or $A_{e-1}^{(1)}$ if $e > 0$, with a specific orientation:
\begin{align*}
A_\infty&:\quad\cdots \longrightarrow-2\longrightarrow -1 \longrightarrow 0 \longrightarrow 1 \longrightarrow 
2\longrightarrow \cdots\\
A_{e-1}^{(1)}&:\quad0\rightleftarrows 1
\qquad
\begin{array}{l}
\\
\,\nearrow\:\:\:\searrow\\
\!\!2\,\longleftarrow\, 1
\end{array}
\qquad
\begin{array}{rcl}\\
0&\!\rightarrow\!&1\\
\uparrow&&\downarrow\\
3&\!\leftarrow\!&2
\end{array}
\qquad
\begin{array}{l}
\\
\:\nearrow\quad\searrow\\
\!4\qquad\quad \!1\\
\nnwarrow\quad\quad\,\sswarrow\\
\:\:3\leftarrow 2
\begin{picture}(0,0)
\put(-152.5,41){\makebox(0,0){0}}
\put(-13.5,53.5){\makebox(0,0){0}}
\end{picture}
\end{array}
\qquad \cdots
\end{align*}
The corresponding (symmetric) Cartan matrix
$(a_{i,j})_{i, j \in I}$ is defined by
\begin{equation}\label{ECM}
a_{i,j} := \left\{
\begin{array}{rl}
2&\text{if $i=j$},\\
0&\text{if $i \nslash j$},\\
-1&\text{if $i \rightarrow j$ or $i \leftarrow j$},\\
-2&\text{if $i \rightleftarrows j$}.
\end{array}\right.
\end{equation}
Here the symbols 
$i \rightarrow j$ and $j \leftarrow i$
both indicate that $i \neq j=i+1\neq i-1$,
$i \rightleftarrows j$ indicates that $i \neq j = i+1= i-1$,
and
$i \nslash j$
indicates that $i \neq j \neq i \pm 1$.

Following \cite{Kac}, let $(\h,\Pi,\Pi^\vee)$ be a realization of the Cartan matrix $(a_{i,j})_{i,j\in I}$, so we have the simple roots $\{\al_i\mid i\in I\}$, the fundamental dominant weights $\{\La_i\mid i\in I\}$, and the normalized invariant form $(\cdot,\cdot)$ such that
$$
(\al_i,\al_j)=a_{i,j}, \quad (\La_i,\al_j)=\de_{i,j}\qquad(i,j\in I).
$$
Let 
$Q_+ := \bigoplus_{i \in I} \Z_{\geq 0} \alpha_i$
denote the positive part of the corresponding root lattice.
For $\alpha \in Q_+$, we write $\height(\alpha)$ for the sum of its 
coefficients when expanded in terms of the $\alpha_i$'s.

\subsection{\boldmath The algebra $\mathbf f$}
Let $\mathbf f$ denote Lusztig's algebra from \cite[$\S$1.2]{Lubook}
attached to the Cartan matrix (\ref{ECM}) over the field $\Q(q)$. We adopt the same conventions as
\cite[$\S$3.1]{KL1}, so our $q$ is Lusztig's $v^{-1}$.
To be more precise, denote
$$
[n]:=\frac{q^n-q^{-n}}{q-q^{-1}},\quad [n]!:=[n][n-1]\dots[1],\quad 
\left[
\begin{matrix}
 n   \\
 m
\end{matrix}
\right]
:=\frac{[n]!}{[n-m]![m]!}.
$$
Then $\mathbf f$ 
is the $\Q(q)$-algebra on generators $\theta_i\:(i \in I)$ subject to
the quantum Serre relations
\begin{equation}\label{qserre}
(\ad_q \theta_i)^{1-a_{j,i}}(\theta_j)=0
\end{equation}
where  
\begin{equation}\label{adq}
(\ad_q x)^n (y):=\sum_{m=0}^n(-1)^m
\left[
\begin{matrix}
 n   \\
 m
\end{matrix}
\right]
x^{n-m}yx^m.
\end{equation}
There is a $Q_+$-grading 
$\mathbf f = \bigoplus_{\alpha \in Q_+} \mathbf f_\alpha$
such that $\theta_i$ is of degree $\alpha_i$.
The algebra $\mathbf f$ possesses a bar-involution
$-:\mathbf f \rightarrow \mathbf f$
that is anti-linear with respect to
the field automorphism sending $q$ to $q^{-1}$, such that
$\overline{\theta_i} = \theta_i$ for each $i \in I$.

If we equip $\mathbf f \otimes \mathbf f$ with algebra structure via the rule
$$
(x_1 \otimes x_2) (y_1 \otimes y_2)
= q^{-(\alpha,\beta)} x_1 y_1 \otimes x_2 y_2
$$
for $x_2 \in \mathbf f_{\alpha}$ and $y_1 \in \mathbf f_{\beta}$,
there is a $Q_+$-graded
comultiplication 
$m^*:\mathbf f  \rightarrow \mathbf f \otimes \mathbf f$, which is
the unique algebra homomorphism
such that $\theta_i \mapsto \theta_i \otimes 1 + 1 \otimes \theta_i$
for each $i \in I$.
For $\alpha,\beta \in Q_+$, we let
$$
m_{\alpha,\beta}: \mathbf f_\alpha 
\otimes \mathbf f_\beta 
\rightarrow \mathbf f_{\alpha+\beta},\qquad
m^*_{\alpha,\beta}:\mathbf f_{\alpha+\beta} \rightarrow \mathbf f_\alpha 
\otimes \mathbf f_{\beta}
$$
denote the multiplication and comultiplication maps induced on individual 
weight spaces,
so $m = \sum m_{\alpha,\beta}$ is the multiplication on $\mathbf f$
and $m^* = \sum m_{\alpha,\beta}^*$.

Finally let $\Laurent := \Z[q,q^{-1}]$ and ${_\Laurent} \mathbf f$
be the $\Laurent$-subalgebra of $\mathbf f$ generated by the quantum
divided powers $\theta_i^{(n)} := \theta_i^n / [n]!$.
The bar
involution induces an involution of ${_\Laurent}\mathbf f$, and also
the map $m^*$ restricts to a well-defined comultiplication
$m^*:{_\Laurent}\mathbf f \rightarrow {_\Laurent}\mathbf f \otimes {_\Laurent}\mathbf f$.

\subsection{\boldmath The algebra $R_\alpha$}\label{sqnh}
The symmetric group $\Sigma_d$ acts on the left on the set 
$I^d$ by place permutation. 
The orbits  are the sets
$$
I^\alpha := \{\bi = (i_1,\dots,i_d)\in I^d\:|\:\alpha_{i_1}+\cdots+\alpha_{i_d}=\alpha\}
$$
for each $\alpha \in Q_+$ with $\height(\alpha)=d$.
As usual, 
we let $s_1,\dots,s_{d-1}$ denote the basic transpositions 
in $\Sigma_d$.

For $\alpha\in Q_+$ of height $d$,
let $R_\alpha$ denote the associative, unital $F$-algebra on generators
$\{e(\bi)\:|\:\bi \in I^\alpha\}\cup\{y_1,\dots,y_d\}\cup\{\psi_1,\dots,\psi_{d-1}\}$
subject only to the following relations
for $\bi,\bj\in I^\alpha$ and all 
admissible $r, s$:
\begin{align}
\qquad e(\bi) e(\bj) &= \de_{\bi,\bj} e(\bi);
\hspace{11.3mm}{\textstyle\sum_{\bi \in I^\alpha}} e(\bi) = 1;\label{R1}\\
y_r e(\bi) &= e(\bi) y_r;
\hspace{20mm}\psi_r e(\bi) = e(s_r{ }\bi) \psi_r;\label{R2PsiE}\hspace{22.8mm}\\
\label{R3Y}
y_r y_s &= y_s y_r;\\
\label{R3YPsi}
\psi_r y_s  &= y_s \psi_r\hspace{42.4mm}\text{if $s \neq r,r+1$};\\
\psi_r \psi_s &= \psi_s \psi_r\hspace{41.8mm}\text{if $|r-s|>1$};\label{R3Psi}\\
\psi_r y_{r+1} e(\bi) 
&= 
\left\{
\begin{array}{ll}
(y_r\psi_r+1)e(\bi) &\hbox{if $i_r=i_{r+1}$},\\
y_r\psi_r e(\bi) \hspace{28mm}&\hbox{if $i_r\neq i_{r+1}$};\hspace{10.5mm}
\end{array}
\right.
\label{R6}\\
y_{r+1} \psi_re(\bi) &=
\left\{
\begin{array}{ll}
(\psi_r y_r+1) e(\bi) 
&\hbox{if $i_r=i_{r+1}$},\\
\psi_r y_r e(\bi)  \hspace{28mm}&\hbox{if $i_r\neq i_{r+1}$};
\end{array}
\right.
\label{R5}\\
\psi_r^2e(\bi) &= 
\left\{
\begin{array}{ll}
0&\text{if $i_r = i_{r+1}$},\\
e(\bi)&\text{if $i_r \nslash i_{r+1}$},\\
(y_{r+1}-y_r)e(\bi)&\text{if $i_r \rightarrow i_{r+1}$},\\
(y_r - y_{r+1})e(\bi)&\text{if $i_r \leftarrow i_{r+1}$},\\
(y_{r+1} - y_{r})(y_{r}-y_{r+1}) e(\bi)\!\!\!&\text{if $i_r \rightleftarrows i_{r+1}$};
\end{array}
\right.
 \label{R4}\end{align}\begin{align}
\psi_{r}\psi_{r+1} \psi_{r} e(\bi)
&=
\left\{\begin{array}{ll}
(\psi_{r+1} \psi_{r} \psi_{r+1} +1)e(\bi)&\text{if $i_{r+2}=i_r \rightarrow i_{r+1}$},\\
(\psi_{r+1} \psi_{r} \psi_{r+1} -1)e(\bi)&\text{if $i_{r+2}=i_r \leftarrow i_{r+1}$},\\
\big(\psi_{r+1} \psi_{r} \psi_{r+1} -2y_{r+1}
\\\qquad\:\quad +y_r+y_{r+2}\big)e(\bi)
\hspace{2.9mm}&\text{if $i_{r+2}=i_r \rightleftarrows i_{r+1}$},\\
\psi_{r+1} \psi_{r} \psi_{r+1} e(\bi)&\text{otherwise}.
\end{array}\right.
\label{R7}
\end{align}
There is a unique $\Z$-grading on
$R_\alpha$ such that
each $e(\bi)$ is of degree 0,
each $y_r$ is of degree $2$, and
each $\psi_r e(\bi)$ is of degree $-a_{i_r,i_{r+1}}$.

The algebra $R_\alpha$ is one of the algebras
introduced
by Khovanov and Lauda in 
\cite{KL1, KL2} (except for $e=2$), and was discovered independently 
by Rouquier in \cite{Ro} (in full generality). 

\subsection{Graded algebras and modules}\label{SSGr}
Let $H$ be a $\Z$-graded $F$-algebra.
Let $\Mod{H}$ denote the abelian category of all graded left $H$-modules,
denoting (degree-preserving) homomorphisms in this category
by $\Hom$.
Let $\Rep{H}$ denote
the abelian subcategory of all
finite dimensional graded left $H$-modules and
 $\Proj{H}$ denote the additive subcategory  of 
all finitely generated projective graded left $H$-modules. 
Denote the corresponding Grothendieck groups by
 $[\Rep{H}]$ and $[\Proj{H}]$, respectively.
We view these as $\Laurent$-modules via 
$q^m[M]:=[M\langle m\rangle]$,
where $M\langle m\rangle$ denotes the module obtained by 
shifting the grading up by $m$:
\begin{equation}\label{obvious}
M\langle m\rangle_n=M_{n-m}.
\end{equation}
Given $f = \sum_{n \in \Z} f_n q^n 
\in \Z[[q,q^{-1}]]$ and a graded module $M$,
we allow ourselves to write simply $f \cdot M$
for $\bigoplus_{n \in \Z} M \langle n \rangle^{\oplus f_n}$.

For $n \in \Z$, we let
$$
\Hom_H(M, N)_n := \Hom_H(M \langle n \rangle, N)
= \Hom_H(M, N \langle -n \rangle)
$$
denote the space of all homomorphisms
that are homogeneous of degree $n$,
i.e. they map $M_i$ into $N_{i+n}$ for each $i \in \Z$.
Set
$$
\HOM_H(M,N) := \bigoplus_{n \in \Z} \Hom_H(M,N)_n,
\quad
\END_H(M) := \HOM_H(M,M).
$$
There is a canonical pairing, the {\em Cartan pairing},
\begin{equation}\label{cartanpairing}
\langle.,.\rangle:\Proj{H}\times\Rep{H} \rightarrow \Laurent,
\quad 
\langle[P],[M]\rangle := \qdim\ \HOM_H(P,M),
\end{equation}
where $\qdim \, V$ denotes $\sum_{n \in \Z} q^n \dim V_n$
for any finite dimensional 
graded vector space $V$.
Note the Cartan pairing is {\em sesquilinear} (anti-linear in the first argument, linear in the second).

Occasionally, we will need to forget the grading on $H$ and work with
ordinary ungraded $H$-modules.
To avoid confusion in the ungraded setting, we denote the category of 
all left $H$-modules (resp.\ finite dimensional left $H$-modules,
resp.\ finitely generated projective left $H$-modules) 
by $\mod{H}$ (resp.\ $\rep{H}$, resp.\ $\proj{H}$).
We denote homomorphisms in these categories by $\hom$.
Let
 $[\rep{H}]$ (resp.\ $[\proj{H}]$) denote the Grothendieck group of
$\rep{H}$ (resp.\ $\proj{H}$).
Given a graded module $M$, we write $\underline{M}$ for the
ungraded module obtained from it by forgetting the grading.
For $M, N \in \Rep{H}$, we have that
\begin{equation}\label{Ehom}
\hom_H(\underline{M}, \underline{N}) = \underline{\HOM_H(M,N)}.
\end{equation}
The following lemmas summarize some standard facts:

\begin{Lemma}[{\cite[Theorem 4.4.6, Remark 4.4.8]{NO}}]\label{no1}
If $M$ is any finitely generated graded $H$-module,
the radical of $\underline{M}$ is a graded submodule of $M$.
\end{Lemma}

\begin{Lemma}[{\cite[Theorem 4.4.4(v)]{NO}}]\label{no2}
If $L \in \Rep{H}$ is irreducible then
$\underline{L} \in \rep{H}$ is irreducible too.
\end{Lemma}

\begin{Lemma}[{\cite[Theorem 9.6.8]{NO}, \cite[Lemma 2.5.3]{BGS}}]\label{no3}
Assume that $H$ is finite dimensional.
If $K \in \rep{H}$ is irreducible,
then there exists an irreducible $L \in \Rep{H}$ 
such that
$\underline{L} \cong K$.
Moreover, $L$ is unique up to isomorphism
and grading shift.
\end{Lemma}

Given $M, L \in \Rep{H}$ with $L$ irreducible, 
we write $[M:L]_q$ for the {\em $q$-composition multiplicity},
i.e. $[M:L]_q := \sum_{n \in \Z} a_n q^n$ where $a_n$ is the multiplicity
of $L\langle n\rangle$ in a graded composition series of $M$.
In view of Lemma~\ref{no2}, we recover the ordinary composition multiplicity
$[\underline{M}:\underline{L}]$ from
$[M:L]_q$ on setting $q$ to $1$.

\subsection{Induction and restriction}
Given $\alpha, \beta \in Q_+$, we set 
$$
R_{\alpha,\beta} := R_\alpha \otimes 
R_\beta,
$$ 
viewed as an algebra in the usual way.
We denote 
the ``outer'' tensor product of an $R_\alpha$-module $M$ and an $R_\beta$-module 
$N$ by $M \boxtimes N$.

There is an obvious embedding
of $R_{\alpha,\beta}$
into the algebra $R_{\alpha+\beta}$
mapping $e(\bi) \otimes e(\bj)$ to $e(\bi\bj)$,
where $\bi\bj$ denotes the concatenation of the two sequences. It is not
a {\em unital} algebra homomorphism. We denote the image of the identity
element of $R_{\alpha,\beta}$ under this map by 
$$
e_{\alpha,\beta} = \sum_{\bi \in I^\alpha, \bj \in I^\beta} e(\bi\bj).
$$
Let $\Ind_{\alpha,\beta}^{\alpha+\beta}$ and $\Res_{\alpha,\beta}^{\alpha+\beta}$
denote the corresponding induction and restriction functors, so
\begin{align}
\Ind_{\alpha,\beta}^{\alpha+\beta} &:= R_{\alpha+\beta} e_{\alpha,\beta}
\otimes_{R_{\alpha,\beta}} ?:\Mod{R_{\alpha,\beta}} \rightarrow \Mod{R_{\alpha+\beta}},\\
\Res_{\alpha,\beta}^{\alpha+\beta} &:= e_{\alpha,\beta} R_{\alpha+\beta}
\otimes_{R_{\alpha+\beta}} ?:\Mod{R_{\alpha+\beta}}\rightarrow \Mod{R_{\alpha,\beta}}.
\end{align}
Note $\Res_{\alpha,\beta}^{\alpha+\beta}$ is just left multiplication by
the idempotent $e_{\alpha,\beta}$, so it is exact and sends finite 
dimensional modules to
finite dimensional modules. 
By \cite[Proposition 2.16]{KL1},
$e_{\alpha,\beta} R_{\alpha+\beta}$ is a graded free left $R_{\alpha,\beta}$-module of finite rank,
so $\Res_{\alpha,\beta}^{\alpha+\beta}$ sends 
finitely generated projective modules to 
finitely generated projective modules.
The functor $\Ind_{\alpha,\beta}^{\alpha+\beta}$ is left adjoint to $\Res_{\alpha,\beta}^{\alpha+\beta}$,
so it also
sends finitely generated projective modules to 
finitely generated projective
modules.
Finally $R_{\alpha+\beta} e_{\alpha,\beta}$ is a graded 
free right $R_{\alpha, \beta}$-module of finite 
rank, so
$\Ind_{\alpha,\beta}^{\alpha+\beta}$ sends finite dimensional 
 modules to 
finite dimensional
modules too.

We will often appeal without mention to the following general facts
about the representation theory of $R_\alpha$, all of which are
noted in \cite[$\S$2.5]{KL1}.

\begin{Lemma}
The algebra $R_\alpha$ has
finitely many isomorphism classes of irreducible
graded modules (up to degree shift), all of 
which are finite dimensional.
Every irreducible $L \in \Rep{R_\alpha}$ has a
unique (up to isomorphism) projective cover 
$P \in \Proj{R_\alpha}$  
with irreducible head isomorphic to $L$, and every indecomposable
projective graded $R_\alpha$-module 
is of this form.
\end{Lemma}

\begin{proof}
All irreducible modules are finite dimensional
because 
$R_\alpha$ is finitely generated over its center; see
\cite[Corollary 2.10]{KL1}.
Every irreducible graded $R_\alpha$-module is a quotient of a finite dimensional
module of the form
$\Ind_{\alpha_{i_1},\dots,\alpha_{i_d}}^\alpha L$
for some $\bi \in I^\alpha$ and
some irreducible graded
$R_{\alpha_{i_1},\dots,\alpha_{i_d}}$-module $L$.
Since $I^\alpha$ is finite and there is only one irreducible graded
$R_{\alpha_{i_1},\dots,\alpha_{i_d}}$-module
up to degree shift, we deduce that 
there are only finitely many irreducible graded $R_\alpha$-modules
up to degree shift.
These statements imply in particular that the graded
Jacobson radical $J(R_\alpha)$ of
$R_\alpha$ is of finite codimension. 
Noting also that $J(R_\alpha)_0
\subseteq J((R_\alpha)_0)$ and that $(R_\alpha)_0$ is finite dimensional, the 
rest of the lemma follows by arguments involving 
lifting homogeneous idempotents from the 
finite dimensional semisimple graded algebra $R_\alpha / J(R_\alpha)$.
\end{proof}

\subsection{\boldmath The functors $\theta_i$ and $\theta_i^*$}
For $i \in I$, let $P(i)$ denote the regular
representation of $R_{\alpha_i}$.
Define a functor
\begin{equation}\label{fri1}
\theta_i := \Ind_{\alpha,\alpha_i}^{\alpha+\alpha_i} (? \boxtimes P(i)):
\Mod{R_\alpha} \rightarrow \Mod{R_{\alpha+\alpha_i}}.
\end{equation}
This functor is exact, and it maps finitely generated
projective modules to finitely generated projective modules, 
so restricts to a functor
$\theta_i: \Proj{R_\alpha}
\rightarrow \Proj{R_{\alpha+\alpha_i}}$.

The functor $\theta_i$ possesses a right adjoint 
\begin{equation}\label{fri2}
\theta_i^*:= \HOM_{R_{\alpha_i}'}(P(i), ?):\Mod{R_{\alpha+\alpha_i}}\rightarrow \Mod{R_\alpha},
\end{equation}
where $R_{\alpha_i}'$ denotes the subalgebra $1 \otimes R_{\alpha_i}$
of $R_{\alpha,\alpha_i}$.
Equivalently, $\theta_i^*$ is defined by
multiplication by the idempotent $e_{\alpha,\alpha_i}$
followed by restriction to the subalgebra $R_\alpha = R_{\alpha} \otimes 1$
of $R_{\alpha,\alpha_i}$.
The functor $\theta_i^*$ is exact,
and it restricts 
to define a functor $\theta^*_i: \Rep{R_{\alpha+\alpha_i}}
\rightarrow \Rep{R_{\alpha}}.$

\subsection{Dualities}\label{ssd}
The algebra $R_\alpha$ possesses a graded anti-automorphism
\begin{equation}\label{star}
*:R_\alpha \rightarrow R_\alpha
\end{equation}
which is the identity on generators. 

Using this we introduce a duality denoted $\circledast$ on $\Rep{R_\alpha}$,
mapping a module $M$ to $M^\circledast := \HOM_F(M, F)$ with the
action defined by $(xf)(m) = f(mx^*)$.
This duality commutes with 
$\theta_i^*$,
i.e. there is an isomorphism of functors
\begin{equation}\label{commd}
\circledast \circ \theta_i^* \cong \theta_i^* \circ 
\circledast:
\Rep{R_{\alpha+\alpha_i}} \rightarrow \Rep{R_\alpha}.
\end{equation}

There is another duality denoted $\#$
on $\Proj{R_\alpha}$ mapping a projective module $P$
to $P^\# := \HOM_{R_\alpha}(P, R_\alpha)$ with the action defined
by $(xf)(p) = f(p) x^*$.
This commutes with
the functor 
$\theta_i$, i.e.
\begin{equation}\label{snow}
\# \circ \theta_i \cong \theta_i \circ \#:
\Proj{R_\alpha} \rightarrow \Proj{R_{\alpha+\alpha_i}}.
\end{equation}
Recalling (\ref{cartanpairing}),
the following lemma makes a connection between 
the dualities $\circledast$ and $\#$.

\begin{Lemma}
For $P \in \Proj{R_\al}$ and $M \in \Rep{R_\al}$,
we have that
$$
\langle [P^\#],[M] 
\rangle = \overline{\langle [P], [M^\circledast] \rangle}.
$$
\end{Lemma}

\begin{proof}
Let $P^{\op}$ denote $P$ viewed as  a right $R_\alpha$-module
with action $px := x^* p$ for $p\in P, x \in R_\alpha$.
We have that
\begin{align*}
\langle [P^\#],[M]\rangle
&=
\qdim\,\HOM_{R_\alpha}(P^\#, M)=
\qdim\,\HOM_{R_\alpha}(P^\#, R_\alpha) \otimes_{R_\alpha} M\\
&=
\qdim\,P^{\op} \otimes_{R_\alpha} M=
\overline{\qdim\,\HOM_F(P^{\op} \otimes_{R_\alpha} M,F)}\\&=
\overline{\qdim\,\HOM_{R_\alpha}(P^{\op},\HOM_F(M,F))}=
\overline{\qdim\,\HOM_{R_\alpha}(P,M^\circledast)}\\
&=\overline{\langle [P],[M^\circledast] \rangle},
\end{align*}
using the observation that $(P^\#)^\# \cong P$.
\end{proof}

\begin{Corollary}\label{dualpims}
Let $L \in \Rep{R_\alpha}$ be an irreducible module
with projective cover $P \in \Proj{R_\alpha}$.
Then the projective cover of $L^\circledast$ is isomorphic to
$P^\#$.
In particular, if $L \cong L^\circledast$ then $P \cong P^\#$.
\end{Corollary}

\subsection{Divided powers}
As explained in detail in \cite[$\S$2.2]{KL1}, 
in the case $\alpha = n \alpha_i$ for some $i \in I$,
the algebra
$R_\alpha$ is isomorphic to the nil-Hecke algebra $NH_n$. It has
a canonical representation on the polynomial algebra $F[y_1,\dots,y_n]$
such that each $y_r$ acts as multiplication by $y_r$ and each $\psi_r$ acts
as the {\em divided difference operator}
$$
\partial_r: f \mapsto \frac{{^{s_r}} f - f}{y_{r}-y_{r+1}}.
$$
Let $P(i^{(n)})$ denote the polynomial representation of $R_{n\alpha_i}$
viewed as a 
graded $R_{n\alpha_i}$-module with grading defined by
$$
\deg(y_1^{m_1} \cdots y_n^{m_n}) := 2m_1+\cdots+2m_n - \frac{1}{2}n(n-1).
$$
Denoting the left 
regular $R_{n\alpha_i}$-module
by $P(i^n)$, it is noted in \cite[$\S$2.2]{KL1} that
\begin{equation}\label{div}
P(i^n) \cong [n]! \cdot P(i^{(n)}).
\end{equation}
In particular, $P(i^{(n)})$ is projective.

Now we can 
generalize the definition of the functors $\theta_i$ and $\theta_i^*$:
for $i \in I$ and $n \geq 1$, set
\begin{align}\label{div1}
\theta_{i}^{(n)}&:= 
\Ind_{\alpha,n \alpha_i}^{\alpha+n\alpha_i} (? \boxtimes P(i^{(n)})):\Mod{R_\alpha} \rightarrow \Mod{R_{\alpha+n\alpha_i}},\\
(\theta_{i}^*)^{(n)}&:= 
\HOM_{R'_{n \alpha_i}}(P(i^{(n)}), ?): \Mod{R_{\alpha+n\alpha_i}} \rightarrow \Mod{R_{\alpha}},\label{div2}
\end{align}
where $R'_{n\alpha_i} := 1 \otimes R_{n\alpha_i} \subseteq R_{\alpha,n\alpha_i}$.
Both functors are exact, 
$\theta_i^{(n)}$ sends finitely generated projective modules to 
finitely generated projective modules,
and $(\theta_i^*)^{(n)}$ sends finite dimensional modules to finite dimensional modules.
By 
transitivity of induction and restriction,
there are isomorphims
$$
\theta_i^n
\cong \Ind_{\alpha,n\alpha_i}^{\alpha+n\alpha_i}(? \boxtimes P(i^n)),
\qquad(\theta_i^*)^n
\cong \HOM_{R_{n\alpha_i}'}(P(i^n), ?).
$$
Hence (\ref{div}) implies that
the $n$th powers
of $\theta_i$ and $\theta_i^*$
decompose as
\begin{equation}\label{dividedform}
\theta_i^n \cong [n]! \cdot \theta_i^{(n)},
\qquad
(\theta_i^*)^n \cong [n]! \cdot (\theta_i^*)^{(n)}.
\end{equation}

\subsection{The Khovanov-Lauda theorem}\label{sklth}
It is convenient to 
abbreviate the direct sums of all our Grothendieck groups by
\begin{equation}\label{like}
[\Proj{R}] := \bigoplus_{\alpha \in Q_+} [\Proj{R_\alpha}],
\qquad
[\Rep{R}] := \bigoplus_{\alpha \in Q_+} [\Rep{R_\alpha}].
\end{equation}
Also, for $\alpha,\beta \in Q_+$,
we identify the Grothendieck group $[\Proj{R_{\alpha,\beta}}]$ with
$[\Proj{R_\alpha}] \otimes_{\Laurent} [\Proj{R_\beta}]$
so that $[P \boxtimes Q]$ is identified with $[P] \otimes [Q]$.
Finally, we observe that 
the exact functors $\theta_i^{(n)}$ (resp.\ $(\theta_i^*)^{(n)}$)
induce $\Laurent$-linear 
endomorphisms of $[\Proj{R}]$ (resp.\ $[\Rep{R}]$) which we denote
by the same notation.
Now we can state the following important theorem 
proved in \cite[Section 3]{KL1}.

\begin{Theorem}[Khovanov-Lauda]\label{klthm}
There is a unique $\Laurent$-module isomorphism
$$
\gamma:{_\Laurent}\mathbf f \stackrel{\sim}{\rightarrow} [\Proj{R}]
$$
such that 
$1 \mapsto [R_0]$ (the class of the left
regular representation of
the trivial algebra $R_0$) 
and $\gamma(x\,\theta_i^{(n)}) = 
\theta_i^{(n)}(\gamma(x))$ for each $x \in {_\Laurent}\mathbf f,
i \in I$ and $n \geq 1$.
Under this isomorphism:
\begin{itemize}
\item[(1)] the 
multiplication $m_{\alpha,\beta}:{_\Laurent} \mathbf f_\alpha \otimes_{\Laurent} {_\Laurent}\mathbf f_\beta
\rightarrow {_\Laurent} \mathbf f_{\alpha+\beta}$ 
corresponds to the induction product induced by the exact
induction functor 
$\Ind_{\alpha,\beta}^{\alpha+\beta}$;
\item[(2)] the comultiplication $m_{\alpha,\beta}^*:{_\Laurent} \mathbf f_{\alpha+\beta} 
\rightarrow {_\Laurent} \mathbf f_\alpha \otimes_{\Laurent}
{_\Laurent} \mathbf f_\beta$
corresponds to the restriction coproduct induced by the 
exact restriction functor
$\Res_{\alpha,\beta}^{\alpha+\beta}$;
\item[(3)] the bar-involution on ${_\Laurent} \mathbf f_{\alpha}$ 
corresponds to the
anti-linear involution induced by the duality $\#$.
\end{itemize}
\end{Theorem}

\begin{Remark}
Theorem~\ref{klthm} establishes in particular that
the functors $\theta_i$ induce
$\Laurent$-linear operators on
$[\Proj{R}]$ that satisfy the quantum Serre relations 
from (\ref{qserre}).
For a more general categorical version of this statement,
see \cite[Proposition 6]{KL2} or \cite[Proposition 4.2]{Ro}.
\end{Remark}

\subsection{\boldmath $q$-Characters}
The dual statement to Theorem~\ref{klthm},
in which $\ga$ gets replaced by 
its graded dual $\ga^*:[\Rep{R}] \hookrightarrow {_{\Laurent}}\mathbf f^*$,
has a natural representation theoretic extension
related to the notion of {\em $q$-character}.
This goes back at a purely combinatorial level 
to work of Leclerc in \cite{Lec}. We only need one basic fact from this
side of the picture.
Given $\alpha \in Q_+$,
let $'\mathbf f_\alpha^*$ denote the 
$\Q(q)$-vector on basis $\{\bi \:|\:\bi \in I^\alpha\}$
and set
$'\mathbf f^* := \bigoplus_{\alpha \in Q_+}
{'\mathbf f_\alpha^*}$.
Given $M \in \Rep{R_\alpha}$,
its {\em $q$-character} means the formal expression
\begin{equation}\label{qch}
\CH M := 
\sum_{\bi \in I^\alpha} (\qdim\ e(\bi) M) \cdot \bi
\in
 {'\mathbf f_\alpha^*}.
\end{equation}
The following theorem is established in \cite{KL1} 
in order 
to prove the surjectivity of the map $\ga$ in Theorem~\ref{klthm}.

\begin{Theorem}[{\cite[Theorem 3.17]{KL1}}]\label{ch}
The map $$
\CH:[\Rep{R}] \rightarrow
 {'\mathbf f^*}, \quad [M] \mapsto \CH M
$$ 
is injective.
\end{Theorem}

\section{Higher level Fock spaces}

Continuing with notation from the previous section, we fix also now
a tuple
$(k_1,\dots, k_l) \in I^l$ for some $l \geq 0$
and set
\begin{equation}\label{note}
\La := \La_{k_1}+\cdots+\La_{k_l}.
\end{equation} 
For each $m=1,\dots,l$, we pick $\kk_m \in \Z$ such that
\begin{equation}\label{multicharge}
\kk_m \equiv k_m \pmod{e}.
\end{equation}
We refer to $l$ here as the {\em level}
and the tuple $(\kk_1,\dots,\kk_l)$ as 
the {\em multicharge}.
Almost everything depends implicitly not just
on $\La$ but on the ordered tuple $(k_1,\dots,k_l)$;
the choice of multicharge plays a significant role only in
$\S\S$\ref{sug}--\ref{construct}.

\subsection{The quantum group}
Let $\g$ be the Kac-Moody algebra corresponding to the Cartan matrix
(\ref{ECM}), so
$\g = \widehat{\mathfrak{sl}}_e(\C)$ if $e > 0$ 
and $\g = 
\mathfrak{sl}_\infty(\C)$ if $e = 0$.
Let $U_q(\g)$ be the quantized enveloping algebra of $\g$. So 
$U_q(\g)$
is the $\Q(q)$-algebra generated by the Chevalley  
generators $E_i,F_i,K_i^{\pm 1}$ for $i\in I$, subject 
only to the usual quantum Serre relations (for all admissible $i,j\in I$):
\begin{align}
K_iK_j&=K_jK_i,\quad \hspace{8.5mm}K_iK_i^{-1}=1,\\
K_iE_jK_i^{-1}&= q^{a_{i,j}}E_j,\quad \hspace{3mm}K_iF_jK_i^{-1}= q^{-a_{i,j}}F_j,\label{wight}\\
[E_i,F_j]&=\de_{i,j}\frac{K_i-K_i^{-1}}{q-q^{-1}},\label{mix}\\
(\ad_q E_i)^{1-a_{j,i}}(E_j)&=0\qquad\qquad\qquad(i\neq j),\label{serree}\\
\quad (\ad_q F_i)^{1-a_{j,i}}(F_j)&=0\qquad\qquad\qquad(i\neq j),\label{serref}
\end{align}
recalling (\ref{adq}).
We consider $U_q(\g)$ as a Hopf algebra with respect to the coproduct given for all $i\in I$ as follows: 
$$
\Delta:\ K_i\mapsto K_i\otimes K_i,\ E_i\mapsto E_i\otimes K_i+1\otimes E_i,\ F_i\mapsto F_i\otimes 1+K_i^{-1}\otimes F_i
$$

The {\em bar-involution} $-:U_q(\g) \rightarrow U_q(\g)$
is the anti-linear involution such that
$$
\overline{K_i} = K_i^{-1},\qquad
\overline{E_i} = E_i, \qquad \overline{F_i} = F_i.
$$
Given a $U_q(\g)$-module $V$, a
{\em compatible bar-involution} on $V$ means
an anti-linear involution
$-:V \rightarrow V$ such that $\overline{xv} = \overline{x}\, \overline{v}$
for all $x \in U_q(\g)$ and $v \in V$.

Also let $\tau:U_q(\g) \rightarrow U_q(\g)$ be the 
anti-linear anti-automorphism defined by
\begin{align}\label{taudef}
\tau&:\
 K_i \mapsto K_i^{-1}, \qquad E_i \mapsto q F_i K_i^{-1},
\qquad F_i \mapsto q^{-1} K_i E_i.\:\:\:\\\intertext{Note $\tau$ is not an involution: its inverse $\tau^{-1}$ is given
by the formulae}
\label{tauniv}
\tau^{-1}&:\ K_i \mapsto K_i^{-1}, \qquad E_i \mapsto q^{-1} F_i K_i,
\qquad F_i \mapsto q K_i^{-1} E_i.
\end{align}
In other words, $\tau^{-1} \circ - = - \circ \tau$.

Everything so far makes sense over the ground ring $\Laurent$ too.
Let $U_q(\g)_\Laurent$ denote Lusztig's $\Laurent$-form for $U_q(\g)$, which is
the $\Laurent$-subalgebra generated by the quantum divided
powers $E_i^{(n)} := E_i^n / [n]!$ and $F_i^{(n)} := F_i^n / [n]!$
for all $i \in I$ and $n \geq 1$.
The bar-involution, the comultiplication $\Delta$ and the anti-automorphism
$\tau$ descend in the obvious way to this $\Laurent$-form. 

\subsection{Recollections about upper crystal bases}
Let $\A := \Q[q,q^{-1}] \cong \Q \otimes_\Z \Laurent$.
Let $\A_0$ (resp.\ $\A_\infty$) denote the subalgebra of $\Q(q)$ consisting of all 
rational functions which are regular at zero (resp.\ at infinity).
According to \cite[$\S$2.1]{KaG}, a {\em balanced triple} 
in a $\Q(q)$-vector space $V$ is a triple $(V_\A, 
V_0, V_\infty)$ where $V_\A$ is an $\A$-submodule of $V$,
$V_0$ is an $\A_0$-submodule of $V$ and $V_\infty$ is an $\A_\infty$-submodule
of $M$, such that the following two properties hold:
\begin{itemize}
\item[(1)]
The natural multiplication maps
$\Q(q)\otimes_{\A} V_\A \rightarrow V,
\Q(q)\otimes_{\A_0} V_0 \rightarrow V$
and 
$\Q(q)\otimes_{\A_\infty} V_\infty \rightarrow V$
are all isomorphisms.
\item[(2)] Setting $E := V_\A \cap V_0 \cap V_\infty$, one of the following
three equivalent conditions holds:
\begin{itemize}
\item[(a)] the natural map $E \rightarrow V_0 / q V_0$ is an isomorphism;
\item[(b)] the natural map $E \rightarrow V_\infty / q^{-1} V_\infty$ is an isomorphism;
\item[(c)] the natural maps $\A \otimes_\Q E \rightarrow V_\A$,
$\A_0 \otimes_\Q E \rightarrow V_0$ and $\A_\infty \otimes_\Q E \rightarrow V_\infty$ are all isomorphisms.
\end{itemize}
\end{itemize}
These isomorphisms provide a canonical way 
to lift any ``local'' basis
for $V_0 / q V_0$ (or for $V_\infty / q^{-1} V_\infty$) 
to a ``global'' basis for $V$.

For an integrable $U_q(\g)$-module $V$, let
$\tilde e_i$ and $\tilde f_i$ be Kashiwara's upper crystal operators
on $V$ from \cite{Ka1}; see also \cite[(3.1.2)]{KaG}.
Recall an {\em upper crystal lattice at $q=0$}
is a free $\A_0$-submodule $V_0$ of $V$ such that
\begin{itemize}
\item[(1)]
$V \cong \Q(q)\otimes_{\A_0} V_0$;
\item[(2)]
$V_0$ 
is the direct sum of its weight spaces;
\item[(3)] $V_0$ is invariant under the actions of $\tilde e_i, \tilde f_i$.
\end{itemize}
The notion $V_\infty$ of an {\em upper crystal lattice at $q=\infty$} 
is defined similarly, replacing $\A_0$ with $\A_\infty$.

An {\em upper crystal basis at $q=0$}
is a pair $(V_0, B_0)$ where $V_0$ is an upper crystal lattice at
$q=0$ and $B_0$ is a basis of the $\Q$-vector space $V_0 / q V_0$ such that
\begin{itemize}
\item[(1)] 
each element of $B_0$ is a weight vector, i.e. it 
is the image of a weight vector in $V_0$ under
the natural map $V_0 \rightarrow V_0 / q V_0$;
\item[(2)] writing also $\tilde e_i, \tilde f_i$
for the $\Q$-linear endomorphisms of $V_0 / qV_0$ induced by 
$\tilde e_i, \tilde f_i$, we have that
$\tilde e_i B_0 \subseteq B_0 \cup \{0\}$
and
$\tilde f_i B_0 \subseteq B_0 \cup \{0\}$;
\item[(3)]
for $b, b' \in B_0$, $b' = \tilde f_i b$ if and only if
$\tilde e_i b' = b$.
\end{itemize}

If $(V_0, B_0)$ is an upper crystal basis at $q=0$, 
there is an induced structure of an
{\em abstract crystal} on the set $B_0$ in the sense of
\cite{Kas}, that is, 
there is a canonically associated crystal datum
$(B_0, \tilde e_i, \tilde f_i, \eps_i, \phi_i, \wt)$.
Here, for $b \in B_0$, $\wt(b)$ denotes the {\em weight} of $b$,
and
$$
\eps_i(b) := \max\{k \geq 0\:|\:\tilde e_i^k (b) \neq 0\},\quad
\phi_i(b) := \max\{k \geq 0\:|\:\tilde f_i^k (b) \neq 0\},
$$
for each $i \in I$.
It is automatically the case that
\begin{equation}\label{us0}
(\wt(b), \alpha_i) = \phi_i(b) - \eps_i(b).
\end{equation}

By \cite[Proposition 6]{Ka1},
upper crystal bases at $q=0$ behave well under tensor product.
More precisely, if $(V_0, B_0)$ and $(V_0', B_0')$ are
upper crystal bases at $q=0$ in integrable modules $V$ and $V'$,
then $(V_0 \otimes V_0', B_0 \times B_0')$
is an upper crystal basis at $q=0$ in $V \otimes V'$.
Moreover, there is an explicit combinatorial rule describing the
crystal operators $\tilde e_i$ and $\tilde f_i$ on
$B_0 \times B_0'$ in terms of the ones on $B_0$ and $B_0'$.
We refer the reader to \cite[Proposition 6]{Ka1} for the 
precise statement of this; note though that we are 
using the opposite comultiplication to the one used in \cite[(1.5)]{Ka1}
so the order of tensors needs to be flipped when translating this
tensor product rule into our setup.


\subsection{\boldmath The module $V(\La)$}\label{md}
Let $V(\La)$ denote the integrable highest weight
module for $U_q(\g)$ of highest weight $\La$, where $\La$
is the dominant integral weight fixed in (\ref{note}).
Fix also a choice of a non-zero highest weight vector $v_\La \in V(\La)$.
The module $V(\La)$ possesses a unique
compatible bar-involution
$-:V(\La) \rightarrow V(\La)$
such that $\overline{v_\La} = v_\La$.

The {\em contravariant form} $(.,.)$ on $V(\La)$ is the
unique symmetric bilinear form such that 
\begin{itemize}
\item[(1)]
$E_i$ and $F_i$ are biadjoint, i.e.
$(E_i v, w) = (v, F_i w)$ and $(F_i v, w) = (v, E_i w)$
for all $v, w \in V(\La)$ and $i \in I$;
\item[(2)] 
$(v_\La, v_\La) = 1$.
\end{itemize}
Actually, 
we usually prefer to work with a slightly different form
on $V(\La)$, namely, the {\em Shapovalov form}
$\langle.,.\rangle$.
By definition, this is the unique 
sesquilinear form  (anti-linear in the first argument, linear in the second)
on $V(\La)$
such that 
\begin{itemize}
\item[(1)] $\langle u v, w\rangle = \langle v, \tau(u) w\rangle$ for
all $u \in U_q(\g)$ and $v, w \in V(\La)$;
\item[(2)] $\langle v_\La, v_\La\rangle = 1$.
\end{itemize}
Define the {\em defect}
of $\al \in Q_+$ (relative to $\La$) by setting
\begin{equation}\label{defdef}
\defect(\al):=((\La,\La)-(\La-\al,\La-\al))/2 = (\La,\al)-(\al,\al)/2.
\end{equation}
The contravariant and Shapovalov forms are closely related:

\begin{Lemma}\label{contra2}
For vectors $v, w \in V(\La)$ with $v$ of weight $\La-\alpha$,
we have that
\begin{itemize}
\item[(1)]
$\langle v,w \rangle = q^{\defect(\alpha)}(\overline{v},w) = \langle \overline{w}, \overline{v}\rangle$;
\item[(2)]
$(v,w) = \overline{(\overline{w}, \overline{v})}$;
\item[(3)]
$\langle v,w \rangle =
q^{2\defect(\alpha))} \overline{\langle w,v \rangle}$.
\end{itemize}
\end{Lemma}

\begin{proof}
Mimic the proof of \cite[Lemma 2.6]{BKariki}.
\end{proof}

Let $V(\La)_\Laurent$ denote the {\em standard $\Laurent$-form} 
for $V(\La)$, that is,
the $U_q(\g)_\Laurent$-submodule of $V(\La)$ generated by the 
highest weight vector $v_\La$. This is free as an $\Laurent$-module.
Let $V(\La)_\Laurent^*$ denote the {\em costandard $\Laurent$-form} for $V(\La)$,
that is, the dual lattice
\begin{align*}
V(\La)_\Laurent^* &= \{v \in V(\La)\:|\:( v,w )
\in \Laurent\text{ for all }
w\in V(\La)_\Laurent\}\\
&= \{v \in V(\La)\:|\:\langle v,w \rangle
\in \Laurent\text{ for all }
w\in V(\La)_\Laurent\}.
\end{align*}
As explained in \cite[$\S$3.3]{KaG},
the results of \cite{Ka2} imply that $V(\La)$
has a unique upper crystal basis 
$(V(\La)_0, B(\La)_0)$ at $q=0$
such that
\begin{itemize}
\item[(1)]
the $\La$-weight space of $V(\La)_0$ is equal to $\A_0 v_\La$;
\item[(2)]
$v_\La + q V(\La)_0 \in B(\La)_0$.
\end{itemize}
We will describe an explicit combinatorial realization 
of the underlying abstract crystal
in $\S$\ref{scc} below.

According to \cite[Lemma 4.2.1]{KaG}, 
$(V(\La)^*_\A, V(\La)_0, \overline{V(\La)_0})$ is a balanced triple,
where $V(\La)_\A^* := \Q \otimes_{\Z} V(\La)_\Laurent^*$.
Hence there is a canonical lift of the upper crystal basis
$B(\La)_0$ to a basis of $V(\La)$. This is Kashiwara's
{\em upper global crystal basis}
of $V(\La)$, which is Lusztig's {\em dual-canonical basis}.
The dual basis to the upper global crystal basis
under the contravariant form $(.,.)$ is the
{\em lower global crystal basis}.
This is Lusztig's {\em canonical basis}
as in \cite[$\S$14.4]{Lubook}.

Lusztig's approach gives moreover that
the canonical basis is a basis for 
$V(\La)_\Laurent$ as a free $\Laurent$-module,
and that each vector in the canonical basis is bar-invariant.
Hence the dual-canonical basis is a basis
for
$V(\La)_\Laurent^*$ as a free $\Laurent$-module,
and by Lemma~\ref{contra2}(2)
each vector in the dual-canonical basis is bar-invariant too.
(These statements can also be deduced without 
invoking Lusztig's geometric construction using 
the Fock space approach explained below.)

\subsection{Combinatorics of multipartitions}\label{cmps}
A {\em partition} is a non-increasing sequence $\la = (\la_1,\la_2,\dots)$
of non-negative integers; we set
$|\la| := \la_1+\la_2+\cdots$.
An {\em $l$-multipartition}
is an ordered $l$-tuple of partitions 
$\la = (\la^{(1)} , \dots,\la^{(l)})$;
we set $|\la| := |\la^{(1)}|+\cdots+|\la^{(l)}|$.
We let $\mathscr P$ (resp.\ $\Par$) 
denote the set of all partitions (resp. $l$-multipartitions).

The {\em Young diagram} of the multipartition $\la = (\la^{(1)}, 
\dots, \la^{(l)})\in \Par$ is 
$$
\{(a,b,m)\in\Z_{>0}\times\Z_{>0}\times \{1,\dots,l\}\mid 1\leq b\leq 
\la_a^{(m)}\}.
$$
The elements of this set
are called the {\em nodes of $\la$}. More generally, a {\em node} is an 
element of $\Z_{>0}\times\Z_{>0}\times \{1,\dots,l\}$. 
Usually, we identify the multipartition $\la$ with its 
Young diagram and visualize it as a column vector of Young diagrams. 
For example, $((3,1), (4,2))$ is the Young diagram
\begin{equation}\label{yde}
\begin{array}{l}
\diagram{&&\cr \cr}\\\diagram{&&&\cr &\cr}
\end{array}
\end{equation}
A node $A\in\la$ is called {\em removable (for $\la$)}\, if $\la\setminus \{A\}$ has a shape of a multipartition. A node $B\not\in\la$ is called {\em addable (for $\la$)}\, if $\la\cup \{B\}$ has a shape of a multipartition. We use the notation
$$
\la_A:=\la\setminus \{A\},\qquad \la^B:=\la\cup\{B\}.
$$

The {\em residue} $\res\,A$ of 
the node $A = (a,b,m)$ is the integer
\begin{equation}\label{resdef}
\res\, A := \kk_m+b-a \in \Z.
\end{equation}
Although 
this depends implicitly on the fixed
choice of multicharge from (\ref{multicharge}), we will normally
only be interested in residues modulo $e$:
for $i \in I$ we say that $A$ is an {\em $i$-node} 
if $\res\,A \equiv i \pmod{e}$.
Given $\la \in \Par$, we define its
{\em content} by
\begin{align}\label{contdef}
\cont(\la) &:= \sum_{i\in I} n_i
\alpha_i \in Q_+,
\end{align}
where $n_i$ is the number of
$i$-nodes $A\in\la$.
For $\al \in Q_+$, let
\begin{equation}\label{padef}
\Par_\alpha := \{\la \in \Par\:|\:\cont(\la) = \alpha\}
\end{equation}
denote the set of all $l$-multipartitions of content $\alpha$.
Note that every $\la \in \Par_\alpha$
has $|\la| = \height(\alpha)$.

\subsection{Some partial orders}\label{tpo}
We now define two partial orders $\unlhd$ and $\leq$
on $\Par$.
The first of these is the {\em dominance ordering}
which is defined
by $\mu \unlhd \la$ if
\begin{equation}
\sum_{a=1}^{m-1}|\mu^{(a)}|+\sum_{b=1}^c\mu_b^{(m)}\leq 
\sum_{a=1}^{m-1}|\la^{(a)}|+\sum_{b=1}^c\la_b^{(m)}
\end{equation}
for all $1\leq m\leq l$ and $c\geq 1$, with equality for
$m=l$ and $c \gg 1$.
In other words, $\mu$ is obtained from $\la$ by moving nodes down in
the diagram.
In case $l=1$ this is just the usual
notion of the dominance ordering on partitions.

For the second ordering we treat the cases
$e=0$ and $e > 0$ separately.
Assume first that $e=0$.
Let $\leq$ be the dominance ordering on 
$Q_+ = \bigoplus_{i \in \Z} \Z_{\geq 0} \alpha_i$, i.e.
$\alpha \leq \beta$ if $\beta - \alpha \in Q_+$.
For $\la \in \Par$,
let $\cont_m(\la) \in Q_+$ denote the content of the
$m$th component $\la^{(m)}$ of $\la$ defined in the analogous way to 
(\ref{contdef}); in particular, $\cont(\la) = \cont_1(\la)+\cdots+\cont_l(\la)$.
Then declare that $\mu \leq \la$ if
\begin{equation}\label{zd}
\cont_1(\mu)+\cdots+\cont_m(\mu) \leq
\cont_1(\la)+\cdots+\cont_m(\la)
\end{equation}
for each $m=1,\dots,l$, with equality in the case $m=l$.

\begin{Lemma}\label{zerot}
Suppose that $e=0$ and $\la,\mu \in \Par$.
Then $\mu \leq \la$ implies $\mu \unlhd \la$.
\end{Lemma}

\begin{proof}
Consider the set $Q_+^l$ of $l$-tuples
$\alpha =(\alpha^{(1)},\dots,\alpha^{(l)})$ of elements
of $Q_+$.
The assumption that $e=0$ implies that the map
$$
\rho:\Par \rightarrow Q_+^l,\qquad 
\la \mapsto \left(\cont_1(\la), \cdots, \cont_l(\la)\right)
$$
is injective.
For any $\alpha \in Q_+^l$, $1 \leq m \leq l$ and $a \geq 1$, let
$$
r^{(m)}_a(\alpha) := \#\big\{
i \in \Z\:\big|\:1 \leq a+\min(i-k_m,0) \leq (\La_i,\alpha^{(m)})\big\}.
$$
The key point about this definition is that
$r^{(m)}_a(\rho(\la)) = \la^{(m)}_a$
for $\la \in \Par$.
Define a pre-order $\unlhd$ on $Q_+^l$ by
declaring that $\beta \unlhd \alpha$ if
$$
\sum_{a=1}^{m-1} \height(\beta^{(a)}) + \sum_{b=1}^c r_b^{(m)}(\beta)
\leq 
\sum_{a=1}^{m-1} \height(\alpha^{(a)}) + \sum_{b=1}^c r_b^{(m)}(\alpha)
$$
for all $1 \leq m \leq l$ and $c \geq 1$, with equality for $m=l$
and $c \gg 1$.
Then for $\la,\mu \in \Par$ we have that $\mu \unlhd \la$ if and only if $\rho(\mu) \unlhd 
\rho(\la)$. 
Also 
let $\preceq$ be the partial order on $Q_+^l$ defined
by $\beta \preceq \alpha$ if 
$\beta^{(1)}+\cdots+\beta^{(m)} \leq \alpha^{(1)}+\cdots+\alpha^{(m)}$
for each $m=1,\dots,l$ with equality in the case $m=l$.
We obviously have that
$\mu \leq \la$ if and only if $\rho(\mu) \preceq \rho(\la)$.
With these definitions, we are reduced to showing 
for $\alpha,\beta \in Q_+^l$
that $\beta \preceq \alpha$ implies $\beta \unlhd \alpha$.

So now take $\alpha,\beta \in Q_+^l$ with $\beta \preceq \alpha$.
If $\beta^{(1)} = \alpha^{(1)}$ then we get that $\beta \unlhd \alpha$
by induction on $l$. So we may assume that $\beta^{(1)} < \alpha^{(1)}$.
Choose $i \in\Z$ so that $\beta^{(1)}+\alpha_i \leq \alpha^{(1)}$.
Let $m>1$ be minimal so that $\alpha_i \leq \beta^{(m)}$.
Then define $\ga \in Q_+^l$ by 
setting $\ga^{(1)} := \beta^{(1)}+\alpha_i,
\ga^{(m)} := \beta^{(m)} - \alpha_i$ and $\ga^{(a)} := \beta^{(a)}$
for all $a \neq 1,m$.
It follows that $\beta \prec \ga \preceq \alpha$.
By induction we get that $\ga \unlhd \alpha$.
Since $\unlhd$ is transitive it remains to observe that
$\beta \unlhd \ga$, which is a consequence of the definitions.
\end{proof}

Assume instead that $e > 0$. 
Before we can define the partial order $\leq$ in this case,
we need to introduce an injective map
\begin{equation}\label{uginj}
\Par \hookrightarrow \mathscr P,\qquad\la\mapsto\hla,
\end{equation}
which
is the inverse of the map sending a partition to its {\em $(l, e)$-quotient}
as introduced by Uglov in the first two paragraphs of \cite[p.273]{Uglov}.
Explicitly,
given $\la \in \Par$, the partition $\hla$ 
is defined as follows.
Consider an abacus display with rows (horizontal runners)
indexed by
$1,\dots,l$ from top to bottom, and columns (bead positions)
indexed by $\Z$ from left to 
right.
We represent $\la$ on this abacus by means of the
abacus diagram $A(\la)$ obtained putting a bead
in the $(\kk_m + \la_a^{(m)} - a + 1)$th column of the $m$th row
for each $m=1,\dots,l$ and $a \geq 1$.
Now reindex the positions on the abacus by the index set $\Z$,
so that the $m$th row and $n$th column
is indexed instead by the integer $els + e(m-1) +t$,
where $s$ and $t$ are defined by
writing $n = es + t$ for some 
$1 \leq t \leq e$.
Finally, working with this new indexing,
$\hla$ is the unique partition such that
the beads of $A(\la)$ 
are in exactly the positions indexed by the integers
$(\kk_1+\cdots+\kk_l + \hla_b -b+1)$ for all $b \geq 1$.

For example, suppose $e=2, l=3,
(\kk_1,\kk_2,\kk_3) = (11,7,2)$, 
$\mu = (\varnothing, (1^2), \varnothing)$
and $\la = (\varnothing, (1),(1))$.
Then the abacus diagrams representing $\mu$ and $\la$,
with bead positions indexed by $\Z$ in the manner just described,
are as follows:
$$
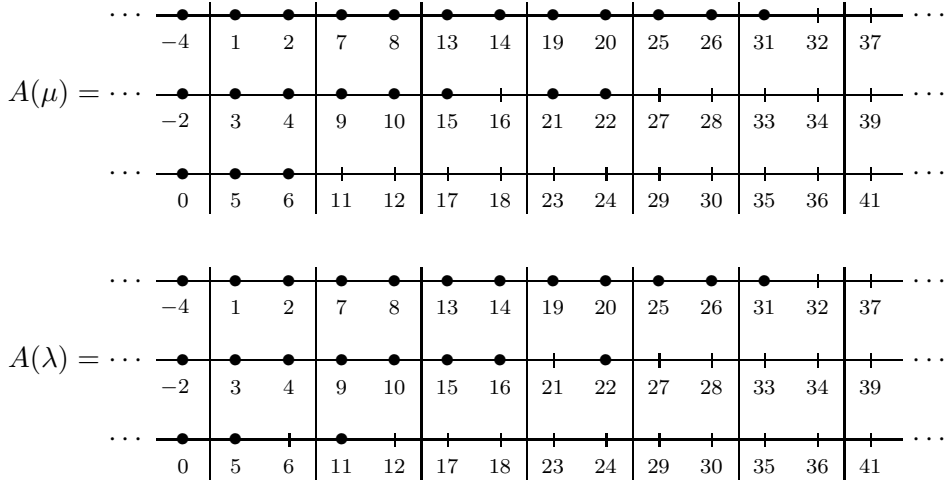
\begin{picture}(187,100)
\put(-10,95){\line(0,-1){80}}
\put(30,95){\line(0,-1){80}}
\put(70,95){\line(0,-1){80}}
\put(110,95){\line(0,-1){80}}
\put(150,95){\line(0,-1){80}}
\put(190,95){\line(0,-1){80}}
\put(230,95){\line(0,-1){80}}

\put(-30,90){\line(1,0){282}}
\put(-30,60){\line(1,0){282}}
\put(-30,30){\line(1,0){282}}
\put(-23,87.5){$\bullet$}
\put(-3,87.5){$\bullet$}
\put(17,87.5){$\bullet$}
\put(37,87.5){$\bullet$}
\put(57,87.5){$\bullet$}
\put(77,87.5){$\bullet$}
\put(97,87.5){$\bullet$}
\put(117,87.5){$\bullet$}
\put(137,87.5){$\bullet$}
\put(157,87.5){$\bullet$}
\put(177,87.5){$\bullet$}
\put(197,87.5){$\bullet$}
\put(220,87,5){\line(0,1){5}}
\put(240,87,5){\line(0,1){5}}

\put(-23,57.5){$\bullet$}
\put(-3,57.5){$\bullet$}
\put(17,57.5){$\bullet$}
\put(37,57.5){$\bullet$}
\put(57,57.5){$\bullet$}
\put(77,57.5){$\bullet$}
\put(117,57.5){$\bullet$}
\put(137,57.5){$\bullet$}
\put(100,57,5){\line(0,1){5}}
\put(160,57,5){\line(0,1){5}}
\put(180,57,5){\line(0,1){5}}
\put(200,57,5){\line(0,1){5}}
\put(220,57,5){\line(0,1){5}}
\put(240,57,5){\line(0,1){5}}

\put(-23,27.5){$\bullet$}
\put(-3,27.5){$\bullet$}
\put(17,27.5){$\bullet$}
\put(40,27,5){\line(0,1){5}}
\put(60,27,5){\line(0,1){5}}
\put(80,27,5){\line(0,1){5}}
\put(100,27,5){\line(0,1){5}}
\put(120,27,5){\line(0,1){5}}
\put(140,27,5){\line(0,1){5}}
\put(160,27,5){\line(0,1){5}}
\put(180,27,5){\line(0,1){5}}
\put(200,27,5){\line(0,1){5}}
\put(220,27,5){\line(0,1){5}}
\put(240,27,5){\line(0,1){5}}

\put(-69,60.2){\makebox(0,0){$A(\mu) = $}}
\put(0,80){\makebox(0,0){$_{1}$}}
\put(20,80){\makebox(0,0){$_{2}$}}
\put(40,80){\makebox(0,0){$_{7}$}}
\put(60,80){\makebox(0,0){$_{8}$}}
\put(80,80){\makebox(0,0){$_{13}$}}
\put(100,80){\makebox(0,0){$_{14}$}}
\put(120,80){\makebox(0,0){$_{19}$}}
\put(140,80){\makebox(0,0){$_{20}$}}
\put(160,80){\makebox(0,0){$_{25}$}}
\put(180,80){\makebox(0,0){$_{26}$}}
\put(200,80){\makebox(0,0){$_{31}$}}
\put(220,80){\makebox(0,0){$_{32}$}}
\put(240,80){\makebox(0,0){$_{37}$}}
\put(262,90){\makebox(0,0){${\cdots}$}}
\put(-42,90){\makebox(0,0){${\cdots}$}}
\put(262,60){\makebox(0,0){${\cdots}$}}
\put(-42,60){\makebox(0,0){${\cdots}$}}
\put(262,30){\makebox(0,0){${\cdots}$}}
\put(-42,30){\makebox(0,0){${\cdots}$}}

\put(0,50){\makebox(0,0){$_{3}$}}
\put(-23,50){\makebox(0,0){$_{-2}$}}
\put(-23,80){\makebox(0,0){$_{-4}$}}
\put(-20,20){\makebox(0,0){$_{0}$}}
\put(20,50){\makebox(0,0){$_{4}$}}
\put(40,50){\makebox(0,0){$_{9}$}}
\put(60,50){\makebox(0,0){$_{10}$}}
\put(80,50){\makebox(0,0){$_{15}$}}
\put(100,50){\makebox(0,0){$_{16}$}}
\put(120,50){\makebox(0,0){$_{21}$}}
\put(140,50){\makebox(0,0){$_{22}$}}
\put(160,50){\makebox(0,0){$_{27}$}}
\put(180,50){\makebox(0,0){$_{28}$}}
\put(200,50){\makebox(0,0){$_{33}$}}
\put(220,50){\makebox(0,0){$_{34}$}}
\put(240,50){\makebox(0,0){$_{39}$}}

\put(0,20){\makebox(0,0){$_{5}$}}
\put(20,20){\makebox(0,0){$_{6}$}}
\put(40,20){\makebox(0,0){$_{11}$}}
\put(60,20){\makebox(0,0){$_{12}$}}
\put(80,20){\makebox(0,0){$_{17}$}}
\put(100,20){\makebox(0,0){$_{18}$}}
\put(120,20){\makebox(0,0){$_{23}$}}
\put(140,20){\makebox(0,0){$_{24}$}}
\put(160,20){\makebox(0,0){$_{29}$}}
\put(180,20){\makebox(0,0){$_{30}$}}
\put(200,20){\makebox(0,0){$_{35}$}}
\put(220,20){\makebox(0,0){$_{36}$}}
\put(240,20){\makebox(0,0){$_{41}$}}
\end{picture}
$$
$$
\begin{picture}(187,85)
\put(-10,85){\line(0,-1){80}}
\put(30,85){\line(0,-1){80}}
\put(70,85){\line(0,-1){80}}
\put(110,85){\line(0,-1){80}}
\put(150,85){\line(0,-1){80}}
\put(190,85){\line(0,-1){80}}
\put(230,85){\line(0,-1){80}}

\put(-30,80){\line(1,0){282}}
\put(-30,50){\line(1,0){282}}
\put(-30,20){\line(1,0){282}}
\put(-23,77.5){$\bullet$}
\put(-3,77.5){$\bullet$}
\put(17,77.5){$\bullet$}
\put(37,77.5){$\bullet$}
\put(57,77.5){$\bullet$}
\put(77,77.5){$\bullet$}
\put(97,77.5){$\bullet$}
\put(117,77.5){$\bullet$}
\put(137,77.5){$\bullet$}
\put(157,77.5){$\bullet$}
\put(177,77.5){$\bullet$}
\put(197,77.5){$\bullet$}
\put(220,77,5){\line(0,1){5}}
\put(240,77,5){\line(0,1){5}}

\put(-23,47.5){$\bullet$}
\put(-3,47.5){$\bullet$}
\put(17,47.5){$\bullet$}
\put(37,47.5){$\bullet$}
\put(57,47.5){$\bullet$}
\put(77,47.5){$\bullet$}
\put(97,47.5){$\bullet$}
\put(137,47.5){$\bullet$}
\put(120,47,5){\line(0,1){5}}
\put(160,47,5){\line(0,1){5}}
\put(180,47,5){\line(0,1){5}}
\put(200,47,5){\line(0,1){5}}
\put(220,47,5){\line(0,1){5}}
\put(240,47,5){\line(0,1){5}}

\put(-23,17.5){$\bullet$}
\put(-3,17.5){$\bullet$}
\put(37,17.5){$\bullet$}
\put(20,17,5){\line(0,1){5}}
\put(60,17,5){\line(0,1){5}}
\put(80,17,5){\line(0,1){5}}
\put(100,17,5){\line(0,1){5}}
\put(120,17,5){\line(0,1){5}}
\put(140,17,5){\line(0,1){5}}
\put(160,17,5){\line(0,1){5}}
\put(180,17,5){\line(0,1){5}}
\put(200,17,5){\line(0,1){5}}
\put(220,17,5){\line(0,1){5}}
\put(240,17,5){\line(0,1){5}}

\put(-69,50.2){\makebox(0,0){$A(\la) = $}}
\put(0,70){\makebox(0,0){$_{1}$}}
\put(20,70){\makebox(0,0){$_{2}$}}
\put(40,70){\makebox(0,0){$_{7}$}}
\put(60,70){\makebox(0,0){$_{8}$}}
\put(80,70){\makebox(0,0){$_{13}$}}
\put(100,70){\makebox(0,0){$_{14}$}}
\put(120,70){\makebox(0,0){$_{19}$}}
\put(140,70){\makebox(0,0){$_{20}$}}
\put(160,70){\makebox(0,0){$_{25}$}}
\put(180,70){\makebox(0,0){$_{26}$}}
\put(200,70){\makebox(0,0){$_{31}$}}
\put(220,70){\makebox(0,0){$_{32}$}}
\put(240,70){\makebox(0,0){$_{37}$}}
\put(262,80){\makebox(0,0){${\cdots}$}}
\put(-42,80){\makebox(0,0){${\cdots}$}}
\put(262,50){\makebox(0,0){${\cdots}$}}
\put(-42,50){\makebox(0,0){${\cdots}$}}
\put(262,20){\makebox(0,0){${\cdots}$}}
\put(-42,20){\makebox(0,0){${\cdots}$}}

\put(0,40){\makebox(0,0){$_{3}$}}
\put(-23,40){\makebox(0,0){$_{-2}$}}
\put(-23,70){\makebox(0,0){$_{-4}$}}
\put(-20,10){\makebox(0,0){$_{0}$}}
\put(20,40){\makebox(0,0){$_{4}$}}
\put(40,40){\makebox(0,0){$_{9}$}}
\put(60,40){\makebox(0,0){$_{10}$}}
\put(80,40){\makebox(0,0){$_{15}$}}
\put(100,40){\makebox(0,0){$_{16}$}}
\put(120,40){\makebox(0,0){$_{21}$}}
\put(140,40){\makebox(0,0){$_{22}$}}
\put(160,40){\makebox(0,0){$_{27}$}}
\put(180,40){\makebox(0,0){$_{28}$}}
\put(200,40){\makebox(0,0){$_{33}$}}
\put(220,40){\makebox(0,0){$_{34}$}}
\put(240,40){\makebox(0,0){$_{39}$}}

\put(0,10){\makebox(0,0){$_{5}$}}
\put(20,10){\makebox(0,0){$_{6}$}}
\put(40,10){\makebox(0,0){$_{11}$}}
\put(60,10){\makebox(0,0){$_{12}$}}
\put(80,10){\makebox(0,0){$_{17}$}}
\put(100,10){\makebox(0,0){$_{18}$}}
\put(120,10){\makebox(0,0){$_{23}$}}
\put(140,10){\makebox(0,0){$_{24}$}}
\put(160,10){\makebox(0,0){$_{29}$}}
\put(180,10){\makebox(0,0){$_{30}$}}
\put(200,10){\makebox(0,0){$_{35}$}}
\put(220,10){\makebox(0,0){$_{36}$}}
\put(240,10){\makebox(0,0){$_{41}$}}
\end{picture}
$$
It follows from this that $\bar\mu = (11,7^2,5^4,2^3)$ and $\bar\la = (11,7^2,5,4^2,2^4,1^5)$.

Now we define the second partial order in the $e > 0$
case by
declaring that $\mu \geq \la$ if $\hmu \unlhd \hla$ in the dominance
ordering on partitions.
Here are some examples:
\begin{itemize}
\item[(1)] If $l=1$ then we have simply that $\hla = \la$,
so $\geq$ is the same as $\unlhd$.
\item[(2)]
Suppose $l=2, e=2, (\kk_1,\kk_2) = (4,1)$,
$\mu = ((1^3),\varnothing)$ and $\la  = ((1), (2))$.
Then $\bar\mu = (4,2^2,1)$ and $\bar\la = (4,3,2)$.
In this example we have that $\mu > \la$ and $\mu \rhd \la$.
\item[(3)] Suppose $l=4, e=3$, and $(\kk_1,\kk_2,\kk_3,\kk_4)$
satisfies $\kk_1 > \kk_2 > \kk_3 > \kk_4$, 
$\kk_1 \equiv \kk_4 \equiv 0 \pmod{3}$ and
$\kk_2 \equiv \kk_3 \equiv 1 \pmod{3}$.
If
$\mu = (\varnothing, (2), \varnothing, (1))$ and
$\la = ((1),\emptyset, (2),
\emptyset)$ then one can check always that $\mu > \la$,
although $\mu$ and $\la$ are incomparable in the dominance ordering.
\end{itemize}
We remark that
in the proof of \cite[Theorem 3.4(2)]{Aclass}, Ariki appears to claim
for fixed $\mu, \la \in \Par$ and 
$\kk_1 \gg \kk_1 \gg \kk_3 \gg \kk_4$ 
that $\mu > \la$ implies $\mu \lhd \la$. 
The example (3) above shows that this is false. 
In Lemma~\ref{weaker} below, we prove a slightly weaker statement
which is still enough for the subsequent arguments 
in \cite{Aclass} to make sense, 
as we explain in detail later on.

Let $\lex$ denote the {\em lexicographic ordering} on partitions, so
for partitions $\la,\mu \in \mathscr{P}$ we have that
$\mu\lex\la$ if and only if $\mu_1=\la_1,\dots,\mu_{a-1}=\la_{a-1}$
and $\mu_a < \la_a$ for some $a \geq 1$. 
We extend this notion to $l$-multipartitions: for 
$\la, \mu \in \Par$ we have that $\mu\lex\la$ if and only 
$\mu^{(1)} = \la^{(1)},\dots,\mu^{(m-1)}=\la^{(m-1)}$ and
$\mu^{(m)} \lex \la^{(m)}$ for some $1 \leq m \leq l$. 
It is obvious that this total order refines the dominance ordering on $\Par$ 
in the sense
that $\mu \lhd \la$ implies $\mu \lex \la$.

\begin{Lemma}\label{weaker}
Assume we are given $\alpha \in Q_+$ such that $\kk_m - \kk_{m+1}
\geq \height(\alpha)+e$ for $m=1,\dots,l-1$.
Then $\mu > \la$ implies that $\mu \lex \la$ for all $\la,\mu \in \Par_\alpha$.
\end{Lemma}

\begin{proof}
The lemma is vacuous in the case $e=0$, as it never happens that
$\mu > \la$ under the given assumptions,
recalling from (\ref{zd}) that $\mu > \la$ implies $\cont(\mu)=\cont(\la)$
in the $e=0$ case.
Assume from now on that $e > 0$.
It suffices to show for $\la,\mu \in \Par_\alpha$ that 
$\mu \gex \la$ implies $\bar\mu \gex \bar\la$. 
For this, choose $1\leq m \leq l$ and $a \geq 1$ such that 
$\mu^{(1)}=\la^{(1)}, \dots, \mu^{(m-1)}=\la^{(m-1)}$,
$\mu^{(m)}_1 = \la^{(m)}_1,\dots, \mu^{(m)}_{a-1} = \la^{(m)}_{a-1}$
and $\mu^{(m)}_a > \la^{(m)}_a$.
The reader may find it helpful to keep in mind the examples of 
the corresponding abacus diagrams $A(\mu)$
and $A(\la)$ displayed above.

The rows $1,...,m-1$ of $A(\mu)$ and $A(\la)$ are exactly the same.
Moreover, in the $m$th row, all 
the beads corresponding to the parts $\mu^{(m)}_b=\la^{(m)}_b$ with $b<a$ 
occupy the same positions in $A(\mu)$ and $A(\la)$. 
Also the bead $B$ in $A(\mu)$ corresponding to the part $\mu^{(m)}_a$ 
is strictly to the right of the bead $B'$ in $A(\la)$ 
corresponding to the part $\la^{(m)}_a$ (the part $\la^{(m)}_a$ could be $0$ 
but it still makes sense to consider the corresponding bead).

Let $B$ occupy the position indexed by $p \in \Z$. 
By the choice of $m$ and $a$ and the 
assumptions on $\kk_1,\dots,\kk_l$, 
a position indexed by an integer
$> p$ is occupied
in $A(\mu)$ if and only if it is occupied in $A(\la)$.
Assume that there are $t$ such occupied 
positions in $A(\mu)$ (or $A(\la)$). 
Then we have that 
$\bar\mu_s=\bar\la_s$ for $s=1,2,\dots,t$. 

To finish the proof it suffices to show that $\bar\mu_{t+1}>\bar\la_{t+1}$. Note that $\bar\mu_{t+1}$ is equal to the number of unoccupied positions indexed by integers $<p$ in $A(\mu)$. 
By the assumptions on $\kk_1,\dots,\kk_l$, 
such positions can only exist in rows $m,m+1,\dots,l$. Moreover, the number 
of such positions in rows $m+1,\dots,l$ 
is determined just by $p$ and the fixed numbers 
$\kk_{m+1},\dots,\kk_l$. 

Now, let $p'$ be the largest integer 
such that $p'<p$ and the position indexed by $p'$ is occupied 
in $A(\la)$. By the assumptions,
$p'$ may index the position occupied by $B'$
or it may index a position in rows $1,\dots,m-1$ that is to the right of $B'$.
Note $\bar\la_{t+1}$ is equal to the number of unoccupied positions
indexed by integers $<p'$ in $A(\la)$. As in the previous paragraph, 
such unoccupied positions only exist in
rows $m,m+1,\dots,l$, and the number of such positions in rows $m+1,\dots,l$
is exactly the same as before.
Finally the number of unoccupied positions indexed by integers $< p'$ in row
$m$ of $A(\la)$ is always strictly smaller than the number of unoccupied positions indexed by integers $< p$ in row $m$ of $A(\mu)$ because of the
presence of the extra bead $B'$.
Hence $\bar\la_{t+1}  < \bar\mu_{t+1}$.
\end{proof}

\subsection{Fock space}\label{expre}
Now we proceed to introduce the higher level Fock space $F(\La)$
following the 
exposition in \cite{Abook}.
Given nodes $A$ and $B$ from the diagram of a multipartition,
we say that $A$ is {\em row-above} (resp.\ {\em row-below}) $B$
if $A$ lies in a row that is strictly above (resp.\ below) the row 
containing $B$
in the Young diagram when visualized as in (\ref{yde}).
Given $\la \in \Par$, $i \in I$, a removable $i$-node $A$
and an addable $i$-node $B$, 
we set
\begin{equation}\label{EDKWeight}
\:\:\:\:\:\: d_i(\la):=\#\{\text{addable $i$-nodes of $\la$}\}
-\#\{\text{removable $i$-nodes of $\la$}\};
\end{equation}
\begin{equation}\label{EDMUA}
\begin{split}
d_A(\la):= \#\{\text{addable $i$-nodes of $\la$ row-below $A$}\}\hspace{34.5mm}\\
-\#\{\text{removable $i$-nodes of $\la$ row-below $A$}\};\hspace{10mm}
\end{split}
\end{equation} 
\begin{equation}\label{EDMUB}
\begin{split}
d^B(\la):=\#\{\text{addable $i$-nodes of $\la$ row-above $B$}\}\hspace{34.5mm}\\
-\#\{\text{removable $i$-nodes of $\la$ row-above $B$}\}.\hspace{10mm}
\end{split}
\end{equation} 
Note that
$d_i(\la) = (\La - \cont(\la), \alpha_i)$.

Now define $F(\La)$ to be the $\Q(q)$-vector space on basis
$\{M_\la\:|\:\la \in \Par\}$
with $U_q(\g)$-action defined by
\begin{align}\label{act1}
E_i {M}_\la:=\sum_A q^{d_A(\la)}{M}_{\la_A},&\qquad
F_i {M}_\la:=\sum_B q^{-d^B(\la)}{M}_{\la^B},\\
K_i {M}_\la&:=q^{d_i(\la)}{M}_\la,\label{act3}
\end{align}
where the first sum is over all removable $i$-nodes $A$ for $\la$, and the 
second sum is over all addable $i$-nodes $B$ for $\la$. 
When $l=1$, this construction originates in work of Hayashi \cite{Hayashi}
and Misra and Miwa \cite{MiM}.
When $l > 1$, $F(\La)$ was first studied in \cite{JMMO}.
In that case it is simply the tensor product
of $l$ level one Fock spaces, indeed,
we can identify $F(\La)$ in general
with the tensor product
\begin{equation}\label{tp}
F(\La) = F(\La_{k_1}) \otimes \cdots \otimes F(\La_{k_l}),
\end{equation}
on which the $U_q(\g)$-structure is defined via
the comultiplication $\De$ fixed above,
so that
$M_\la$ is identified with
$M_{\la^{(1)}} \otimes \cdots \otimes M_{\la^{(l)}}$
for each $\la \in \Par$.

Let $F(\La)_\Laurent$ denote the free $\Laurent$-submodule
of $F(\La)$ spanned by the $M_\la$'s, which 
is invariant under the action of $U_q(\g)_\Laurent$.
Also let $F(\La)_0$  denote the 
free $\A_0$-submodule of $F(\La)$
spanned by the $M_\la$'s and set
$$
C(\La)_0 := \{M_\la + q F(\La)_0\:|\:\la \in \Par\}.
$$
The pair
$(F(\La)_0, C(\La)_0)$ is then
an upper crystal basis at $q=0$.
The proof of this statement in level one
goes back to Misra and Miwa \cite{MiM};
the proof for higher levels is a consequence of the level one result
in view of (\ref{tp})
and \cite[Proposition 6]{Ka1}.
Hence we get induced the structure of abstract crystal on the 
underlying index set $\Par$ that parametrizes $C(\La)_0$, 
with crystal datum denoted 
\begin{equation}\label{cp}
(\Par, \tilde e_i, \tilde f_i, \eps_i, \phi_i, \wt).
\end{equation}
We give an explicit combinatorial description of 
this crystal in the next subsection.
This explicit description in level one is a reformulation of
the results in \cite{MiM}; in higher levels it follows from the level one
description together with
(\ref{tp}) and \cite[Proposition 6]{Ka1}.

\subsection{Crystals}\label{scc}
The crystal datum (\ref{cp}) can be described in purely combinatorial
terms as follows.
First, for $\la \in \Par$, we have that $\wt(\la) = \La - \cont(\la)$,
as follows from (\ref{act3}).

Given also  $i \in I$, let $A_1,\dots,A_n$ denote the addable and removable $i$-nodes
of $\la$ ordered so that $A_m$ is row-above $A_{m+1}$ for each
$m=1,\dots,n-1$.
Consider the sequence $(\sigma_1,\dots,\sigma_n)$
where $\sigma_r = +$ if $A_r$ is addable or $-$ if $A_r$ is removable.
If we can find $1 \leq r < s \leq n$ such that $\sigma_r=  -$,
$\sigma_s = +$ and $\sigma_{r+1}=\cdots=\sigma_{s-1} = 0$
then replace $\sigma_r$ and $\sigma_s$ by $0$.
Keep doing this until we are left with a sequence
$(\sigma_1,\dots,\sigma_n)$ in which no $-$ appears to the left of
a $+$. This
is called the {\em reduced $i$-signature} of $\la$.

If $(\sigma_1,\dots,\sigma_n)$ is the reduced $i$-signature of $\la$,
then we have that
\begin{align*}
\eps_i(\la) = \#\{r=1,\dots,n\:|\:\sigma_r = -\},\qquad
&\phi_i(\la) = \#\{r=1,\dots,n\:|\:\sigma_r = +\}.
\end{align*}
By (\ref{us0}) (or directly from the combinatorics) we also have that
\begin{equation}\label{us}
(\La - \cont(\la), \alpha_i) = d_i(\la) = \phi_i(\la) - \eps_i(\la).
\end{equation}
Finally, if $\eps_i(\la) > 0$, we have that
$\tilde e_i \la = \la_{A_r}$ 
where $r$ indexes the leftmost $-$ in the reduced $i$-signature.
Similarly, if $\phi_i(\la) >0 $ we have that
$\tilde f_i \la = \la^{A_r}$ 
where $r$ indexes the rightmost $+$ in the reduced $i$-signature.

Because $\varnothing$ is a highest weight vector
in this crystal
of weight $\La$, we deduce from \cite[Theorem 3.3.1]{KaG} that the subcrystal
\begin{equation}\label{cpt}
(\RPar,\tilde e_i, \tilde f_i, \eps_i, \phi_i, \wt)
\end{equation}
that is the connected component 
of
$(\Par,\tilde e_i, \tilde f_i, \eps_i, \phi_i, \wt)$
generated by $\varnothing$
gives an explicit combinatorial realization of
the abstract crystal underlying the highest weight module
$V(\La)$.
We refer to multipartitions from $\RPar$
as {\em restricted multipartitions}.
Also for $\alpha \in Q_+$ set
\begin{equation*}
\RPar_\alpha := \RPar \cap \Par_\alpha.
\end{equation*}
These are the restricted multipartitions of content $\alpha$.

\begin{Remark}\rm\label{c3}
The problem of
finding a more explicit combinatorial description of the
subset $\RPar$ of $\Par$
has received quite a lot of attention in the literature;
see also Remark~\ref{remarkbelow} below.
Here are some special cases.

(1) Suppose that $e > 0$ and $l=1$.
Then $\RPar$
is the usual set of all
{\em $e$-restricted partitions}, that is, 
partitions $\la$ such that 
$\la_a - \la_{a+1} < e$ for $a \geq 1$.

(2)
Suppose that $e=0$ and 
$k_1 \geq \cdots \geq k_l$.
Then $\RPar$
consists of all $l$-multipartitions $\la$
such that $\la_{a+k_m - k_{m+1}}^{(m)}
\leq \la_a^{(m+1)}$ for all $m=1,\dots,l-1$ and $a \geq 1$;
see \cite[(2.52)]{BKariki} or \cite{Vaz}.
\end{Remark}

\subsection{\boldmath The dual-canonical basis
of $V(\La)$}\label{dcqc}
The vector $M_{\varnothing}$ is a highest weight vector
of weight $\La$.
Moreover, the $\La$-weight space of $F(\La)$
is one dimensional. By complete reducibility, 
it follows that there is a canonical $U_q(\g)$-module
homomorphism
\begin{equation}\label{pidef}
\pi:F(\La) \twoheadrightarrow V(\La),
\qquad
M_\varnothing \mapsto v_\La.
\end{equation}
For any $\la \in \Par$, we define
\begin{equation}\label{stmn}
S_\la := \pi(M_\la),
\end{equation}
and call this a {\em standard monomial} in $V(\La)$.
Applying $\pi$ to (\ref{act1}), we get that
\begin{align}\label{sact1}
E_i S_\la=\sum_A q^{d_A(\la)}S_{\la_A},&\qquad
F_i S_\la=\sum_B q^{-d^B(\la)}S_{\la^B},
\end{align}
where the first sum is over all removable $i$-nodes $A$ for $\la$, and the second sum is over all addable $i$-nodes $B$ for $\la$. 

By \cite[Theorem 3.3.1]{KaG}, 
the upper crystal lattice $V(\La)_0$
from $\S$\ref{md} coincides with
the image under $\pi$ of the upper crystal lattice
$F(\La)_0$ from $\S$\ref{expre}, i.e.
$V(\La)_0$ is the $\A_0$-span of the
standard monomials.
Moreover by the definition (\ref{cpt})
we have that
\begin{equation}\label{bla0}
B(\La)_0 = \{S_\la+qV(\La)_0\:|\:\la \in \RPar\}.
\end{equation}
Thus we have given an explicit construction of
$(V(\La)_0, B(\La)_0)$,
the upper crystal basis of $V(\La)$ at $q=0$,
via the Fock space $F(\La)$.

Recall also from $\S$\ref{md}
that the dual-canonical basis of $V(\La)$
is the canonical lift of $B(\La)_0$ using 
the balanced triple
$(V(\La)^*_\A, V(\La)_0, \overline{V(\La)_0})$.
In other words, in terms of our explicit parametrization,
the dual-canonical basis of $V(\La)$ is the basis
$\{D_\la\:|\:\la \in \RPar\}$
in which $D_\la$ is the unique vector 
in $V(\La)_\A^* \cap V(\La)_0 \cap \overline{V(\La)_0}$ such that
\begin{equation}\label{glory}
D_\la \equiv S_\la \pmod{q V(\La)_0}
\end{equation}
for each $\la \in \RPar$.
As we noted before, this is already a basis for
the costandard lattice $V(\La)_\Laurent^*$ as a free
$\Laurent$-module, and each $D_\la$ is bar-invariant.

\begin{Proposition}[{\cite[Proposition 5.3.1]{KaG}}]
\label{dca}
For $\la \in \RPar_\alpha$ 
and $i \in I$ we have that
\begin{align*}
E_i D_\la &=
[\eps_i(\la)] D_{\tilde e_i \la}
+ \sum_{\substack{\mu \in \RPar_{\al-\al_i}\\
\eps_i(\mu) < \eps_i(\la)-1
}}
x_{\la,\mu;i}(q) D_\mu,\\
F_i D_\la &=
[\phi_i(\la)] D_{\tilde f_i \la}
+ \sum_{\substack{\mu \in \RPar_{\al+\al_i} \\
\phi_i(\mu) < \phi_i(\la)-1
}}
y_{\la,\mu;i}(q) D_\mu,
\end{align*}
for bar-invariant
 $x_{\la,\mu;i}(q) \in q^{\eps_i(\la)-2} \Z[q^{-1}]$ and
$y_{\la,\mu;i}(q) \in q^{\phi_i(\la)-2} \Z[q^{-1}]$.
(In these two formulae, the first term on the right hand side should
be interpreted as 0 if $\eps_i(\la)=0$ (resp.\ $\phi_i(\la)=0$).)
\end{Proposition}

Finally for $\mu \in \Par_\alpha$, consider the expansion
of the standard monomial $S_\mu$ in terms of the dual-canonical basis:
\begin{equation}\label{qdec}
S_\mu = \sum_{\la \in \RPar_\alpha} d_{\la,\mu}(q) D_\la.
\end{equation}
At this point all we know about the coeffcients
$d_{\la,\mu}(q)$ is that they belong to
$\delta_{\la,\mu}+q\A_0$.

\begin{Remark}
We will prove eventually that 
$d_{\la,\mu}(q) = 1$ if $\la = \mu$,
$d_{\la,\mu}(q)=0$ if $\la \not\!\!\!\unlhd\ \mu$,
and $d_{\la,\mu}(q) \in q \Z[q]$ if $\la \lhd \mu$;
see Theorem~\ref{tri} and Corollary~\ref{maindecthm}.
Moreover we will show that
$d_{\la,\mu}(q)$ is equal to the 
multiplicity $[S(\mu):D(\la)]_q$
of a certain irreducible graded module $D(\la)$
as a composition factor of the graded Specht module $S(\mu)$ for the
cyclotomic Hecke algebra associated to $\La$,
which will imply further
that the
coefficients of the polynomials 
$d_{\la,\mu}(q)$ are non-negative integers.
\end{Remark}

\subsection{\boldmath Triangularity of standard monomials}\label{tzero}
In order to establish the desired triangularity properties
of the coefficients $d_{\la,\mu}(q)$,
we need to exploit the existence of 
a well-behaved bar-involution on $F(\La)$.
Unfortunately
the construction of this bar-involution
 in the case $e > 0$ is rather indirect, so we prefer
to assume its existence first and proceed to derive the
important consequences, 
postponing the actual
construction until later on; see $\S$\ref{construct}. 

\begin{Hypothesis}\label{assump}
We are given an explicit compatible bar-involution on $F(\La)$
and a partial order $\preceq$ on $\Par$ such that
\begin{itemize}
\item[(1)] $\overline{M_\la} = M_\la + 
(\text{a $\Z[q,q^{-1}]$-linear combination of $M_\mu$'s
for $\mu \prec \la$})$;
\item[(2)] $\mu \prec \la$ implies $\mu \lex \la$.
\end{itemize}
\end{Hypothesis}

Let us explain right away how to construct such a bar-involution
in the case $e=0$;
note this approach does not work for $e > 0$, the problem being 
the lack of integrality of
Lusztig's quasi-$R$-matrix in the affine case.
First we take the partial order
$\preceq$ in the $e=0$ case to be the partial order $\leq$ from (\ref{zd}),
which satisfies Hypothesis~\ref{assump}(2) by Lemma~\ref{zerot}.
Then, to construct the bar-involution itself, we start 
in level one by defining the bar-involution on $F(\La)$ simply
to be the unique anti-linear endomorphism
fixing all the
monomial basis vectors $M_\la$.
It is easy to check from (\ref{act1}) 
that this is a compatible
bar-involution (assuming of course that $e=0$ and $l=1$).
For higher levels,
we identify $F(\La)$ with the tensor product
(\ref{tp}) and use Lusztig's tensor product construction 
from \cite[$\S$27.3]{Lubook} (adapted to our choice of comultiplication)
to get  an induced compatible bar-involution on
$F(\La)$.
It is immediate from this construction, the definition (\ref{zd}) 
and the integrality of the quasi-$R$-matrix
from \cite[$\S$24.1]{Lubook} that this satisfies 
Hypothesis~\ref{assump}(1); see \cite[$\S$2.3]{BKariki}.

For the remainder of the subsection, we assume that
Hypothesis~\ref{assump} holds.  Then we can introduce
the {\em dual-canonical basis}
$\{L_\la\:|\:\la \in \Par\}$ of 
$F(\La)$ by letting $L_\la$ denote the unique
bar-invariant vector in 
$F(\La)$ such that
\begin{equation}\label{lllk}
L_\la = M_\la + \text{(a $q\Z[q]$-linear combination of
$M_\mu$'s with $\mu \prec \la$})
\end{equation}
for each $\la \in \Par$.
The existence and uniqueness of these vectors follows from
Lusztig's lemma \cite[Lemma 24.2.1]{Lubook} and 
the triangularity of the bar-involution from
Hypothesis~\ref{assump}(1).
Recall the map $\pi$
and the dual-canonical basis
vectors $D_\la \in V(\La)$ from $\S$\ref{dcqc}.
The following theorem was established already 
in the case $e=0$ in \cite[Theorem 2.2]{BKariki}
(via \cite[Theorem 26]{Bdual}); the
proof given here in the general case repeats the same argument.

\begin{Proposition}\label{katheorem}
For $\la \in \Par$, we have that
$\displaystyle\pi(L_\la) = \left\{\begin{array}{ll}
D_\la&\text{if $\la \in \RPar$}\\
0&\text{otherwise.}
\end{array}\right.$
\end{Proposition}

\begin{proof}
Let $F(\La)_\A := \Q \otimes_{\Z} F(\La)_\Laurent$.
In view of (\ref{lllk}),
this can be described alternatively
as the $\A$-span of the
dual-canonical basis $\{L_\la\:|\:\la \in \Par\}$.
Similarly, the upper crystal lattice $F(\La)_0$
(resp.\ its image $\overline{F(\La)_0}$ under the bar-invoution)
is the $\A_0$-span (resp.\ the $\A_\infty$-span)
of the dual-canonical basis.
It follows immediately that 
$(F(\La)_\A,
F(\La)_0, \overline{F(\La)_0})$
is a balanced triple. Moreover, our dual-canonical basis of $F(\La)$ is
the canonical lift of the upper crystal basis
$C(\La)_0$ arising from this balanced triple.
Also the 
image of the upper crystal lattice $F(\La)_0$ at $q=0$ 
under the bar-involution
is an upper
crystal lattice at $q=\infty$.
This puts us in the setup of \cite[$\S$5.2]{KaG}.

By \cite[Proposition 5.2.1]{KaG},
the image of
$(F(\La)_\A,
F(\La)_0, \overline{F(\La)_0})$
under the map $\pi$ is a balanced triple
in $V(\La)$.
Its intersection with the $\La$-weight space
of $V(\La)$ is $(\A v_\La, \A_0 v_\La, \A_\infty v_\La)$,
which is the same thing as for the
balanced triple $(V(\La)^*_\A, V(\La)_0, V(\La)_\infty)$
constructed earlier.
Hence by \cite[Proposition 5.2.2]{KaG}
our two balanced triples coincide:
$$
\pi(F(\La)_\A) = V(\La)^*_\A,\quad
\pi(F(\La)_0) = V(\La)_0,\quad
\pi(\overline{F(\La)_0}) = \overline{V(\La)_0}.
$$
As $L_\la \in F(\La)_\A \cap F(\La)_0 \cap \overline{F(\La)_0}$
we deduce that $\pi(L_\la) \in V(\La)_\A^* \cap V(\La)_0
\cap \overline{V(\La)_0}$ for every $\la \in \Par$.

Now (\ref{lllk}) implies that $M_\la \equiv L_\la \pmod{q F(\La)_0}$
for every $\la$.
Also we know that
$S_\la \equiv D_\la\pmod{q V(\La)_0}$
if $\la \in \RPar$ and $S_\la \equiv 0 \pmod{q V(\La)_0}$
otherwise.
As $\pi(M_\la) = S_\la$, we deduce that $\pi(L_\la)$ is equal to $D_\la$
plus a $q \A_0$-linear combination of $D_\mu$'s if
$\la \in \RPar$, and
$\pi(L_\la)$ is a $q \A_0$-linear combination of $D_\mu$'s otherwise.
But also $\pi(L_\la) \in \overline{V(\La)_0}$
for every $\la$, so it is an $\A_\infty$-linear
combination of $D_\mu$'s.
Since $\A_\infty \cap q \A_0 = \{0\}$, 
we conclude that $\pi(L_\la) = D_\la$ if $\la \in \RPar$
and $\pi(L_\la) = 0$ otherwise.
\end{proof}

Now we can define polynomials $d_{\la,\mu}(q) \in \Z[q]$
for {\em every} $\la,\mu \in \Par$
from the expansion
\begin{equation}\label{ml}
M_\mu = \sum_{\la \in \Par} d_{\la,\mu}(q) L_\la.
\end{equation}
Applying the map $\pi$ to (\ref{ml}) and using Proposition~\ref{katheorem},
we get that
\begin{equation}\label{mls}
S_\mu = \sum_{\la \in \RPar} d_{\la,\mu}(q) D_\la
\end{equation}
for $\mu \in \Par$ and $\la \in \RPar$.
This establishes that the
polynomial $d_{\la,\mu}(q)$ defined here agrees
with the one defined earlier in (\ref{qdec})
when $\la \in \RPar$, so our notation is consistent with the earlier
notation.

\begin{Theorem}\label{tri}
Given $\la,\mu \in \Par_\alpha$ 
we have that $d_{\la,\mu}(q) = 1$ if $\la=\mu$,
$d_{\la,\mu}(q)= 0$
if $\la \not\preceq \mu$, and
$d_{\la,\mu}(q) \in q \Z[q]$ if $\la \prec \mu$.
Hence:
\begin{itemize}
\item[(1)] The vectors $\{S_\la\:|\:\la \in \RPar\}$
give a basis for $V(\La)_\Laurent^*$ as a free $\Laurent$-module.
\item[(2)] For $\la \in \Par_\alpha \setminus\RPar_\alpha$, the standard monomial
$S_\la$ 
can be expressed as a $q \Z[q]$-linear combination of
$S_\mu$'s for $\mu \in \RPar_\alpha$ with $\mu \prec \la$.
\item[(3)] For $\la \in \RPar_\alpha$, the difference
$S_\la - \overline{S_\la}$ is a $q \Z[q]$-linear combination
of $S_\mu$'s for $\mu \in \RPar_\alpha$ with $\mu \prec \la$.
\end{itemize}
\end{Theorem}

\begin{proof}
Use (\ref{lllk}), Hypothesis~\ref{assump}
and the fact that $\{D_\la\:|\:\la \in \RPar\}$
is a bar-invariant basis for $V(\La)^*_\Laurent$
as a free $\Laurent$-module.
\end{proof}

\begin{Remark}\label{blox}
When $e=0$, 
the Fock space $F(\La)$
is categorified by a certain graded highest weight category
arising from parabolic category $\mathcal O$ 
attached to the finite general linear Lie algebra
$\mathfrak{gl}_n(\C)$. The monomial basis $\{M_\la\}$
corresponds to the standard objects in this category and
the dual-canonical basis
$\{L_\la\}$ corresponds to the irreducible objects.
Apart from the grading (which comes via \cite{BGS})
this is developed in detail in \cite{BKariki};
see especially \cite[Theorem 3.1]{BKariki}.
When $e > 0$ we expect
that 
the Fock space $F(\La)$
should be categorified in similar fashion
by the cyclotomic $\xi$-Schur algebras
of \cite{DJM} for $\xi$ a primitive
$e$th root of unity (though many 
questions about the grading remain open); see \cite{Y1} for a 
related conjecture.
We speculate that there may be a 
version of the theory of \cite{BKschur} establishing
a Morita equivalence between the
cyclotomic $\xi$-Schur algebras and certain blocks of quantum
parabolic category $\mathcal O$ at the root of unity $\xi$.
\end{Remark}

\subsection{\boldmath The quasi-canonical basis}\label{sqcb}
In this subsection we continue to assume that
Hypothesis~\ref{assump} holds.
Introduce a new basis $\{P_\la\:|\:\la \in \Par\}$
for $F(\La)$, which we call the {\em quasi-canonical basis},
by setting
\begin{equation}\label{dark}
P_\la := \sum_{\mu \in \Par} d_{\la,\mu}(q) M_\mu.
\end{equation}
So we have simply transposed the transition matrix
appearing in (\ref{ml}), i.e. we are mimicking 
BGG reciprocity at a combinatorial level.
Let $(p_{\la,\mu}(-q))_{\la, \mu \in \Par}$ be the inverse of 
the unitriangular matrix $(d_{\la,\mu}(q))_{\la,\mu \in \Par}$, so that
\begin{equation}\label{pm}
M_\la = \sum_{\mu \in \Par} p_{\la,\mu}(-q) P_\mu,
\qquad
L_\mu = \sum_{\la \in \Par} p_{\la,\mu}(-q) M_\la
\end{equation}
by (\ref{ml}) and (\ref{dark}).
Also define a sesquilinear form $\langle.,.\rangle$ on $F(\La)$,
which we call the {\em Shapovalov form},
by declaring that
\begin{equation}\label{starz}
\langle 
M_\la, \overline{M_\mu}\rangle := \delta_{\la,\mu}
\end{equation}
for all $\la, \mu \in \Par$.
Note that $\langle v, w \rangle = \langle \overline{w}, \overline{v}\rangle$
for all $v, w \in F(\La)$.
Moreover using (\ref{starz}) it is routine to check that
\begin{equation}\label{shapprop}
\langle xu, v \rangle = \langle u, \tau(x) v\rangle
\end{equation}
for all $x \in U_q(\g)$ and $u, v \in F(\La)$; see also
\cite[(2.41)]{BKariki}.

\begin{Lemma} \label{duals}
For $\la,\mu \in \Par$, we have that $\langle P_\la, L_\mu \rangle =
\delta_{\la,\mu}$.
\end{Lemma}

\begin{proof}
Since $L_\mu$ is bar-invariant, we have from 
(\ref{dark}), (\ref{pm}) and (\ref{shapprop})
that
\begin{align*}
\langle P_\la, L_\mu \rangle &=
\big\langle \sum_{\sigma \in \Par} 
d_{\la,\sigma}(q) M_\sigma, \sum_{\tau \in \Par} p_{\tau,\mu}(-q^{-1}) \overline{M_\tau} \big\rangle\\
&=
\sum_{\sigma,\tau \in \Par}
d_{\la,\sigma}(q^{-1}) p_{\tau,\mu}(-q^{-1}) \langle M_\sigma, \overline{M_\tau}\rangle\\
&=
\sum_{\sigma \in \Par}
d_{\la,\sigma}(q^{-1}) p_{\sigma,\mu}(-q^{-1}) = \delta_{\la,\mu}.
\end{align*}
This proves the lemma.
\end{proof}

Next recall the definition of the Shapovalov form 
$\langle.,.\rangle$ on $V(\La)$ from $\S$\ref{md}.
We introduce a new basis 
$\{Y_\la\:|\:\la \in \RPar\}$ for $V(\La)$
 by declaring that
\begin{equation}\label{sdef}
\langle Y_\la, D_\mu \rangle = \delta_{\la,\mu}
\end{equation}
for all 
$\la,\mu \in \RPar$.
This is actually a basis for the standard lattice $V(\La)_\Laurent$
as a free $\Laurent$-module.
We call it the {\em quasi-canonical basis} of $V(\La)$.
The precise relationship between the quasi-canonical and the usual canonical
basis of $V(\La)$ is explained by the following lemma.

\begin{Lemma}\label{party}
The canonical basis for $V(\La)$
is $\bigcup_{\alpha \in Q_+} \{q^{-\defect(\alpha)}Y_\la\:|\:\la \in \RPar_\alpha\}$.
In particular, we have that
$\overline{Y_\la} = q^{-2\defect(\alpha)} Y_\la$
for each $\la \in \RPar_\alpha$.
\end{Lemma}

\begin{proof}
Recall that the canonical basis for $V(\La)$ is the dual 
basis to the dual-canonical basis with respect to the contravariant
form $(.,.)$. Moreover
vectors from the canonical basis are bar-invariant.
Using these two things, the lemma follows from 
the definition (\ref{sdef}) and Lemma~\ref{contra2}.
\end{proof}

Since the vector $M_\varnothing \in F(\La)$ is a non-zero highest weight
vector of weight $\La$, there is a canonical embedding
\begin{equation}\label{pidual}
\pi^*:V(\La) \hookrightarrow F(\La),
\qquad
v_\La
\mapsto M_\varnothing.
\end{equation}
The Shapovalov form on $V(\La)$ is actually the restriction of the
Shapovalov form on $F(\La)$ via this embedding, that is, we have that
\begin{equation}\label{sameshap}
\langle v, v' \rangle = \langle \pi^*(v), \pi^*(v')\rangle
\end{equation}
for all $v, v' \in V(\La)$.
This holds because it is true when $v=v'=v_\La$, and both forms have the
property (\ref{shapprop}).
Note also for $\pi$ as in (\ref{pidef})
that
\begin{equation}\label{piinv}
\pi \circ \pi^* = \id_{V(\La)}.
\end{equation}
This holds because it is true on the highest weight vector $v_\La$.
The following lemma shows that
$\pi^*$ is adjoint to $\pi$ with respect to the
Shapovalov forms.

\begin{Lemma}\label{adjl}
We have that
$\langle v, \pi(w)\rangle
= \langle \pi^*(v),w\rangle$
for all $v \in V(\La)$ and $w \in F(\La)$.
\end{Lemma}

\begin{proof}
Consider the orthogonal complement
$(\ker\pi)^\perp$ to $\ker\pi$
with respect to the form $\langle.,.\rangle$.
By Proposition~\ref{katheorem},
$\ker \pi$ has basis 
$\{L_\mu\:|\:\mu \in \Par \setminus \RPar\}$.
Hence by Lemma~\ref{duals},
the vector $\pi^*(v_\La) = M_\varnothing = P_\varnothing$
belongs to $(\ker \pi)^\perp$.
Moreover 
$(\ker\pi)^\perp$ is a $U_q(\g)$-submodule
of $F(\La)$ thanks to (\ref{shapprop}).
As $v_\La$ generates $V(\La)$, 
we deduce that the image of $\pi^*$
lies in $(\ker\pi)^\perp$.
Now take any $v \in V(\La)$ and $w \in F(\La)$.
By (\ref{piinv}), we can write $w = \pi^*(v') + z$ for some 
$v' \in V(\La)$ and $z \in \ker \pi$.
Using (\ref{sameshap}) and the observation just made,
we get that
$$
\langle v, \pi(w) \rangle = \langle v, v' \rangle
= \langle \pi^*(v), \pi^*(v')\rangle
=
\langle \pi^*(v), \pi^*(v') +z\rangle = \langle \pi^*(v), w \rangle,
$$
as required.
\end{proof}

In the case $e=0$, the following theorem was established 
in \cite[$\S$2.6]{BKariki}.

\begin{Theorem}\label{bgg}
We have that $\pi^*(Y_\la) = P_\la$
and $\pi(P_\la) = Y_\la$
for any $\la \in \RPar$.
Hence
\begin{equation}
Y_\la = \sum_{\mu \in \Par} d_{\la,\mu}(q) S_\mu
\end{equation}
for all $\la \in \RPar$.
\end{Theorem}

\begin{proof}
Applying Lemma~\ref{adjl}, Proposition~\ref{katheorem} 
and the definition (\ref{sdef}), 
we have for $\la \in \RPar$ and any $\mu \in \Par$ that
$$
\langle \pi^*(Y_\la), L_\mu \rangle
=
\langle Y_\la, \pi(L_\mu) \rangle
= 
\langle Y_\la, D_\mu \rangle = \delta_{\la,\mu}.
$$
This establishes that $\pi^*(Y_\la) = P_\la$ thanks to Lemma~\ref{duals}.
Combined with (\ref{piinv}) we deduce that $Y_\la = \pi(P_\la)$.
The final statement follows by applying $\pi$ 
to the definition (\ref{dark}).
\end{proof}

\begin{Remark}
More generally, we can define vectors $Y_\la \in V(\La)$
for any $\la \in \Par$ by setting
$Y_\la := \pi(P_\la) = \sum_{\mu \in \Par} d_{\la,\mu}(q) S_\mu$.
These are expected to correspond to {\em Young modules}
at the categorical level; see \cite[Theorem 4.6]{BKariki} where this is justified in
the case $e=0$.
\end{Remark}

\begin{Remark}\label{rcan}
Most of the rest of the literature in this subject
works with the {\em canonical basis} $\{T_\la\:|\:\la \in \Par\}$
for $F(\La)$ rather than the quasi-canonical basis introduced here,
where $T_\la \in F(\La)$ is the unique bar-invariant vector with
\begin{equation}
\overline{T_\la} = M_\la + (\text{a $q^{-1}\Z[q^{-1}]$-linear combination
of $M_\mu$'s with $\mu \prec \la$}).
\end{equation}
In the categorification mentioned in Remark~\ref{blox}, the
canonical basis $\{T_\la\}$ 
should correspond to the indecomposable tilting modules,
whereas the quasi-canonical basis $\{P_\la\}$
corresponds to the indecomposable projectives; see \cite[Theorem 3.1]{BKariki}
in the $e=0$ case and 
\cite[Theorem 11]{VV1} (combined with \cite[Proposition 4.1.5(ii)]{D})
for $e > 0, l = 1$.
Actually there is a close connection
between the quasi-canonical and canonical bases, as follows.
Let 
\begin{equation}
\La^t := \La_{-k_l}+\cdots+\La_{-k_1}
\end{equation} 
and define
the Fock space $F(\La^t)$ as above but
replacing $(k_1,\dots,k_l)$ everywhere
with $(-k_l,\dots,-k_1)$.
For an $l$-multipartition $\la$, set
\begin{equation}\label{tra}
\la^t := ((\la^{(l)})^t,\dots,(\la^{(1)})^t),
\end{equation}
where
$(\la^{(m)})^t$ denotes the usual transpose of a partition.
Then the anti-linear vector space isomorphism
\begin{equation}
t:F(\La) \stackrel{\sim}{\rightarrow} F(\La^{t}),\qquad
M_\la \mapsto M_{\la^t}
\end{equation}
has the property that $(P_\la)^t = T_{\la^t}$
for each $\la \in \Par$.
We omit the proof
since we do not need this result here.
The isomorphism $t$ corresponds to {\em Ringel duality} at 
the categorical level; see \cite{Mtilt}.
\end{Remark}

\begin{Remark}\label{monday}
The canonical basis $\{T_\la\:|\:\la \in \Par\}$ 
from Remark~\ref{rcan} is a
lower global crystal basis in the sense of Kashiwara \cite{KaG}.
The underlying lower crystal operators $\tilde e_i'$ and $\tilde f_i'$
induce another structure 
\begin{equation}\label{cp2}
(\Par, \tilde e_i', \tilde f_i', \eps_i', \phi_i', \wt)
\end{equation}
of abstract crystal
on the index set $\Par$ that is different from the one in (\ref{cp}).
It can be described explicitly in exactly the same way as in $\S$\ref{scc},
except that at the beginning we list the addable and removable
$i$-nodes of $\la$ as $A_1,\dots,A_n$ so that $A_m$ is row-below
$A_{m+1}$ for each $m=1,\dots,n-1$ (the reverse order to the one used before).
This follows because Kashiwara's combinatorial tensor product rule
for lower crystal bases at $q=\infty$ from \cite[Theorem 1]{Ka2}
is the opposite of the one for upper crystal bases at $q=0$
from \cite[Proposition 6]{Ka1}.
Equivalently, by direct comparison of the combinatorics,
it is the case that
\begin{equation}\label{tcrys}
\tilde f_i'\la = (\tilde f_{-i}(\la^t))^t
\end{equation}
for each $i \in I$ and $\la \in \Par$,
where $\la^t\in\Part$ is as in (\ref{tra})
and $\tilde f_{-i}$ 
is the upper 
crystal operator on $\Part$
defined exactly as in $\S$\ref{scc} but computing residues via the 
multicharge
$(-\kk_l,\dots,-\kk_1)$ instead of $(\kk_1,\dots,\kk_l)$.
\end{Remark}

\begin{Remark}\label{primecrys}
Let $(\RPar)' := (\RPart)^t$, where $\RPart$ is the set from (\ref{cpt})
but defined from $(-\kk_l,\dots,-\kk_1)$ instead
of $(\kk_1,\dots,\kk_l)$.
We refer to elements of $(\RPar)'$ as {\em regular multipartitions}.
In view of (\ref{tcrys}), this is the vertex set of the connected component
of the crystal (\ref{cp2}) 
generated by the empty multipartition 
$\varnothing$, which gives another realization
\begin{equation}
((\RPar)', \tilde e_i', \tilde f_i', \eps_i', \phi_i',\wt)
\end{equation}
of the abstract crystal attached to the module $V(\La)$ 
different to the one in (\ref{cpt}). 
There is a canonical bijection
\begin{equation}\label{cisod}
\RPar \rightarrow (\RPar)',\qquad \la \mapsto \la'
\end{equation}
such that $\varnothing' = \varnothing$ and
$\tilde f'_i(\la') = (\tilde f_i\la)'$ for all $\la \in \RPar$
and $i \in I$.
For example, in the special case that $e > 0$ and $l = 1$, 
the set $(\RPar)'$ is the usual set of
{\em $e$-regular partitions}, that is, partitions $\la$
that do not have $e$ or more non-zero repeated parts.
The map $\la \mapsto \la'$ in this case is the composition first
of the map $\la \mapsto \la^t$ followed by the Mullineux involution on
$e$-regular partitions; see \cite{KBrIII, Kmull, FK}.
\end{Remark}

\begin{Remark}
In view of Lemma~\ref{party} and Theorem~\ref{bgg},
the vectors 
$q^{-\defect(\alpha)} P_\la$ for $\la \in \RPar_\alpha$
must coincide with some of the
canonical basis elements of $F(\La)$ from Remark~\ref{rcan}.
We can make this precise
using Remark~\ref{primecrys}: we have that
\begin{equation}
P_\la = q^{\defect(\alpha)} T_{\la'}
\end{equation}
for each $\la \in \RPar_\alpha$,
where $\la'$ is as in
(\ref{cisod}).
Arguing exactly as in the proof of \cite[Corollary 2.8]{BKariki},
it follows easily for each $\la \in \RPar_\alpha$ and $\mu \in \Par_\alpha$
that
\begin{itemize}
\item[(1)] $d_{\la,\mu}(q) = 0$ unless $\la \preceq \mu \preceq \la'$;
\item[(2)] $d_{\la,\mu}(q) \in q \Z[q] \cap q^{\defect(\alpha)-1}\Z[q^{-1}]$
if $\la \prec \mu \prec \la'$;
\item[(3)] $d_{\la,\la}(q) = 1$ and $d_{\la,\la'}(q) = q^{\defect(\alpha)}$.
\end{itemize}
\end{Remark}

\subsection{Twisted Fock space}\label{sug}
We now turn our attention to the problem of constructing
a bar-involution on $F(\La)$ 
as in Hypothesis~\ref{assump} when $e > 0$.
In preparation for this, we need to recall
the twisted version
of Fock space, whose construction in higher levels is
due to Takemura and Uglov \cite{TU}.
Our exposition
follows 
\cite[$\S$2.5]{BKariki} in the case $e=0$
and \cite{Uglov} in the case $e > 0$ (noting our $q$ is equal to
$q^{-1}$ there).

We first
introduce a different ordering on the nodes of a multipartition.
Say that a node $A = (a,b,m)$ is {\em residue-above} 
node $B = (c,d,n)$ (or $B$ is {\em residue-below} $A$)
if either
$\res\, A > \res\, B$ or $\res\, A = \res\, B$ and $m > n$.
The following lemma relates this ordering on nodes
to the one used earlier.

\begin{Lemma}\label{sdsd}
Let $\la \in \Par_\alpha$
and $i \in I$.
\begin{itemize}
\item[(1)]
Assume $\kk_{m} - \kk_{m+1} \geq \height(\alpha)$
for all $m=1,\dots,l-1$.
Let $A$ be a removable $i$-node for $\la$ and $B$ be either an 
addable or a removable $i$-node for $\la$.
Then $B$ is row-below $A$ if and only if $B$ is residue-below $A$.
\item[(2)]
Assume $\kk_{m} - \kk_{m+1} \geq \height(\alpha)+1$
for all $m=1,\dots,l-1$.
Let $A$ be either an addable or a removable $i$-node for $\la$
and $B$ be an addable $i$-node for $\la$.
Then $A$ is row-above $B$ if and only if $A$ is residue-above $B$.
\end{itemize}
\end{Lemma}

Given $\la \in \Par$, $i \in I$, a removable $i$-node $A$
and an addable $i$-node $B$, 
we set
\begin{equation}\label{EDMUA2}
\begin{split}
\ddd_A(\la):= \#\{\text{addable $i$-nodes of $\la$ residue-below $A$}\}\hspace{17.5mm}\\
-\#\{\text{removable $i$-nodes of $\la$ residue-below $A$}\};
\end{split}
\end{equation} 
\begin{equation}\label{EDMUB2}
\begin{split}
\ddd^B(\la):=\#\{\text{addable $i$-nodes of $\la$ residue-above $B$}\}\hspace{17.5mm}\\
-\#\{\text{removable $i$-nodes of $\la$ residue-above $B$}\}.
\end{split}
\end{equation}
By Lemma~\ref{sdsd}, 
we have that $\ddd_A(\la) = d_A(\la)$
under the hypotheses of part (1) of the lemma,
and $\ddd^B(\la) = d^B(\la)$
under the hypotheses of part (2).

Now we can define the twisted Fock space $\widetilde{F}(\La)$
to be the $\Q(q)$-vector
space on basis $\{\widetilde{M}_\la\:|\:\la \in \Par\}$.
We make $\widetilde{F}(\La)$ into a $U_q(\g)$-module
by defining
\begin{align}\label{uact1}
E_i \widetilde{M}_\la:=\sum_A q^{\ddd_A(\la)}\widetilde{M}_{\la_A},&\qquad
F_i \widetilde{M}_\la:=\sum_B q^{-\ddd^B(\la)}\widetilde{M}_{\la^B},\\
 K_i \widetilde{M}_\la&:=q^{d_i(\la)}\widetilde{M}_\la,\label{uact3}
\end{align}
where the first sum is over all removable $i$-nodes $A$ for $\la$, and the 
second sum is over all addable $i$-nodes $B$ for $\la$. 
These are almost the same as the formulae (\ref{act1})--(\ref{act3}), 
but we have replaced
$d_A(\la)$ and $d^B(\la)$ from before with
$\ddd_A(\la)$ and $\ddd^B(\la)$ defined using the new ordering on nodes.

If $l=1$ we have simply that $\ddd_A(\la) = d_A(\la)$
and $\ddd^B(\la) = d^B(\la)$ for all addable nodes $A$ and removable
nodes $B$, so in this case
 we can simply
{\em identify} $\widetilde{F}(\La)$ with $F(\La)$
by identifying $\widetilde{M}(\la)$ with $M(\la)$
for each $\la \in \Par$.
In particular this shows that the formulae (\ref{uact1})--(\ref{uact3})
give a well-defined action of $U_q(\g)$ on $\widetilde{F}(\La)$
in the level one case,
since we already knew that for $F(\La)$.
For a proof that this action is well defined
for arbitrary level and $e > 0$, we refer to \cite[Theorem 2.1]{Uglov}.
When $e=0$ there is a different approach
noted in \cite[$\S$2.5]{BKariki}: in that case we can simply identify
\begin{equation}\label{id0}
\widetilde{F}(\La) = F(\La_{k_l}) \otimes\cdots\otimes F(\La_{k_1})
\end{equation}
by identifying $\widetilde{M}_\la$ with $M_{\la^{(l)}}\otimes\cdots\otimes
M_{\la^{(1)}}$ for each $\la \in \Par$.
The formulae (\ref{uact1})--(\ref{uact3}) 
describe the natural action of $U_q(\g)$
on this tensor product, so they give
a well-defined $U_q(\g)$-action in the $e=0$ case too.

Recalling the partial order $\geq$ from $\S$\ref{tpo},
the twisted Fock space $\widetilde{F}(\La)$ 
possesses a canonical compatible bar-involution
with the property that
\begin{equation}\label{barinv3}
\overline{\widetilde{M}_\la} = \widetilde{M}_\la + \text{(a $\Z[q,q^{-1}]$-linear combination
of $\widetilde{M}_\mu$'s for $\mu > \la$)}
\end{equation}
for any $\la \in \Par$.
The existence of this bar-involution is established
in \cite{LT1,LT2} in the case that $l = 1$ and $e > 0$ by reinterpreting $\widetilde{F}(\La)$
in that case as a semi-infinite wedge
as in \cite{Stern}.
The construction of the bar-involution in the level one case 
from \cite{LT1,LT2}
was extended to higher levels in the $e > 0$ case
by Uglov; see \cite[Proposition 4.11]{Uglov}.
In the case $e=0$, it is clear from (\ref{id0}) and (\ref{tp})
that the space $\widetilde{F}(\La)$ is just the
same as the space $F(\La)$ but reversing
the order of the underlying
sequence $k_1,\dots,k_l$.
So we get the compatible bar-involution in this case from the same
construction as explained at the beginning of $\S$\ref{tzero};
see also \cite[$\S$2.5]{BKariki}.

At this point one can repeat almost word-for-word the development 
from $\S\S$\ref{expre}--\ref{sqcb}, replacing the Fock space $F(\La)$
with the twisted Fock space $\widetilde{F}(\La)$ 
and using the known
bar-involution from (\ref{barinv3}) in place of the hypothesized
bar-involution from Hypothesis~\ref{assump}. 
All we actually need from this here
is the definition of the {\em dual-canonical basis}
$\{\widetilde{L}_\la\:|\:\la \in \Par\}$ of 
$\widetilde{F}(\La)$, which is defined by
letting $\widetilde{L}_\la$ denote the unique
bar-invariant vector such that
\begin{equation}\label{lllk2}
\widetilde{L}_\la = \widetilde{M}_\la + \text{(a $q\Z[q]$-linear combination of
$\widetilde{L}_\mu$'s with $\mu >\la$}).
\end{equation}
From this, we obtain polynomials $\ddd_{\la,\mu}(q) \in \Z[q]$ such that
\begin{align}\label{pmtilde}
\widetilde{M}_\mu = \sum_{\la \in \Par} \ddd_{\la,\mu}(q) \widetilde{L}_\la.
\end{align}
These have the property that $\ddd_{\la,\mu}(q) = 0$
unless $\la \geq \mu$, $\ddd_{\la,\mu}(q) = 1$ if $\la = \mu$,
and $\ddd_{\la,\mu}(q) \in q \Z[q]$ if $\la > \mu$.

\begin{Remark}\label{remarkabove}
The dual-canonical basis $\{\widetilde{L}_\la\:|\:\la \in \Par\}$
is an upper 
global crystal basis in the sense of \cite{KaG}.
It leads to yet another abstract
crystal structure on the index set $\Par$,
which can be described combinatorially by the same method as
in $\S$\ref{scc}, except that
one needs to start
by listing the addable and removable $i$-nodes of $\la \in \Par$
as $A_1,\dots,A_n$ so that $A_m$ is residue-above $A_{m+1}$
for each $m=1,\dots,n-1$.
Let $\RRPar$ denote the vertex set of the 
connected component of this crystal generated 
by $\varnothing$.
This provides another realization of the abstract crystal
associated to $V(\La)$.
In the case $e=0$, the set 
$\RRPar$ happens to be the same as the set $(\RPar)'$ introduced in 
Remark~\ref{primecrys}.
\end{Remark}

\begin{Remark}\label{remarkbelow}
There is a special case in which the set $\RRPar$ from Remark~\ref{remarkabove} 
has an elementary description.
Suppose that
$\kk_1 \geq \cdots \geq \kk_l$ and either $e=0$ or 
$\kk_l > \kk_1 - e$.
Then $\RRPar$ consists of all $l$-multipartitions
$\la$ such that 
\begin{itemize}
\item[(1)] $\la_a^{(m)}+\kk_m - \kk_{m+1}\geq \la_a^{(m+1)}$
for each $a \geq 1$ and $m=1,\dots,l-1$;
\item[(2)] if $e > 0$ then $\la_a^{(l)}+e+\kk_l-\kk_1 \geq \la_a^{(1)}$
for each $a \geq 1$;
\item[(3)] it is impossible find nodes
$\{A_i \:|\:i \in I\}$ from the bottoms of 
columns of the same length in $\la$
such that $\res\, A_i \equiv i \pmod{e}$ for each $i \in I$.
\end{itemize}
This follows from \cite[(2.53)]{BKariki} in the case $e = 0$,
and it is a reformulation of
 a result from \cite{FLOTW} in the case $e > 0$.
\end{Remark}

\begin{Remark}\label{cata}
In the $e=0$ case, the twisted 
Fock space $\widetilde{F}(\La)$
can be categorified by means of 
the opposite parabolic category $\mathcal O$
to the one mentioned in Remark~\ref{blox}.
By Arkhipov-Soergel reciprocity,
this categorification of $\widetilde{F}(\La)$
is the Ringel dual of the categorification of 
$F(\La)$; see \cite[$\S$4.3]{BKariki}.
When $e > 0$ there should also be a highest weight category
categorifying the twisted Fock space $\widetilde{F}(\La)$
arising from rational Cherednik algebras, 
although the picture here 
is not yet complete;
see \cite[$\S$6.8]{RoC} and also \cite[$\S$8]{VV} which develops another
approach in terms of affine parabolic category $\mathcal O$.
Under a stability hypothesis similar to the one in Proposition~\ref{ythm}
below, this category is known to be equivalent to the 
one mentioned in 
Remark~\ref{blox} arising from cyclotomic $\xi$-Schur algebras;
see \cite[Theorem 6.8]{RoC}.
\end{Remark}

\subsection{\boldmath Construction of the bar-involution}\label{construct}
In this subsection we assume that $e > 0$ and explain how to
construct a bar-involution on $F(\La)$ as in Hypothesis~\ref{assump}.
To do this we exploit the following stability result of Yvonne.

\begin{Proposition}[{\cite[Theorem 5.2]{Y2}}]\label{ythm}
Let $k_1,\dots,k_l$ be fixed as in (\ref{note}).
For each $\alpha \in Q_+$, there exists an integer $N_\alpha \gg 0$ 
such that
the transition matrix 
$(\ddd_{\la,\mu}(q))_{\la,\mu \in \Par_\alpha}$ 
is the same matrix for every multicharge
$(\kk_1,\dots,\kk_l) \in \Z^l$ with 
$\kk_1 \equiv k_1,\dots,\kk_l\equiv k_l\pmod{e}$ and 
$\kk_m-\kk_{m+1} \geq N_\alpha$ for $m=1,\dots,l-1$.
\end{Proposition}

\begin{Remark} 
Conjecturally, one can take
$N_\alpha := \height(\alpha)$;
see \cite[Remark 5.3]{Y2}.
\end{Remark}

Now define
$d_{\la,\mu}(q) \in \Z[q]$ for any $\la,\mu \in \Par$ as follows.
If $\cont(\la) \neq \cont(\mu)$ we set $d_{\la,\mu}(q) := 0$.
Otherwise, if $\cont(\la) = \cont(\mu) = \alpha$ for some $\alpha \in Q_+$,
pick $(\kk_1,\dots,\kk_l)\in\Z^l$ so that
$\kk_1 \equiv k_1,\dots,\kk_l\equiv k_l\pmod{e}$ and $\kk_m - \kk_{m+1} \geq N_\alpha$
for $m=1,\dots,l-1$, and then
set $d_{\la,\mu}(q) := \ddd_{\la,\mu}(q)$, i.e. the polynomial defined as in (\ref{pmtilde})
for this choice of multicharge. Proposition~\ref{ythm} implies that $d_{\la,\mu}(q)$
is well defined independent of the particular choice of $(\kk_1,\dots,\kk_l)$.
Finally let $(p_{\la,\mu}(-q))_{\la,\mu \in \Par}$ be the inverse of the
matrix $(d_{\la,\mu}(q))_{\la,\mu \in \Par}$ and define
an anti-linear endomorphism $-$ of $F(\La)$ by setting
\begin{equation}\label{final}
\overline{M_\mu} := \sum_{\kappa, \la \in \Par}
p_{\kappa,\la}(-q) d_{\la,\mu}(q^{-1}) M_\kappa
\end{equation}
for each $\mu \in \Par$.
The following theorem shows that this is a compatible bar-involution
on $F(\La)$ as in Hypothesis~\ref{assump} (taking the
order $\preceq$ there to be $\lexeq$).

\begin{Theorem}
The map (\ref{final}) is a compatible bar-involution on $F(\La)$
with the following property for every $\la \in \Par$:
\begin{equation}\label{semi}
\overline{M_\la} = M_\la + (\text{a $\Z[q,q^{-1}]$-linear combination
of $M_\mu$'s for $\mu \lex \la$}).
\end{equation}
\end{Theorem}

\begin{proof}
Fix some $d \geq 0$ and let $F(\La)_{\leq d}$ denote
the subspace of $F(\La)$
spanned by 
all $M_\la$'s for $\la \in \Par$ with $|\la| \leq d$.
We claim that the restriction of $-$ to $F(\La)_{\leq d}$
is an involution with the property (\ref{semi})
for all $\la$ with $|\la| \leq d$, and moreover that 
$\overline{E_i v} = E_i \overline{v}$ 
and $\overline{F_i v} = F_i \overline{v}$
for all $v\in F(\La)_{\leq (d-1)}$ and $i \in I$.
The theorem follows from claim by 
letting $d \rightarrow \infty$.

To prove the claim, 
set $N := \max\{d+e,N_\alpha\:|\:\alpha \in Q_+, \height(\alpha)\leq d\}$.
Choose the multicharge so that $\kk_1\equiv k_1,\dots,\kk_l \equiv k_l\pmod{e}$
and $\kk_m - \kk_{m+1} \geq N$ for $m=1,\dots,l$.
Defining $\widetilde{F}(\La)$ with this choice of multicharge, we define
a vector space isomorphism
$$
\iota:F(\La) \stackrel{\sim}{\rightarrow} \widetilde{F}(\La),
\qquad
M_\la \mapsto \widetilde{M}_\la.
$$
Using Lemma~\ref{sdsd}, (\ref{act1}) and (\ref{uact1}),
we see that $\iota(E_i v) = E_i \iota(v)$
and $\iota(F_i v) = F_i \iota(v)$ for all $v \in F(\La)_{\leq(d-1)}$.
By the definition (\ref{final}), the bar-involution on $F(\La)$ is the unique
anti-linear map fixing the vectors $\sum_\la p_{\la,\mu}(-q) M_\la$
for all $\mu \in \Par$.
Moreover for $\la,\mu \in \Par$ with $|\mu| \leq d$, we have that
$d_{\la,\mu}(q) = \ddd_{\la,\mu}(q)$, so recalling (\ref{pmtilde})
we see that $\iota$ maps $\sum_\la p_{\la,\mu}(-q) M_\la$
to the bar-invariant vector
$\widetilde{L}_\mu$
for $|\mu| \leq d$.
This shows that $\iota(\overline{v}) = \overline{\iota(v)}$
for all $v \in F(\La)_{\leq d}$.
Putting these things together gives that
$-$ is an involution on $F(\La)_{\leq d}$ such that
$\overline{E_i v} = E_i \overline{v}$ 
and $\overline{F_i v} = F_i \overline{v}$
for all $v\in F(\La)_{\leq(d-1)}$ and $i \in I$.
Moreover from (\ref{lllk2}) we get that
$$
\overline{M_\la} = M_\la + (\text{a $\Z[q,q^{-1}]$-linear combination
of $M_\mu$'s for $\mu > \la$})
$$
for $|\la| \leq d$.
Finally an application of Lemma~\ref{weaker} gives (\ref{semi})
for $|\la| \leq d$. This establishes the claim.
\end{proof}

\begin{Remark}
It would be interesting to find a direct construction of the
bar-involution on $F(\La)$ in the $e > 0$ case
by-passing twisted Fock space;
such a construction should also produce a natural choice for the partial
order $\preceq$.
\end{Remark}

\section{Graded branching rules and categorification of $V(\La)$}

Continue with $F$ denoting an algebraically closed field, but assume
also now that $\xi \in F^\times$ is an invertible element and take the
integer $e$ 
to be the smallest positive integer such that $1+\xi+\cdots+\xi^{e-1} = 
0$, setting $e := 0$
if no such integer exists. All other notation is the same as in the previous
sections for this choice of $e$; in particular, we have fixed  $\La$ 
as in (\ref{note}).

\subsection{\boldmath The algebra $R^\La_\alpha$}\label{slump}
Following \cite[$\S$3.4]{KL1}, we let
\begin{equation}
R^\La_\alpha := R_\alpha \big / \big\langle y_1^{(\La,\alpha_{i_1})} e(\bi)\:\big|\:\bi \in I^\alpha\big \rangle.
\end{equation}
We use the same notation for elements of $R^\La_\alpha$ as in $R_\alpha$,
relying on context to distinguish which we mean.

For any $i \in I$, the 
embedding $R_\alpha = R_\alpha \otimes 1 
\hookrightarrow R_{\alpha,\alpha_i}\hookrightarrow R_{\alpha+\alpha_i}$
factors through the quotients to induce a (not necessarily injective)
graded algebra homorphism
\begin{equation}\label{iota1}
\iota_{\alpha,\alpha_i}:R_\alpha^\La \rightarrow R_{\alpha+\alpha_i}^\La.
\end{equation}
This maps the identity element of $R_\alpha^\La$ to the idempotent
$e_{\alpha,\alpha_i} \in R_{\alpha+\alpha_i}^\La$.

\subsection{\boldmath The algebra $H^\La_\alpha$}\label{sah}
Let $H_d$ denote the affine Hecke algebra
associated to the symmetric group $\Sigma_d$
on generators
$\{X_1^{\pm 1},\dots,X_d^{\pm 1}\} \cup \{T_1,\dots,T_{d-1}\}$
if $\xi \neq 1$, or its degenerate analogue on generators
$\{x_1,\dots,x_d\}\cup\{s_1,\dots,s_{d-1}\}$
if $\xi = 1$. For the full relations, which are quite standard, we refer
the reader to \cite{BKyoung}, noting here just that
\begin{equation*}
\left\{
\begin{array}{lll}
T_r^2 = (\xi-1) T_r + \xi,
&T_r X_r T_r = \xi X_{r+1}
&\qquad\text{if $\xi \neq 1$,}\\
\,s_r^2 = 1, 
&\;s_r x_{r+1} = x_r s_r+1
&\qquad\text{if $\xi = 1$.}
\end{array}\right.
\end{equation*}
Then we consider the cyclotomic quotient
\begin{equation}\label{ECHA}
H_d^\La := \left\{
\begin{array}{ll}
H_d \Big/ \big\langle \,\textstyle\prod_{i\in I}(X_1-\xi^i)^{(\La,\al_i)}\,\big\rangle
&\text{if $\xi \neq 1$,}\\
H_d \Big/ \big\langle \,\textstyle\prod_{i\in I}(x_1-i)^{(\La,\al_i)}\,\big\rangle
&\text{if $\xi = 1$}.
\end{array}\right.
\end{equation}
We refer to this algebra simply as 
the {\em cyclotomic Hecke algebra} if $\xi\neq 1$
and the {\em degenerate cyclotomic Hecke algebra} if $\xi =1$.

There is a natural system 
$\{e(\bi)\:|\:\bi \in I^d\}$
of mutually orthogonal idempotents in $H_d^\La$ called
{\em weight idempotents}; see \cite{BKyoung}.
These are characterized uniquely by 
the property that 
$e(\bi) M = M_\bi$ for any finite dimensional $H_d^\La$-module $M$,
where
\begin{equation}
M_\bi := \left\{\begin{array}{ll}
\{v \in M\:|\:(X_r -\xi^{i_r})^N v = 0 \text{ for $N \gg 0$}\}
&\text{if $\xi \neq 1$,}\\
\{v \in M\:|\:(x_r -i_r)^N v = 0 \text{ for $N \gg 0$}\}
&\text{if $\xi = 1$.}
\end{array}\right.
\end{equation}
Note all but finitely many of the $e(\bi)$'s are zero, and their sum
is the identity element in $H_d^\La$.

Given $\al \in Q_+$ of height $d$, we set
\begin{equation}
e_\alpha := \sum_{\bi \in I^\alpha} e(\bi) \in H^\La_d.
\end{equation}
As a consequence of \cite{LM} or \cite[Theorem~1]{cyclo},
$e_{\alpha}$ is either zero or it is a primitive central idempotent
in $H^\La_d$.
Hence the algebra
\begin{equation}\label{fe}
H^\La_\alpha := e_{\alpha} H^\La_d
\end{equation}
is either zero or it is a single {\em block}
of the algebra $H^\La_d$, and we have that
\begin{equation}\label{blockdec}
H^\La_d = \bigoplus_{\alpha \in Q_+, \height(\alpha)=d} H^\La_\alpha
\end{equation}
as a direct sum of algebras.
For $h\in H_d^\La$, 
we still write $h$ for  the projection $e_{\al}h\in H_\al^\La$. 

The natural embedding of $H_d$ into $H_{d+1}$
factors through the quotients to induce an embedding
of $H_d^\La$ into $H_{d+1}^\La$.
Composing this on the right with the inclusion
$H_\alpha^\La \hookrightarrow H_d^\La$ and then on the left with
multiplication by the idempotent $e_{\alpha,\alpha_i}$,
we obtain a non-unital algebra homomorphism
\begin{equation}\label{iota2}
\iota_{\alpha,\alpha_i}: H^\La_\alpha \rightarrow H^\La_{\alpha+\alpha_i}.
\end{equation}
Just like in (\ref{iota1}), this maps the identity element of $H^\La_\alpha$
to the idempotent 
\begin{equation}
e_{\alpha,\alpha_i} := \sum_{\bi \in I^{\alpha+\alpha_i},
i_{d+1}=i} e(\bi).
\end{equation}

\subsection{The isomorphism theorem}
According to the main theorem of \cite{BKyoung}, the
cyclotomic algebras $R^\La_\alpha$ and
$H^\La_\alpha$ are isomorphic.
Although not used explicitly here, we note that a closely
related result for the affine algebras 
has been obtained independently by Rouqiuer
in \cite[$\S$3.2.6]{Ro}.

\begin{Theorem}[{\cite{BKyoung}}]\label{ISO}
For $\alpha \in Q_+$ of height $d$,
there is an algebra isomorphism
$\rho:R^\La_\alpha \stackrel{\sim}{\rightarrow} H^\La_\alpha$
such that 
\begin{align*}
e(\bi)&\mapsto e(\bi),\\
y_r e(\bi) &\mapsto \left\{\begin{array}{ll}
(1-\xi^{-i_r} X_r)e(\bi)&\text{if $\xi \neq 1$,}\\
(x_r-i_r) e(\bi)&\text{if $\xi=1$,}
\end{array}\right.
\end{align*}
for each $r=1,\dots,d$ and $\bi \in I^\alpha$.
Moreover, the following diagram commutes 
for all $\alpha \in Q_+$ and $i \in I$:
\begin{equation}\label{cds}
\begin{CD}
R^\La_{\alpha}&@>\iota_{\alpha,\alpha_i}>>&R^\La_{\alpha+\alpha_i}\\
@V\rho VV&&@VV\rho V\\
H^\La_\alpha&@>>\iota_{\alpha,\alpha_i}>&H^\La_{\alpha+\alpha_i}.
\end{CD}
\end{equation}
\end{Theorem}

\begin{Remark}\rm 
In \cite{BKyoung}, one can also find
formulae for the images of the generators $\psi_r e(\bi)$, but we do not
need to know these explicitly here. 
\end{Remark}

Henceforth, we will simply {\em identify} the algebras $R^\La_\alpha$
and $H^\La_\alpha$ via the isomorphism $\rho$ from Theorem~\ref{ISO}.
Despite the fact that $R^\La_\alpha = H^\La_\alpha$ now,
we will usually talk in terms of $R^\La_\alpha$
when discussing {\em graded} representation theory and $H^\La_\alpha$
when discussing {\em ungraded} representation theory.
For instance, as in $\S$\ref{SSGr},
for a graded $R^\La_\alpha$-module $M$ we write
$\underline{M}$ for the ungraded $H^\La_\alpha$-module
obtained from $M$ by forgetting the grading.

\subsection{\boldmath $i$-Induction and $i$-restriction}\label{sir}
For $i \in I$ and $\alpha \in Q_+$, let
$\e_i$ and $\f_i$ be the functors
\begin{align}\label{ei}
\e_{i}&:=e_{\alpha,\alpha_i} H^\La_{\alpha+\alpha_i} \otimes_{H^\La_{\alpha+\alpha_i}} ?
:\mod{H^\La_{\alpha+\alpha_i}} \rightarrow \mod{H^\La_{\alpha}},\\
\f_{i}&:=\:\:\:\:\:H^\La_{\alpha+\alpha_i} e_{\alpha,\alpha_i} \otimes_{H^\La_{\alpha}} ?
:\mod{H^\La_\alpha} \rightarrow \mod{H^\La_{\alpha+\alpha_i}},\label{fi}
\end{align}
viewing $e_{\alpha,\alpha_i} H^\La_{\alpha+\alpha_i}$
(resp.\ $H^\La_{\alpha+\alpha_i} e_{\alpha,\alpha_i}$)
as a left (resp.\ right) $H^\La_\alpha$-module
via the homomorphism (\ref{iota2}).
The functor $\e_i$ is particularly simple to understand:
it is just multiplication by the idempotent
$e_{\alpha,\alpha_i}$ followed by restriction to $H^\La_\alpha$
via the homomorphism $\iota_{\alpha,\alpha_i}$.
We use the same notation $\e_i$ and $\f_i$
for the direct sums of these functors over all $\alpha \in Q_+$.
They are exactly Robinson's {\em $i$-restriction}
and {\em $i$-induction} functors
as in \cite[$\S$13.6]{Abook}
or \cite[(8.4), (8.6)]{Kbook}.
In particular, it is known that $\e_i$ and $\f_i$ are biadjoint,
hence both are exact and send projectives to projectives.
They obviously both send finite dimensional (resp.\ finitely generated)
modules to finite dimensional (resp.\ finitely generated) modules.

Similarly define functors
$E_i$ and $F_i$ by setting
\begin{align}\label{gre}
E_{i}&:=\hspace{20.5mm}e_{\alpha,\alpha_i} R^\La_{\alpha+\alpha_i} \otimes_{R^\La_{\alpha+\alpha_i}} ?
:\Mod{R^\La_{\alpha+\alpha_i}} \rightarrow \Mod{R^\La_{\alpha}},\\
F_{i}&:=
R^\La_{\alpha+\alpha_i} e_{\alpha,\alpha_i} \otimes_{R^\La_{\alpha}} ? \langle 1\!-\!(\La\!-\!\alpha,\alpha_i) \rangle
:\Mod{R^\La_\alpha} \rightarrow \Mod{R^\La_{\alpha+\alpha_i}},\label{grf}
\end{align}
interpreting the tensor products via (\ref{iota1}), then taking
the direct sums over all $\alpha \in Q_+$.
By (\ref{cds}), these are graded versions of $\e_i$ and $\f_i$ in the sense that 
\begin{equation}
\underline{E_i (M)} \cong \e_i (\underline{M}),\qquad
\underline{F_i (M)} \cong \f_i (\underline{M})
\end{equation}
for any graded $R^\La_\alpha$-module $M$.
In particular, we deduce from this that $E_i$ and $F_i$ are both
exact, and send finite dimensional (resp.\ finitely generated projective)
modules to finite dimensional (resp.\ finitely generated
projective) modules, since we already know that for $\e_i$ and $\f_i$.

Also define a functor $K_i$ by letting
\begin{equation}
K_i:\Mod{R^\La_\alpha} \rightarrow \Mod{R^\La_\alpha}
\end{equation}
denote the degree shift functor 
$M\mapsto M \langle (\La-\alpha,\alpha_i) \rangle$.
If we use this functor to cancel out the degree shifts in  (\ref{grf}), we
see that
\begin{equation}\label{ath}
F_i K_i \langle -1 \rangle \cong
R_{\alpha+\alpha_i}^\La e_{\alpha,\alpha_i} \otimes_{R_\alpha^\La} ?.
\end{equation}
Combining this with adjointness of tensor and hom we deduce:

\begin{Lemma}\label{cad}
There is a canonical adjunction making
$(F_i K_i \langle-1\rangle, E_i)$ into an adjoint pair.
\end{Lemma}

There is an equivalent way to describe the
functors $E_i$ and $F_i$ which relates them to
the functors $\theta_i^*$ and $\theta_i$ from (\ref{fri1})--(\ref{fri2}).
To formulate this, we first introduce the
{\em inflation} and {\em truncation} functors
\begin{equation}
\infl:\Mod{R^\La_\alpha} \rightarrow \Mod{R_\alpha}
\qquad
\pr:\Mod{R_\alpha} \rightarrow \Mod{R^\La_\alpha}.
\end{equation}
So for $M \in \Mod{R^\La_\alpha}$,
we write $\infl\, M$ for its pull-back through the natural
surjection $R_\alpha \twoheadrightarrow R^\La_\alpha$,
and for $N \in \Mod{R_\alpha}$ we write
$\pr\, N$ for $R^\La_\alpha \otimes_{R_\alpha} N$, which is the
largest graded quotient of $N$ that factors through to $R^\La_\alpha$.
Note $\pr$ depends implicitly on the fixed choice of $\La$,
but we omit it from our notation since this should be clear from 
context.
We obviously have that
\begin{equation}\label{ididid}
\pr \circ \infl  = \Id.
\end{equation}
Observe also that $(\pr, \infl)$ is an adjoint pair
in a canonical way.
Hence, $\pr$ sends projective modules 
to projective modules.
It follows easily that $\infl$ and $\pr$ restrict to
functors
\begin{equation}\label{prprpr}
\infl:\Rep{R^\La_\alpha} \rightarrow \Rep{R_\alpha},\qquad
\pr:\Proj{R_\alpha} \rightarrow \Proj{R^\La_\alpha}.
\end{equation}

\begin{Lemma}\label{sun}
There are canonical isomorphisms of functors
$E_i \cong \pr \circ \theta_i^* \circ \infl$
and $F_i K_i \langle -1 \rangle 
\cong \pr \circ \theta_i \circ \infl.$
\end{Lemma}

\begin{proof}
For $E_i$, note that both $\infl \circ E_i$ and $\theta_i^* \circ
\infl$ are defined on $M \in \Mod{R^\La_{\alpha+\alpha_i}}$
by multiplying by the idempotent $e_{\alpha,\alpha_i}$.
Hence $\infl \circ E_i \cong \theta_i^* \circ \infl$.
Using also (\ref{ididid})
this implies that $E_i \cong \pr \circ \theta_i^* \circ \infl$.

For $F_i$, there is a canonical adjunction making 
$(\pr \circ \theta_i, \theta_i^* \circ \infl)$ into an adjoint pair.
Hence for $M \in \Mod{R^\La_\alpha}$
and $N \in \Mod{R^\La_{\alpha+\alpha_i}}$ we have natural 
isomorphisms
\begin{align*}
\Hom_{R^\La_{\alpha+\alpha_i}}(\pr \ \theta_i \ \infl (M), N)
&\cong
\Hom_{R_\alpha}(\infl\,  M , \theta_i^* \ \infl(N))\\
&\cong
\Hom_{R_\alpha}(\infl\, M, \infl \ E_i (N))
= \Hom_{R^\La_\alpha}(M, E_i N).
\end{align*}
This establishes that $\pr \circ \theta_i \circ \infl$
is left adjoint to $E_i$.
Hence $\pr \circ \theta_i \circ \infl \cong
F_i K_i \langle -1 \rangle$ by 
Lemma~\ref{cad} and unicity of adjoints.
\end{proof}

\subsection{Cyclotomic duality}\label{scd}
The anti-automorphism $*:R_\alpha\rightarrow R_\alpha$
from (\ref{star}) descends to the quotient $R^\La_\alpha$,
yielding a graded anti-automorphism
$*:R^\La_\alpha \rightarrow R^\La_\alpha$. Using this we
can define a duality $\circledast$ on $\Rep{R^\La_\alpha}$ 
in the same way as the duality $\circledast$
was defined on $\Rep{R_\alpha}$ in $\S$\ref{ssd}. 
It is then clear that 
$\circledast$ commutes with inflation, i.e.
\begin{equation}
\infl \circ \circledast \cong \circledast \circ \infl
\end{equation}
as functors from $\Rep{R^\La_\alpha}$ to $\Rep{R_\alpha}$.
The following lemma follows by the same argument as (\ref{commd}):

\begin{Lemma}\label{commutes}
There is an isomorphism $\circledast \circ E_i
\cong E_i \circ \circledast.$
\end{Lemma}

In view of the next lemma, the duality 
$\circledast$ on $\Rep{R^\La_\alpha}$ restricts
to give a well-defined duality
$\circledast$ on the subcategory $\Proj{R^\La_\alpha}$ too.

\begin{Lemma}\label{pins}
For $P \in \Proj{R^\La_\alpha}$, the dual
$P^\circledast$ is a graded projective module.
\end{Lemma}

\begin{proof}
It suffices to show that $\underline{P^\circledast}$
is a projective $H^\La_\alpha$-module. This follows because 
$H^\La_\alpha$ is a symmetric
algebra, so its injective modules are projective; see \cite{MM} 
or \cite[Theorem A.2]{BKschur} in the
degenerate case.
\end{proof}

Although not used explicitly here, 
we note for completeness 
that there is another duality $\#$ on $\Rep{R^\La_\alpha}$
mapping $M$ to $M^\# := \HOM_{R^\La_\alpha}(M, R^\La_\alpha)$
with action defined by $(xf)(p) = f(p)x^*$.
The fact that this is exact (hence a duality) follows because 
$R^\La_\alpha$
is injective by Lemma~\ref{pins}.
The duality $\#$ obviously restricts to a
well-defined duality $\#$ on $\Proj{R^\La_\alpha}$.
Recalling the duality $\#$ on $\Proj{R_\alpha}$
from $\S$\ref{ssd}, it is also clear that 
$\#$ commutes with truncation, i.e. 
\begin{equation}\label{comp}
\pr \circ \# \cong \# \circ \pr
\end{equation}
as functors from
$\Proj{R_\alpha}$ to $\Proj{R^\La_\alpha}$.

\begin{Remark}\rm
We conjecture that $R^\La_\alpha$
is a graded symmetric algebra
in the sense that 
it possesses a homogeneous symmetrizing
form $\tau:R^\La_\alpha \rightarrow F$
of degree $-2\defect(\alpha)$.
By general principles this would imply that
there is an isomorphism of functors
$\# \cong \langle 2\defect(\alpha) \rangle\circ\circledast$; 
see e.g. \cite[Theorem 3.1]{Ric}.
Given this, one could deduce that
$F_i$ commutes with $\circledast$
(because an argument similar to the proof of (\ref{snow}) 
shows already that
$F_i K_i \langle -1\rangle$ commutes with $\#$).
The latter statement can be proved indirectly
by appealing to the formalism of
Remark~\ref{fancyrem} below and \cite[Theorem 5.16]{Ro}.
\end{Remark}

\subsection{Cyclotomic divided powers}\label{sdivp}
Lemma~\ref{sun} also
makes it clear how to define
divided powers $E_i^{(n)}$ and $F_i^{(n)}$ of the functors
$E_i$ and $F_i$. For $n \geq 1$, set
\begin{align*}
E_i^{(n)} &:=\hspace{27.3mm} \pr \circ (\theta_i^*)^{(n)} \circ \infl:
\Mod{R^\La_{\alpha+n\alpha_i}} \rightarrow \Mod{R^\La_\alpha},\\
F_i^{(n)} &:= \pr \circ \theta_i^{(n)} \circ \infl 
\langle 
 n^2 - n(\La-\alpha,\alpha_i) \rangle
:\Mod{R^\La_\alpha} \rightarrow \Mod{R^\La_{\alpha+n\alpha_i}},
\end{align*}
recalling (\ref{div1})--(\ref{div2}).
Again we use the same notation $E_i^{(n)}$ and $F_i^{(n)}$ for the 
direct sums of these functors over all $\alpha \in Q_+$.

\begin{Lemma}\label{divp}
There are isomorphisms 
$E_i^n \cong [n]! \cdot E_i^{(n)}$
and
$F_i^n \cong [n!] \cdot F_i^{(n)}$.
Hence
$E_i^{(n)}$ and $F_i^{(n)}$ are exact,
and they send finite dimensional (resp.\ finitely generated projective) modules
to finite dimensional (resp.\ finitely generated projective) modules.
\end{Lemma}

\begin{proof}
This follows by Lemma~\ref{sun}, (\ref{ididid}) and (\ref{dividedform}).
\end{proof}

\subsection{Ungraded irreducible representations and branching rules}\label{glab}
It is time to recall Grojnowski's classification \cite{G} of finite
dimensional irreducible $H^\La_\alpha$-modules
in terms of the 
crystal associated to the 
highest weight module $V(\La)$,
which was inspired by the modular branching rules of \cite{KBrIII}.
We will use the explicit realization 
$(\RPar,
\tilde e_i, \tilde f_i, \eps_i, \phi_i, \wt)$
of the crystal from $\S$\ref{scc}.

For $\la \in \RPar_\al$, we define an
$H^\La_\al$-module $\underline{D}(\la)$ recursively as follows.
To start with, let
$\underline{D}(\varnothing)$ denote the trivial representation 
of $H^\La_0 \cong F$.
Now suppose that we are given 
$\la \in \RPar_\al$ for $\height(\alpha) > 0$.
Choose any $i \in I$ such that $\eps_i(\la) \neq 0$
and define
\begin{equation}\label{irreps}
\underline{D}(\la) := \soc (f_i \underline{D}(\tilde e_i \la)),
\end{equation}
where $\underline{D}(\tilde e_i \la)$ is the recursively defined
$H^\La_{\alpha-\alpha_i}$-module. 
It is known that $\underline{D}(\la)$ does not depend
up to isomorphism on the particular choice of $i$.
Moreover:

\begin{Theorem}[Grojnowski]\label{GTM}
The modules
$\{\underline{D}(\la)\:|\:\la \in \RPar_\alpha\}$
give a complete set of pairwise non-isomorphic irreducible 
$H^\La_\alpha$-modules.
Moreover, the following hold for any $i \in I$ and $\la \in 
\RPar_\al$:
\begin{itemize}
\item[(1)] $e_i \underline D(\la)$ is non-zero if and only if
$\eps_i(\la) \neq 0$, in which case
$e_i \underline D(\la)$ has irreducible socle and head both isomorphic
to $\underline D(\tilde e_i \la)$.
\item[(2)]
 $f_i \underline D(\la)$ is non-zero if and only if
$\phi_i(\la) \neq 0$, in which case
$f_i \underline D(\la)$ has irreducible socle and head both isomorphic
to $\underline D(\tilde f_i \la)$.
\item[(3)] In the Grothendieck group we have that
\begin{align*}
[e_i \underline D(\la)] 
&= 
\eps_i(\la) [\underline D(\tilde e_i \la)]
+ \sum_{\substack{\mu \in \RPar_{\al-\al_i}\\
\eps_i(\mu) < \eps_i(\la)-1}} u_{\mu,\la;i}(1) 
[\underline D(\mu)],\\
[f_i \underline D(\la)] 
&= 
\phi_i(\la) [\underline D(\tilde f_i \la)]
+ \sum_{\substack{\mu \in \RPar_{\al+\al_i} \\
\phi_i(\mu) < \phi_i(\la)-1}} v_{\mu,\la;i}(1) 
[\underline D(\mu)],
\end{align*}
for some coefficients $u_{\mu,\la;i}(1), v_{\mu,\la;i}(1) \in \Z_{\geq 0}$.
(The first term on the right hand side of these formulae
should be interpreted as zero if $\eps_i(\la) = 0$ (resp.\
$\phi_i(\la) = 0$).)
\item[(4)] There are algebra isomorphisms
\begin{align*}
f:F[x] / (x^{\eps_i(\la)}) &\stackrel{\sim}{\rightarrow}
\underline{\End}_{H^\La_{\al-\al_i}}(e_i \underline D(\la)),\\
g:F[x] / (x^{\phi_i(\la)}) &\stackrel{\sim}{\rightarrow}
\underline{\End}_{H^\La_{\al+\al_i}}(f_i \underline D(\la)).
\end{align*}
\end{itemize}
\end{Theorem}

\begin{proof}
See \cite{G} or \cite{Kbook} which gives an exposition of
Grojnowski's methods
in the degenerate case.
More precisely, the first statement is \cite[Theorem 10.3.4]{Kbook},
then \cite[Theorem 8.3.2]{Kbook} gives (1) and (2),
and \cite[Theorems 5.5.1 and 8.5.9]{Kbook} gives (3) and (4).
\end{proof}

\begin{Remark}\rm
The results of Theorem~\ref{GTM}(1)--(4) are extended to
the divided powers
$e_i^{(n)}$ and $f_i^{(n)}$ 
of the functors $e_i$ and $f_i$ in \cite[Proposition 5.20]{CR}.
\end{Remark}

\subsection{Graded irreducible representations and branching rules}\label{sb}
Now we lift the parametrization of irreducible modules 
from $H^\La_\alpha$ to $R^\La_\alpha$
using Theorem~\ref{ISO}.

\begin{Theorem}\label{clas} For each $\la \in \RPar_\alpha$, there exists
an irreducible graded $R^\La_\alpha$-module $D(\la)$
such that
\begin{itemize}
\item[(1)] $D(\la)^\circledast \cong D(\la)$;
\item[(2)] $\underline{D(\la)} \cong \underline{D}(\la)$
as an $H^\La_\alpha$-module.
\end{itemize}
The module $D(\la)$ is determined uniquely up to isomorphism by these
two conditions. Moreover, the modules
$\{D(\la)\langle m \rangle\:|\:\la \in \RPar_\alpha, m \in \Z\}$ 
give a complete set of pairwise non-isomorphic irreducible 
graded $R^\La_\alpha$-modules.
\end{Theorem}

\begin{proof}
By Lemma~\ref{no3} and Theorem~\ref{ISO},
there exists an irreducible graded $R^\La_\alpha$-module
$D(\la)$ satisfying (2), but this is only unique
up to isomorphism and grading shift. The fact that $D(\la)$
can be chosen so that it also satisfies (1) is explained at the end of
\cite[$\S$3.2]{KL1}. This pins down the choice of grading shift and 
makes $D(\la)$ unique up to isomorphism. 
The final statement
follows from Theorem~\ref{GTM} and Lemma~\ref{no2}.
\end{proof}

The following theorem lifts the remaining parts of Theorem~\ref{GTM}
to the graded setting.

\begin{Theorem}\label{gg}
For any $\la \in \RPar_\al$ and $i \in I$, we have:
\begin{itemize}
\item[(1)] $E_i D(\la)$ is non-zero if and only if
$\eps_i(\la) \neq 0$, in which case
$E_i D(\la)$ has irreducible 
socle  isomorphic
to $D(\tilde e_i \la)\langle \eps_i(\la)-1\rangle$
and head isomorphic to
$D(\tilde e_i \la)\langle 1-\eps_i(\la)\rangle$.
\item[(2)]
 $F_i D(\la)$ is non-zero if and only if
$\phi_i(\la) \neq 0$, in which case
$F_i D(\la)$ has irreducible 
socle isomorphic to
 $D(\tilde f_i \la) \langle \phi_i(\la)-1\rangle$
and
head isomorphic to
 $D(\tilde f_i \la) \langle 1-\phi_i(\la)\rangle$.
\item[(3)] In the Grothendieck group we have that
\begin{align*}
[E_i D(\la)] 
&= 
[\eps_i(\la)] \, [D(\tilde e_i \la)]
+ \sum_{\substack{\mu \in \RPar_{\al-\al_i}\\
\eps_i(\mu) < \eps_i(\la)-1}} u_{\mu,\la;i}(q)
[D(\mu)],\\
[F_i D(\la)] 
&= 
[\phi_i(\la)] \, [D(\tilde f_i \la)]
+ \sum_{\substack{\mu \in \RPar_{\al+\al_i} \\
\phi_i(\mu) < \phi_i(\la)-1}} v_{\mu,\la;i}(q) 
[D(\mu)],
\end{align*}
for some $u_{\mu,\la;i}(q), v_{\mu,\la;i}(q) \in 
\Z[q,q^{-1}]$ with non-negative coefficients.
(The first term on the right hand side of these formulae
should be interpreted as zero if $\eps_i(\la) = 0$ (resp.\
$\phi_i(\la) = 0$).)
\item[(4)] Viewing $F[x]$ as a graded algebra
by putting $x$ in degree 2, 
there are graded algebra isomorphisms
\begin{align*}
f:F[x] / (x^{\eps_i(\la)}) &\stackrel{\sim}{\rightarrow}
\END_{R^\La_{\al-\al_i}}(E_i D(\la)),\\
g:F[x] / (x^{\phi_i(\la)}) &\stackrel{\sim}{\rightarrow}
\END_{R^\La_{\al+\al_i}}(F_i D(\la)).
\end{align*}
\end{itemize}
\end{Theorem}

\begin{proof}
We first consider (4).
Let $d := \height(\alpha)$.
By Theorem~\ref{GTM} and (\ref{Ehom}), we have
an isomorphism $f:F[x] / (x^{\eps_i(\la)})
\stackrel{\sim}{\rightarrow}\END_{R^\La_{\alpha-\alpha_i}}(E_i D(\la))$,
but we do not know yet that this is an isomorphism of graded algebras.
For this, we go back into the proof of Theorem~\ref{GTM}(4)
to find that the map $f$ sends $x$ to the endomorphism of
$E_i D(\la) = e_{\alpha-\alpha_i,\alpha_i} D(\la)$ defined
by multiplication by $y_d$ (which centralizes
elements in the image of $\iota_{\alpha-\alpha_i,\alpha_i}$). 
Since $\deg(y_d) = 2$ this shows that $f$ is indeed
an isomorphism of graded algebras.
Similarly the isomorphism
$g:F[x] / (x^{\phi_i(\la)}) \stackrel{\sim}{\rightarrow}
\End_{R^\La_{\alpha+\alpha_i}}(F_i D(\la))$ from
Theorem~\ref{GTM}(4) maps $x$ to the endomorphism
of $F_i D(\la)$
defined by multiplication by the central element
$y_1+\cdots+y_{d+1} \in R_{\alpha+\alpha_i}^{\La}$.
So this is an isomorphism of graded algebras too.
This completes the proof of (4).

Now consider (1). We know already that
$E_i D(\la)$ is non-zero if and only if $\eps_i(\la) \neq 0$
by Theorem~\ref{GTM}(1), and moreover, assuming this is the case, 
the head (resp.\ socle) of 
$E_i D(\la)$ must be isomorphic to
$D(\tilde e_i \la) \langle m \rangle$ 
(resp.\ 
$D(\tilde e_i \la)\langle n \rangle$) 
for some $m,n \in \Z$.
Applying (4), we define a filtration
$$
\{0\} = M_{\eps_i(\la)} \subset \cdots \subset M_1\subset M_{0} = E_i D(\la)
$$
by setting $M_k := \im f(x)^k$.
By the simplicity of the head of $E_i D(\la)$, 
each section $M_{k-1} / M_{k}$ for $k=1,\dots,\eps_i(\la)$
must have irreducible head 
isomorphic to $D(\tilde e_i \la)\langle m+2k-2\rangle$.
Since $[e_i \underline{D}(\la):\underline{D}(\tilde e_i \la)] = \eps_i(\la)$
by Theorem~\ref{GTM}(3), there can be no other
composition
factors in $M_{k-1} / M_k$ that are 
isomorphic to $\underline{D}(\tilde e_i \la)$ on forgetting the grading.
This argument shows that
$$
[E_i D(\la): D(\tilde e_i \la)] = q^{m+\eps_i(\la)-1} [\eps_i(\la)].
$$
Moreover, the submodule $M_{\eps_i(\la)-1}$ at the bottom
of our filtration has irreducible socle
isomorphic to 
$D(\tilde e_i \la)\langle n \rangle$, so 
in fact $n = m+2\eps_i(\la)-2$ and 
$M_{\eps_i(\la)-1}$ is irreducible.
In view of the first part of Theorem~\ref{GTM}(3),
to complete the proof of (1) and the first part of (3), 
it remains to show that $m=1-\eps_i(\la)$.
But 
$D(\la)$ is self-dual and $E_i$ commutes with duality 
by Lemma~\ref{commutes}, hence $E_i D(\la)$ is self-dual too.
So the $q$-multiplicity $q^{m+\eps_i(\la)-1} [\eps_i(\la)]$ 
computed above
must be bar-invariant. This implies that
$m=1-\eps_i(\la)$ as required.

Finally consider (2) and the second part of (3).
Assume $\phi_i(\la) \neq 0$.
Entirely similar argument to the previous paragraph shows that
$F_i D(\la)$ has irreducible head
$D(\tilde f_i \la)\langle m\rangle$, socle
$D(\tilde f_i \la)\langle m + 2 \phi_i(\la)-2\rangle$,
and
$$
[F_i D(\la): D(\tilde f_i \la)]_q = q^{m+1-\phi_i(\la)} [\phi_i(\la)],
$$
for some $m \in \Z$. To complete the proof of the theorem,
we need to show that $m = 1-\phi_i(\la)$. 
This time we do not know that 
$F_i$ commutes with duality, so we must give a different argument
than before. By (1), we have that
$E_i D(\tilde f_i \la)\langle-\eps(\la)\rangle$
has irreducible socle isomorphic to $D(\la)$.
Hence using Lemma~\ref{cad}, we get that
\begin{align*}
\HOM_{R^\La_\alpha}(D(\la), D(\la)) &\cong\HOM_{R^\La_\alpha}(D(\la), E_i D(\tilde f_i \la)\langle -\eps_i(\la)\rangle)\\
&\cong\HOM_{R^\La_{\alpha+\al_i}}(F_i D(\la)\langle d_i(\la)-1 \rangle, 
D(\tilde f_i \la)\langle -\eps_i(\la)\rangle)\\
&\cong
\HOM_{R^\La_{\alpha+\al_i}}(D(\tilde f_i \la)\langle m+d_i(\la)-1\rangle, 
D(\tilde f_i \la)\langle-\eps_i(\la)\rangle).
\end{align*}
We deduce by Schur's lemma that
$m+d_i(\la)-1=-\eps_i(\la)$.
An application of (\ref{us}) completes the proof.
\end{proof}

\begin{Remark}\rm
It is also known that the polynomials $u_{\mu,\la;i}(q)$ and $v_{\mu,\la;i}(q)$
in Theorem~\ref{gg}(3) are bar-invariant.
This follows because $D(\la)^\circledast \cong D(\la)$
and the linear endomorphisms
of the Grothendieck group
induced by the functors
$E_i$ and $F_i$ commute with $\circledast$;
the latter statement is a consequence of
Theorem~\ref{cat} below and the bar-invariance of $E_i, F_i \in U_q(\g)$.
\end{Remark}

\subsection{\boldmath Mixed relations}
The next goal is to check that the ``mixed relation''
from (\ref{mix}) holds on the Grothendieck group.
We do this by appealing to some general results of Chuang and Rouquier
from \cite{CR} and Rouquier from \cite{Ro}.
We remark that in 
a previous version of this article, we established 
the required relations in Corollary~\ref{mixed2} (but {\em not}
the stronger Theorem~\ref{fancythm}) in a more elementary way,
using instead Theorem~\ref{gg} and a graded 
version of an argument of Grojnowski;
see e.g. \cite[Lemma~9.4.3]{Kbook}. 

To start with, following the formalism of \cite{Ro} (but inevitably using 
somewhat different notation, in particular, we have switched the roles of
$E_i$ and $F_i$),
we need to define endomorphisms
of functors
\begin{equation}\label{endos}
y:F_i \rightarrow F_i,\qquad
\psi:F_i F_j \rightarrow F_j F_i
\end{equation}
for all $i,j \in I$.
It suffices by additivity to
define natural homomorphisms
$y_M: F_i M \rightarrow F_i M$ and $\psi_M: F_i F_j(M) \rightarrow
F_j F_i(M)$
for each $M \in \Rep{R^\La_\alpha}$ and $\alpha \in Q_+$ of 
height $d$.
For this, $y_M$ is the homogeneous endomorphism of degree 2
defined by right multiplication
by $y_{d+1} \in R_{\alpha+\alpha_i}^\La$,
and $\psi_M$ is the homogeneous endomorphism of degree $-a_{i,j}$
defined by right multiplication
by $\psi_{d+1} \in R_{\alpha+\alpha_i+\alpha_j}^\La$.
See \cite[$\S$7.2]{CR} for similar definitions in the ungraded
setting.

Now, given $i \in I$,
let us denote the unit and the counit arising from the
adjunction from Lemma~\ref{cad}
by $\eta:\Id \rightarrow E_iF_i$ and
$\epsilon: F_i E_i \rightarrow \Id$, respectively.
On modules from $\Rep{R^\La_\alpha}$, 
these natural transformations
define maps that are homogeneous of degrees
$1-(\La-\alpha,\alpha_i)$ and $1 + (\La-\alpha,\alpha_i)$,
respectively.
Note also that the natural transformations
$$
\mathbf{1} (y^n) \circ \eta: \Id \rightarrow E_i F_i,
\qquad
\epsilon \circ (y^n) \bid:F_i E_i \rightarrow \Id
$$
define homogeneous maps 
of degrees $2n+1-(\La-\alpha,\alpha_i)$
and $2n+1+(\La-\alpha,\alpha_i)$, respectively, 
on modules from $\Rep{R^\La_\alpha}$.
Finally given also $j \in I$, define 
\begin{equation*}
\sigma:
F_j E_i 
\stackrel{\eta \mathbf{11}}{\longrightarrow}
E_i F_i F_j E_i
\stackrel{\mathbf{1} \psi \mathbf{1}}{\longrightarrow} E_i F_j F_i E_i
\stackrel{\mathbf{11} \epsilon}{\longrightarrow} E_i F_j,
\end{equation*}
which yields a map that is homogeneous of degree zero on every module;
cf. \cite[$\S$4.1.3]{Ro}.

\begin{Theorem}[{\cite{CR, Ro}}]\label{fancythm}
Suppose that
$\alpha \in Q_+$, $i, j \in I$, and 
set $a := (\La-\alpha,\alpha_i)$.
Let $M \in \Rep{R^\La_\alpha}$.
\begin{itemize}
\item[(1)]
If $i=j$ and $a \geq 0$ the natural transformation
$\sigma+\sum_{n=0}^{a-1} \bid (y^n) \circ \eta$
defines an isomorphism of graded modules
$$
F_j E_i(M) \oplus \bigoplus_{n=0}^{a-1} M\langle 2n+1-a\rangle
\stackrel{\sim}{\rightarrow} E_i F_j(M).
$$
\item[(2)]
If $i=j$ and $a \leq 0$ the natural transformation
$\sigma+\sum_{n=0}^{-a-1} \epsilon \circ (y^n) \bid$
defines an isomorphism of graded modules
$$
F_j E_i(M) \stackrel{\sim}{\rightarrow}
E_i F_j (M) \oplus \bigoplus_{n=0}^{-a-1} M\langle -2n-1-a\rangle.
$$
\item[(3)]
If $i \neq j$ the natural transformation $\sigma$
defines an isomorphism of graded modules
$$
F_j E_i(M) \stackrel{\sim}{\rightarrow} E_i F_j(M).
$$
\end{itemize}
\end{Theorem}

\begin{proof}
Since the maps in all cases are homogeneous of degree zero,
it suffices to prove the theorem on forgetting the gradings
everywhere. For (1) and (2), 
using Theorem~\ref{ISO},
we reduce to establishing the
analogous isomorphism 
for $H^\La_\alpha$, which is proved in \cite[Theorem 5.27]{CR} (using also
\cite[$\S$7.2]{CR} which verifies that the axioms
of $\mathfrak{sl}_2$-categorification are satisfied in that setting).
Once (1) and (2) have been established, (3)
follows by \cite[$\S$5.3.5]{Ro}.
\end{proof}

\begin{Corollary}\label{mixed2}
For all $i, j \in I$ and $\alpha \in Q_+$, we have that
$$[E_i,F_j]=\de_{i,j}\frac{K_i-K_i^{-1}}{q-q^{-1}}$$
as 
endomorphisms of $[\Rep{R_\alpha^\La}]$.
\end{Corollary}

\subsection{The graded categorification theorem}\label{sc}
Like in (\ref{like}), let us abbreviate
\begin{equation}\label{likely}
[\Proj{R^\La}] := \bigoplus_{\alpha \in Q_+} [\Proj{R_\alpha^\La}],
\qquad
[\Rep{R^\La}] := \bigoplus_{\alpha \in Q_+} [\Rep{R_\alpha^\La}].
\end{equation}
The exact functors $E_i^{(n)}, F_i^{(n)}$ and $K_i$ induce
$\Laurent$-linear endomorphisms of the
Grothendieck groups $[\Proj{R^\La}]$ and $[\Rep{R^\La}]$.
In view of Theorem~\ref{clas}, $[\Rep{R^\La}]$
is a free $\Laurent$-module on basis
$\{[D(\la)]\:|\:\la \in \RPar\}$.
Also 
let $Y(\la)$ denote the projective cover of $D(\la)$
in $\Rep{R^\La_\alpha}$, for each $\la \in \RPar_\alpha$.
Thus there is a degree-preserving surjection
$$
Y(\la) \twoheadrightarrow D(\la).
$$
The classes $\{[Y(\la)]\:|\:\la \in \RPar\}$ give a basis
for $[\Proj{R^\La}]$ as a free $\Laurent$-module.
Note by Corollary~\ref{dualpims} and (\ref{comp}) that
\begin{equation}
Y(\la)^\# \cong Y(\la).
\end{equation}

Comparing the relations of $U_q(\g)$ and $\mathbf f$, 
it follows easily that there
is an algebra anti-homomorphism 
\begin{equation}\label{flatmap}
\mathbf f \rightarrow U_q(\g),
\qquad
x \mapsto x^\flat
\end{equation}
such that $\theta_i^\flat := q^{-1} F_i K_i$ (which is the same
thing as $\tau^{-1}(E_i)$ according to (\ref{tauniv})).

\begin{Proposition}\label{id1}
There is a unique $\Laurent$-module isomorphism
$\db$ making the following diagram commute
$$
\begin{CD}
{_\Laurent}\mathbf f&@>\sim>\ga>&[\Proj{R}]\\
@V\beta VV&&@VV\pr V\\
V(\La)_\Laurent&@>\sim>\db>&[\Proj{R^\La}],
\end{CD}
$$
where $\beta$ denotes the surjection
$x \mapsto x^\flat v_\La$, $\ga$ is the isomorphism
from Theorem~\ref{klthm}, and $\pr$
is the $\Laurent$-linear map induced by the additive
functor $\pr$ from (\ref{prprpr}).
Moreover, $\db$ intertwines the left action
of $F_i^{(n)} \in U_q(\g)_\Laurent$ on $V(\La)_\Laurent$
with the endomorphism of $[\Proj{R^\La}]$
induced by the divided power functor $F_i^{(n)}$,
for every $i \in I$ and $n \geq 1$.
It obviously intertwines the $K_i$'s too.
\end{Proposition}

\begin{proof}
We first show there exists an
$\Laurent$-module homomorphism
$\db$
making the diagram commute.
Remembering that the map $\beta$ involves a twist
by the anti-homomorphism $x \mapsto x^\flat$, 
the well-known description of
$V(\La)$ by generators and relations
implies 
that $\ker\beta$ is generated as a right ideal
by the elements
$\{\theta_i^{((\La,\alpha_i)+1)}\:|\:i\in I\}$.
Therefore it suffices to check that
$\ker \pr$ is a right ideal of the algebra
$[\Proj{R}]$ and that
$\pr (\ga(\theta_i^{((\La,\alpha_i)+1)})) = 0$ 
for each $i$.
To show that $\ker \pr$ is a right ideal, take $P \in \Proj{R_\alpha}$ and
$Q \in \Proj{R_\beta}$ 
such that
$\pr\,P = \{0\}$. We need to show that
$\pr\,\Ind_{\alpha,\beta}^{\alpha+\beta} (P \boxtimes Q) = \{0\}$, 
or equivalently, that
$\HOM_{R_{\alpha+\beta}}(\Ind_{\alpha,\beta}^{\alpha+\beta} (P \boxtimes Q),
L) = \{0\}$
for every $R_{\alpha+\beta}^\La$-module $L$.
This follows from the isomorphism
$$
\HOM_{R_{\alpha+\beta}}(\Ind_{\alpha,\beta}^{\alpha+\beta}(P \boxtimes Q), L)
\cong \HOM_{R_{\alpha,\beta}}(P \boxtimes Q, \Res_{\alpha,\beta}^{\alpha+\beta}
L).
$$
To show that $\pr (\ga(\theta_i^{((\La,\alpha_i)+1)})) = 0$,
we show equivalently that
$$
\pr\ \theta_i^{(\La,\alpha_i)+1} \infl (D(\varnothing)) = \{0\}.
$$
In view of Lemma~\ref{sun}, this follows
if we can show that
$F_i^{(\La,\alpha_i)+1} D(\varnothing) = \{0\}$.
This holds thanks to Theorem~\ref{gg}(3)
since $\phi_i(\varnothing) = (\La,\alpha_i)$.

The map $\pr$ is obviously surjective, and $\ga$ is an isomorphism,
hence we get that $\db$ is surjective.
It sends the $(\La-\alpha)$-weight space of
$V(\La)_\Laurent$ onto $[\Proj{R^\La_\alpha}]$.
It is an isomorphism because
$[\Proj{R^\La_\alpha}]$ is a free $\Laurent$-module
of rank $|\RPar_\alpha|$
thanks to Theorem~\ref{clas}, which is the same as the rank
of the $(\La-\alpha)$-weight space of $V(\La)_\Laurent$.

The fact that $\db$ intertwines the $F_i^{(n)}$'s
follows from the definitions.
\end{proof}

Recall the Cartan pairing
$\langle.,.\rangle:[\Proj{R^\La}] \times [\Rep{R^\La}]
\rightarrow \Laurent$
from $\S$\ref{SSGr}
and the Shapovalov pairing
$\langle.,.\rangle:V(\La)_\Laurent \times V(\La)_\Laurent^*
\rightarrow \Laurent$
from $\S$\ref{md}.
Let
\begin{equation}\label{dstar}
\dd: [\Rep{R^\La}] \stackrel{\sim}{\rightarrow}
V(\La)_\Laurent^*
\end{equation}
be the dual map
to the isomorphism 
$\db$ of Proposition~\ref{id1}
with respect to these pairings.
Thus, $\dd$ is the $\Laurent$-module isomorphism
defined by the
equation 
\begin{equation}\label{itsid}
\langle \db(x), y \rangle = \langle x, \dd(y)\rangle
\end{equation}
for all $x \in V(\La)_\Laurent^*$ and $y \in [\Rep{R^\La}]$.

Also let
\begin{equation}
\circledast: [\Rep{R^\La}] \rightarrow [\Rep{R^\La}]
\end{equation}
be the anti-linear involution induced by the duality
$\circledast$.

\begin{Proposition}\label{id2}
The isomorphism $\dd$ from (\ref{dstar}) intertwines 
the endomorphism of $[\Rep{R^\La}]$
induced by the divided power functor $E_i^{(n)}$
with
the left action
of $E_i^{(n)} \in U_q(\g)_\Laurent$ on $V(\La)_\Laurent^*$,
for every $i \in I$ and $n \geq 1$.
It obviously intertwines the $K_i$'s too.
Finally, it intertwines the
anti-linear involution $\circledast$ with
the bar-involution
on $V(\La)^*_\Laurent$.
\end{Proposition}

\begin{proof}
For the first statement, it suffices by Lemma~\ref{divp} to show that
$E_i \circ \dd = \dd \circ E_i$.
This is immedate from Proposition~\ref{id1} and 
the defining property (1) of the Shapovalov form from $\S$\ref{md}, 
because the functor $E_i$ is
right adjoint to $F_i K_i \langle-1\rangle$
by Lemma~\ref{cad}, while 
$\tau^{-1}(E_i) = q^{-1} F_i K_i$
according to (\ref{tauniv}).

For the last statement, we show that
$\dd(v^\circledast) - \overline{\dd(v)} = 0$
for each $v \in [\Rep{R^\La_\alpha}]$ by induction on
height. This is clear in the case $\alpha = 0$.
Now take $\alpha\in Q_+$ with $\height(\alpha) > 0$.
Since $E_i$ commutes with $\dd$, with the bar-involution,
and with the duality $\circledast$ by Lemma~\ref{commutes},
we get from the induction hypothesis
that
$E_i(\dd(v^\circledast) - \overline{\dd(v)}) = 0$ 
for every $i \in I$. Hence,
$\dd(v^\circledast) - \overline{\dd(v)}$ is a highest weight vector
of weight different from $\La$, so it is zero.
\end{proof}

Now we can prove the following fundamental theorem 
which makes precise a sense in which
$\Proj{R^\La}$ {categorifies} the $U_q(\g)_\Laurent$-module
$V(\La)_\Laurent$ and
$\Rep{R^\La}$ {categorifies} 
$V(\La)_\Laurent^*$.

\begin{Theorem}\label{cat} The following diagram commutes:
$$
\begin{CD}
V(\La)_\Laurent&@>\sim > \db>&[\Proj{R^\La}]\\
@Va VV&&@VVbV\\
V(\La)_\Laurent^*&@<\sim <\dd<&[\Rep{R^\La}],
\end{CD}
$$
where $a:V(\La)_\Laurent \hookrightarrow V(\La)_\Laurent^*$
is the canonical inclusion,
and
$b:[\Proj{R^\La}]  \rightarrow [\Rep{R^\La}]$
is the $\Laurent$-linear map induced by the natural inclusion
of $\Proj{R_\al^\La}$ into $\Rep{R^\La_\al}$ for each $\al \in Q_+$.
Hence:
\begin{itemize}
\item[(1)] $b$ is injective and becomes an isomorphism
over $\Q(q)$;
\item[(2)] both maps $\db$ and $\dd$ 
commute with the actions of $E_i^{(n)}, F_i^{(n)}$ and $K_i$;
\item[(3)] 
both maps $\db$ and $\dd$ intertwine
the involution $\circledast$ coming from
duality with the bar-involution;
\item[(4)] the isomorphism $\db$ 
identifies 
the Shapovalov form on $V(\La)_\Laurent$
with Cartan form on $[\Proj{R^\La}]$.
\end{itemize}
\end{Theorem}

\begin{proof}
Everything in sight is a free $\Laurent$-module, so it
does no harm to extend scalars from $\Laurent$ to $\Q(q)$.
Denote the resulting $\Q(q)$-linear maps
by $\hat a, \hat b, \hat \db$ and $\hat \dd$.
Actually, we may as well identify
$\Q(q) \otimes_{\Laurent} V(\La)_\Laurent$
and $\Q(q) \otimes_{\Laurent} V(\La)_\Laurent^*$
both with $V(\La)$, so that $\hat a$ is just the identity map,
and then we need to show the following diagram commutes:
$$
\begin{CD}
V(\La)&@>\sim > \hat\db>&\Q(q)\otimes_\Laurent[\Proj{R^\La}]\\
@|&&@VV\hat bV\\
V(\La)&@<\sim <\hat\dd<&\Q(q)\otimes_\Laurent[\Rep{R^\La}].
\end{CD}
$$
Note also that $\hat b$ obviously commutes with $E_i^{(n)},
F_i^{(n)}, K_i$ and $\circledast$,
$\hat\db$ commutes with $F_i^{(n)}$ and $K_i$
by Proposition~\ref{id1},
and $\hat\dd$ commues with $E_i^{(n)}, K_i$
and $\circledast$ by Proposition~\ref{id2}.
Hence (1), (2) and (3) all follow easily 
from the commutativity of this diagram.
Also once the commutativity is established, 
(4) follows immediately by (\ref{itsid}).

To prove that the above diagram commutes, we show by induction that
it commutes on restriction to the $(\La-\alpha)$-weight
spaces for each $\alpha \in Q_+$.
The diagram obviously commutes on restriction to 
the highest weight space, so
assume now that $\alpha > 0$ and that we have already established
the commutativity on restriction to the $(\La-\beta)$-weight spaces
for all $0 \leq \beta < \alpha$.
It suffices to show that
$\hat \dd \hat b \hat \db (F_j w) = F_j w$
for all $j \in I$ and $w$ in the $(\La-\alpha+\alpha_j)$-weight space
of $V(\La)$.
This follows if we can check that
\begin{equation}\label{tocjeck}
\left\langle F_i v, \hat \dd \hat b \hat \db (F_j w) \right\rangle
= \left\langle F_i v, F_j w \right\rangle
\end{equation}
for all $i \in I$ and $v$ in the $(\La-\alpha+\alpha_i)$-weight space
of $V(\La)$.
For this we compute using the defining property
of the Shapovalov form, Propositions~\ref{id1}
and \ref{id2}, and Corollary~\ref{mixed2}:
\begin{align*}
\left\langle F_i v, \hat \dd \hat b \hat \db (F_j w) \right\rangle
&= \left\langle v, q^{-1} K_i E_i \hat \dd \hat b \hat \db (F_j w) \right\rangle= \left\langle v, q^{-1} K_i \hat \dd E_i F_j \hat b \hat \db (w) \right\rangle\\
&= \left\langle v, q^{-1} K_i \hat \dd \left(F_j E_i + \delta_{i,j}\frac{K_i-K_i^{-1}}{q-q^{-1}}\right) \hat b \hat \db (w) \right\rangle.
\end{align*}
By the inductive hypothesis, we know already that our diagram commutes
on the $(\La-\alpha+\alpha_i)$- and $(\La-\alpha+\alpha_i+\alpha_j)$-weight spaces, hence Proposition~\ref{id1} allows us to commute
the $\hat \dd$ and the $F_j$ past each other, to get that
\begin{align*}
\left\langle F_i v, \hat \dd \hat b \hat \db (F_j w) \right\rangle
&= \left\langle v, q^{-1} K_i 
\left(F_j E_i + \delta_{i,j}\frac{K_i-K_i^{-1}}{q-q^{-1}}\right) 
\hat \dd \hat b \hat \db (w) \right\rangle\\
&= \left\langle v, q^{-1} K_i 
E_i F_j
\hat \dd \hat b \hat \db (w) \right\rangle= \left\langle F_i v, F_j
\hat \dd \hat b \hat \db (w) \right\rangle.
\end{align*}
Finally by the inductive hypothesis, we know already
that $\hat\dd \hat b \hat \db(w) = w$, so this completes the
proof of (\ref{tocjeck}).
\end{proof}

\begin{Remark}\rm\label{fancyrem}
In view of Theorem~\ref{fancythm},
Theorem~\ref{cat} can also be formulated
as an example of a $2$-representation of
the $2$-Kac-Moody algebra
$\mathfrak{A}(\mathfrak{g})$ in the sense of Rouquier
\cite{Ro}.
The required data as specified in \cite[Definition 5.1.1]{Ro}
comes from the categories $\Rep{R_\alpha^\La}$
for all $\alpha \in Q_+$, 
together with the functors $F_i$ and $E_i$ from
(\ref{gre})--(\ref{grf}),
the adjunction from Lemma~\ref{cad}, and 
the endomorphisms (\ref{endos}).
\end{Remark}

\subsection{A graded dimension formula}\label{sgdf}
As the first application of Theorem~\ref{cat},
we can derive a combinatorial formula for 
the graded dimension of $R^\La_\alpha$.

Let $\la=(\la^{(1)},\dots,\la^{(l)}) \in \Par$
be an $l$-multipartition, and set $d := |\la|$.
A {\em standard $\la$-tableau} 
$\T=(\T^{(1)},\dots,\T^{(l)})$ is
obtained from the diagram of $\la$ by 
inserting the integers $1,\dots,d$ into the nodes, allowing no repeats,
so that the entries in each individual $T^{(m)}$
are strictly increasing along rows from left to right and 
down columns from top to bottom.
The set of all standard $\la$-tableaux will be denoted by $\St(\la)$.

To each $\T \in \St(\la)$ we associate its {\em residue sequence}
\begin{equation}\label{Pre}
\bi^\T=(i_1,\dots,i_d)\in I^d,
\end{equation}
where $i_r \in I$ 
is the residue of the node occupied by 
$r$ in $\T$ ($1\leq r\leq d$) in the sense of (\ref{resdef}) (reduced 
modulo $e$). 
Recalling (\ref{EDMUA}), 
define the {\em degree} of $\T$ inductively from
\begin{equation}\label{De}
\deg(\T) :=
\left\{\begin{array}{ll}
\deg(\T_{\leq (d-1)}) + d_A(\la)&\text{if $d > 0$,}\\
0&\text{if $d = 0$,}
\end{array}
 \right.
\end{equation}
where for $d > 0$ we let $A$ denote the node of $\T$ containing
 entry $d$, and $\T_{\leq(d-1)}$ denotes the tableau obtained from
$\T$ by removing this entry.

\begin{Theorem}\label{gdim}
For $\alpha \in Q_+$ and 
$\bi, \bj \in I^\alpha$, we have that
\begin{align*}
\qdim\ e(\bi) R^\La_\alpha e(\bj) 
&= 
\sum_{\substack{
\la \in \Par\\ \Stab, \T \in \St(\la) \\
\bi^\Stab = \bi, \bi^\T = \bj}}
q^{2\defect(\alpha)-\deg(\Stab)-\deg(\T)}
=
\sum_{\substack{
\la \in \Par\\ \Stab, \T \in \St(\la) \\
\bi^\Stab = \bi, \bi^\T = \bj}}
q^{\deg(\Stab)+\deg(\T)}.
\end{align*}
\end{Theorem}

\begin{proof}
Given $\bi \in I^\alpha$, let
$F_\bi := F_{i_d} \cdots F_{i_1}$ for short.
The definition (\ref{grf}) implies easily that
$R^\La_\alpha e(\bi) \cong F_\bi R^\La_0
\langle \defect(\alpha)\rangle$ as graded left $R^\La_\alpha$-modules.
So
\begin{align*}
\qdim\ e(\bi) R^\La_\alpha e(\bj)
&= \qdim\ \HOM_{R^\La_\alpha}(R^\La_\alpha e(\bi), R^\La_\alpha e(\bj))\\
&=
\qdim\ \HOM_{R^\La_\alpha}(F_\bi R^\La_0, F_\bj R^\La_0)=
\langle F_\bi R^\La_0, F_\bj R^\La_0\rangle.
\end{align*}
Invoking Theorem~\ref{cat} (especially part (4)), we deduce that
$$
\qdim\ e(\bi) R^\La_\alpha e(\bj)
= \langle F_\bi v_\La, F_\bj v_\La \rangle.
$$
We now proceed to compute this by working in
terms of the monomial basis of
the Fock space $F(\La)$ from $\S$\ref{expre}.
In particular, we will exploit the 
sesquilinear form $\langle.,.\rangle$
on $F(\La)$ from (\ref{starz}).

By considerations involving (\ref{uact1}), we that
$$
F_\bj v_\La = \sum_{\substack{\mu \in \Par, \T \in \St(\mu),
\bi^\T = \bj}}
q^{-{\codeg}(\T)} M_\mu,
$$
where ${\codeg}(\T)$ is defined inductively by
\begin{equation}
\codeg(\T) :=
\left\{\begin{array}{ll}
\codeg(\T_{\leq (d-1)}) + d^A(\la_A)&\text{if $d > 0$,}\\
0&\text{if $d = 0$,}
\end{array}
 \right.
\end{equation}
adopting the same notations as in (\ref{De}).
By \cite[Lemma 3.12]{BKW}, 
we have that $-\codeg(\T) = \deg(\T)-\defect(\alpha)$.
Also $F_\bj v_\La$ is bar-invariant. 
Putting these things together and simplifying,
we get that
\begin{align*}
\qdim\ e(\bi) R^\La_\alpha e(\bj)
&= 
\sum_{\substack{\la\in \Par, \Stab \in \St(\la), \bi^\Stab = \bi \\
\mu \in \Par, \T\in \St(\mu), \bi^\T = \bj}}
q^{\defect(\alpha)-\deg(\Stab)}q^{\defect(\alpha)-\deg(\T)}
 \langle  
M_\la,
\overline{M_\mu}
\rangle.
\end{align*}
In view of (\ref{starz})
this gives the first expression 
for
$\qdim\ e(\bi) R^\La_\alpha e(\bj)$
from the statement of the theorem.

Finally, we note by
Lemma~\ref{contra2}(3) that
$$
\qdim\ e(\bi) R^\La_\alpha e(\bj)
= \langle F_\bi v_\La, F_\bj v_\La \rangle
=q^{2\defect(\alpha)} \overline{\langle F_\bj v_\La, F_\bi v_\La \rangle}.
$$
Then we compute the right hand side of this by similar
substitutions to the previous paragraph.
This gives the second expression
from the statement of the theorem.
\end{proof}

\subsection{Extremal sequences}\label{sex}
For later use, we recall here an elementary but useful
observation from \cite{BKdur} which generalizes almost at once to the
present graded setting.
Given $\bi= (i_1, \dots, i_d)\in I^d$ we can gather 
consecutive equal entries together to
write it in the form 
\begin{equation}
\label{E241201}
\bi=(j_1^{m_1}\dots j_n^{m_n})
\end{equation}
where $j_r\neq j_{r+1}$ for all $1\leq r<n$. For example 
$(2,2,2,1,1,2)=(2^3 1^2 2)$. 

Now take $\alpha \in Q_+$ with $\height(\alpha) = d$.
Given a non-zero $M \in \Rep{R^\La_\alpha}$ and $i \in I$,
we let
\begin{equation}
\eps_i(M) := \max\{k \geq 0\:|\:E_i^k (M) \neq \{0\}\}.
\end{equation}
For example, $\eps_i(D(\la)) =\eps_i(\la)$ by Theorem~\ref{gg}(3).
We say that a sequence $\bi$ of the form 
(\ref{E241201}) is an {\em extremal sequence} for $M$ if 
$m_r=\eps_{j_r}(E_{j_{r+1}}^{m_{r+1}}\dots E_{j_n}^{m_n} M)$
for all $r = n, n-1, \dots, 1$. 
Informally speaking this means that among all $\bi \in I^\alpha$ such that 
$e(\bi) M \neq \{0\}$,
we first choose those with the longest $j_n$-string at the end, 
then among these we choose the ones with the longest  
$j_{n-1}$-string preceding the $j_n$-string at the end, and so on.
It is obvious that if $\bi$ is an extremal sequence for $M$,
then $e(\bi) M \neq \{0\}$.

\begin{Lemma}[{\cite[Corollary 2.17]{BKdur}}]
\label{C241201}
If $\bi = (i_1,\dots,i_d)$ is an extremal sequence for $M \in \Rep{R^\La_\alpha}$
 of the form (\ref{E241201}), then 
$\la := \tilde f_{i_d} \cdots \tilde f_{i_1} \varnothing$
is a well-defined element of $\RPar_\alpha$,
and 
$$
[M:D(\la)]_q = (\qdim\ e(\bi) M)/([m_1]!\dots [m_r]!).
$$ 
\end{Lemma}

\section{Graded Specht modules and decomposition numbers}

We continue with notation as in the previous section,
so $F$ is any algebraically closed field, $\xi \in F^\times$
is of ``quantum characteristic'' $e$ as defined at
the start of the previous section, and $\La$ is fixed according to (\ref{note}).
To this data and every $\alpha \in Q_+$
we have associated a block $H^\La_\alpha$ of a
cyclotomic Hecke algebra with parameter $\xi \in F^\times$
(degenerate if $\xi=1$),
which is isomorphic to the algebra
$R^\La_\alpha$ according to Theorem~\ref{ISO}. 

\subsection{Input from geometric representation theory}
Let us specialize the setup of $\S$\ref{md} at $q=1$, setting
\begin{equation}
V(\La)_\Z := \Z \otimes_{\Laurent} V(\La)_{\Laurent},\qquad
V(\La)^*_\Z := \Z \otimes_{\Laurent} V(\La)^*_{\Laurent},
\end{equation}
where we view $\Z$ as an $\Laurent$-module so that $q$ acts as $1$.
Recalling that $\g = \widehat{\mathfrak{sl}}_e(\C)$ if $e > 0$ 
or $\mathfrak{sl}_\infty(\C)$ if $e = 0$,
let $U(\g)_\Z$ denote the Kostant $\Z$-form for the universal
enveloping algebra of $\g$,
generated by the usual divided powers $e_i^{(n)}$
and $f_i^{(n)}$ in its Chevalley generators.
This acts on $V(\La)_\Z$ and $V(\La)_\Z^*$ so that 
$e_i^{(n)}$ and $f_i^{(n)}$ 
act as $1 \otimes E_i^{(n)}$ and $1 \otimes F_i^{(n)}$, respectively.
In other words, $V(\La)_\Z$ 
is the standard $\Z$-form 
for the irreducible highest weight module
for $\g$ of highest weight $\La$,
and $V(\La)_\Z^*$ is the dual lattice 
under the usual contravariant form $(.,.)$ (which at $q=1$ coincides
with the Shapovalov form).

Paralleling (\ref{likely}) in the ungraded setting, we set 
\begin{equation}\label{likely2}
[\proj{H^\La}] := \bigoplus_{\alpha \in Q_+} [\proj{H_\alpha^\La}],
\qquad
[\rep{H^\La}] := \bigoplus_{\alpha \in Q_+} [\rep{H_\alpha^\La}].
\end{equation}
The exact functors $e_i$ and $f_i$ from (\ref{ei})--(\ref{fi})
induce $\Z$-linear endomorphisms of these spaces.
Also we have the Cartan pairing
$$
(.,.):[\proj{H^\La}] \times [\rep{H^\La}] \rightarrow \Z,\quad
([P],[M]) := \dim \hom_{H^\La_\alpha}(P, M)
$$
for $\alpha \in Q_+$,
$P \in \proj{H^\La_\alpha}$ and $M \in \rep{H^\La_\alpha}$.

If we forget the grading in Theorem~\ref{cat}, we deduce
that there is 
a commuting square
\begin{equation}\label{degr}
\begin{CD}
V(\La)_\Z&@>\sim > \underline\db>&[\proj{H^\La}]\\
@V\underline a VV&&@VV\underline bV\\
V(\La)_\Z^*&@<\sim <\underline\dd<&[\rep{H^\La}],
\end{CD}
\end{equation}
where $\underline a:V(\La)_\Z
 \hookrightarrow V(\La)_\Z^*$
is the canonical inclusion,
$\underline{b}:[\proj{H^\La}]  \hookrightarrow [\rep{H^\La}]$
is induced by the natural inclusion of categories,
$\underline{\db}$ is the unique $\Z$-module isomorphism that
sends 
the 
highest weight vector $v_\La \in V(\La)_\Z$ to
the isomorphism class of the trivial $H^\La_0$-module
and commutes with the $f_i$'s,
and finally $\underline\dd$ is the dual map
to $\underline\db$ with respect to the pairings $(.,.)$.

The following is a deep result underlying almost all
subsequent work in this paper.
It was proved by Ariki \cite{Ariki} as a consequence of
the geometric representation theory of
quantum algebras and affine Hecke algebras developed by
Kazhdan, Lusztig and Ginzburg, as the key step in his proof
of the (generalized) Lascoux-Leclerc-Thibon conjecture.
For an exposition of the proof and a fuller historical
account, we refer 
to \cite[Theorem 12.5]{Abook}.
We cite also our recent work \cite{BKariki} which gives
a quite different proof in the degenerate case $\xi = 1$
based on Schur-Weyl duality for higher levels and
the Kazhdan-Lusztig conjecture in finite type $A$.

\begin{Theorem}[{\cite[Theorem 4.4]{Ariki}}]\label{maina}
Assume $\cha F = 0$.
The isomorphism $\underline{\db}$ from (\ref{degr})
maps the canonical basis of $V(\La)_\Z$
to the basis of $[\proj{H^\La}]$ arising
from projective indecomposable modules.
\end{Theorem}

\begin{Corollary}\label{mainc}
Assume $\cha F = 0$.
The isomorphism $\underline{\dd}$ from (\ref{degr})
maps the basis of $[\rep{H^\La}]$ arising
from irreducible modules to
the dual-canonical basis of $V(\La)_\Z^*$.
\end{Corollary}

\begin{proof}
This follows from the definition of the map
$\underline{\dd}$, since the basis arising
from the irreducible modules is dual to the basis
arising from the projective indecomposable modules under the
Cartan pairing, and the dual-canonical basis is dual to the
canonical basis under the contravariant form.
\end{proof}

We have not yet incorporated any particular parametrization
for the bases mentioned in either Theorem~\ref{maina} or 
Corollary~\ref{mainc}. This is
addressed in detail in Ariki's work via the theory of Specht modules,
as we will explain in the next subsections.
Even before we introduce Specht modules into the picture,
we can show that the bijection
between the isomorphism classes of irreducible modules and
the dual-canonical basis from Corollary~\ref{mainc}
is consistent with the
parametrizations of these two sets by restricted multipartitions
from Theorem~\ref{GTM} and $\S$\ref{dcqc}, respectively.

To do this, recall the
quasi-canonical basis
$\{Y_\la\}$ from $\S$\ref{sqcb} and the dual-canonical basis
$\{D_\la\}$ from $\S$\ref{dcqc}, respectively,
both of which are parametrized
by the set $\RPar$ of restricted multipartitions.
Denote their specializations at $q=1$
by
\begin{equation}
\underline{Y}_\la := 1 \otimes Y_\la \in V(\La)_\Z,\qquad
\underline{D}_\la := 1 \otimes D_\la \in V(\La)_\Z^*,
\end{equation}
for $\la \in \RPar$.
By Lemma~\ref{party},
$\{\underline{Y}_\la\:|\:\la \in \RPar\}$ and
$\{\underline{D}_\la\:|\:\la \in \RPar\}$ are the
canonical and dual-canonical bases of $V(\La)_\Z$
and $V(\La)_\Z^*$, respectively. 

Recall also from Theorem~\ref{GTM}
that the irreducible $H^\La_\alpha$-modules
are denoted $\{\underline{D}(\la)\:|\:\la \in \RPar_\alpha\}$;
they are defined recursively in terms of the crystal graph
by (\ref{irreps}). For $\la \in \RPar_\alpha$, let
$\underline{Y}(\la)$ denote the projective cover
of $\underline{D}(\la)$, so that
\begin{equation}
\underline{Y(\la)} \cong \underline{Y}(\la).
\end{equation}
Now we reformulate Theorem~\ref{maina} and 
Corollary~\ref{mainc} incorporating these explicit parametrizations
as follows:

\begin{Theorem}\label{withlabels}
Assume $\cha F = 0$.
For each $\la \in \RPar_\alpha$,
we have that $\underline{\db}(\underline{Y}_\la)
=[\underline{Y}(\la)]$ and
$\underline{\dd}([\underline{D}(\la)]) = \underline{D}_\la$.
\end{Theorem}

\begin{proof}
Reversing the argument with duality from the proof of Corollary~\ref{mainc}, 
it suffices to prove the second statement.
For that, we know already from Corollary~\ref{mainc}
that there is some bijection $\sigma:\RPar_\alpha \rightarrow \RPar_\alpha$
such that
$$
\underline{\dd}([\underline{D}(\la)]) = \underline{D}_{\sigma(\la)}.
$$
We need to show that $\sigma$ is the identity map. For this we
repeat an easy argument from the proof of \cite[Theorem 4.4]{BKrep},
as follows.

Proceed by induction on $\height(\alpha)$,
the statement being trivial for $\height(\alpha) = 0$.
For the induction step,
take $\alpha > 0$ and $\la \in \RPar_\alpha$.
Write $\la = \tilde f_i \mu$ for some $i \in I$ and
$\mu \in \RPar_{\alpha-\alpha_i}$.
By induction, we know that $\sigma(\mu) = \mu$.
By Proposition~\ref{dca} specialized at $q=1$, 
we know that $f_i^{\phi_i(\mu)} 
\underline{D}_\mu \neq 0$, and
$f_i \underline{D}_\mu = \phi_i(\mu) \underline{D}_\la + (*)$
where $(*)$ is a linear combination of $\underline{D}_\nu$'s
such that $f_i^{\phi_i(\mu)-1} \underline{D}_\nu = 0$.
Since $f_i^{\phi_i(\mu)} \underline{D}_\mu \neq 0$
we have that $f_i^{\phi_i(\mu)-1} \underline{D}_\la \neq 0$.

Applying the commutativity of (\ref{degr})
we deduce that $f_i \underline{D}(\mu)$
has a unique (up to isomorphism) composition factor $D$ such that
$f_i^{\phi_i(\mu)-1} [D] \neq 0$, and 
$\underline{\dd}([D]) = \underline{D}_\la$.
On the other hand, by Theorem~\ref{GTM},
$f_i \underline{D}(\mu)$ has a composition factor isomorphic
to $\underline{D}(\la)$ 
and $f_i^{\phi_i(\mu)-1} [\underline{D}(\la)] \neq 0$.
Hence 
$\underline{\dd}([\underline{D}(\la)]) = \underline{D}_\la$,
i.e. $\sigma(\la) = \la$.
\end{proof}

\subsection{Graded Specht modules}\label{cl}
The cyclotomic Hecke algebra $H^\La_d$ is a cellular algebra in 
the sense of
\cite{GL} with weight poset 
$\{\la \in \Par\:\big|\,\:|\la| = d\}$
partially ordered by $\unlhd$.
For the explicit 
construction of the underlying cell datum,
we refer the reader to \cite{DJM}; see also \cite[$\S$6]{AMR}
for the appropriate modifications in the degenerate case.
The associated cell modules are the so-called {\em Specht modules} \
$\underline{S}(\la)$ for each
$\la \in \Par$ with $|\la| = d$.
If $\cont(\la) = \alpha$ in the sense of (\ref{contdef})
then
$\underline{S}(\la)$ belongs to the block parametrized
by $\alpha$ in the block decomposition (\ref{blockdec});
this follows from the character formula (\ref{chspecht}) below.
Hence, invoking the following standard lemma (taking $e := e_\alpha$), we can project the cellular
structure on $H^\La_d$ to 
the block $H^\La_\alpha$, 
to get also that $H^\La_\alpha$
is a cellular algebra with weight poset $(\Par_\alpha, \unlhd)$
and cell modules $\{\underline{S}(\la)\:|\:\la \in \Par_\alpha\}$.

\begin{Lemma}
Let $A$ be a cellular algebra with cell datum
$(I, M, C, *)$ 
and associated cell modules $\{V(\la)\:|\:\la \in I\}$.
Let $e \in A$ be a central idempotent.
Then $e A e$ is a cellular algebra with cell datum
$(\bar I, \bar M, \bar C, \bar *)$ 
and associated cell modules $\{V(\la)\:|\:\la \in \bar I\}$
where:
\begin{itemize}
\item[(1)] 
$\bar I = \{\la \in I\:|\:e V(\la) = V(\la)\}$;
\item[(2)] $\bar M(\la) = M(\la)$ for each $\la \in \bar \La$;
\item[(3)]
$\bar C^\la_{s,t} = e C^\la_{s,t} e$ for each $\la \in \bar I$
and $s,t \in \bar M(\la)$;
\item[(4)] $\bar *$ is the restriction
of $*$ (which necessarily leaves $e A e$ invariant).
\end{itemize}
\end{Lemma}

In \cite{BKW}, we constructed a canonical
graded lift of the Specht module $\underline{S}(\la)$, i.e. we gave an explicit
construction of a graded $R^\La_\alpha$-module
$S(\la)$ such that 
\begin{equation}
\underline{S(\la)} \cong \underline{S}(\la)
\end{equation}
as $H^\La_\alpha$-modules.
We refer to $S(\la)$ as a {\em graded Specht module}.
Rather than repeat the definition here, we 
just note that the construction produces
an explicit homogeneous basis 
$\{v_\T\:|\:\T \in \St(\la)\}$
for $S(\la)$, in which the vector $v_\T$ belongs to $e(\bi^\T) S(\la)$ and
is of degree
$\deg(\T)$, notation as in (\ref{Pre})--(\ref{De}).
In particular, this means that the $q$-character of $S(\la)$
(by which we mean the $q$-character of its inflation to $R_\alpha$
in the sense of (\ref{qch})) is given by
\begin{equation}\label{chspecht}
\CH S(\la) = \sum_{\T \in \St(\la)} q^{\deg(\T)} \bi^\T.
\end{equation} 
We also derived the following branching rule
for graded Specht modules:

\begin{Proposition}[{\cite{BKW}}]\label{filt}
Let $\la \in \Par_\alpha$, $i \in I$, and $A_1,\dots,A_c$ be all
the removable $i$-nodes of $\la$
in order from bottom to top. Then 
$E_i S(\la)$ has a filtration
$$
\{0\} = V_0 \subset V_1 \subset\cdots\subset V_c = E_i S(\la)
$$
as a graded $R^\La_{\alpha-\alpha_i}$-module such that
$V_m / V_{m-1} \cong S(\la_{A_m})\langle d_{A_m}(\la)\rangle$
for all $1 \leq m \leq c$.
\end{Proposition}

\begin{proof}
This follows from \cite[Theorem 4.11]{BKW} on projecting to 
$R^\La_{\alpha-\alpha_i}$.
\end{proof}

Using this we can identify the image of
$[S(\la)] \in \Rep{R^\La}$
under the isomorphism
$\dd:[\Rep{R^\La}] \rightarrow V(\La)^*_\Laurent$
from Theorem~\ref{cat}.
Of course it is the standard monomial $S_\la$
from (\ref{stmn}):

\begin{Theorem}\label{ids}
For each $\la \in \Par_\al$, we have that
$\dd([S(\la)]) = S_\la$.
\end{Theorem}

\begin{proof}
We proceed by induction on $\height(\alpha)$.
The result is trivial in the case $\height(\alpha) = 0$,
so suppose that $\height(\alpha) > 0$.
We must show that 
$\dd([S(\la)]) - S_\la = 0$ in $V(\La)$.
Since $V(\La)$ is an irreducible highest weight module,
this follows if we can check that
$E_i (\dd([S(\la)]) - S_\la) = 0$
for every $i \in I$.
For this, we have by Theorem~\ref{cat}(2), Proposition~\ref{filt}
and the induction hypothesis
that
$$
E_i \dd([S(\la)])
=
\dd E_i ([S(\la)])
= 
\sum_A q^{d_A(\la)} \dd([S(\la_A)])
= 
\sum_A q^{d_A(\la)} S_{\la_A}.
$$
By (\ref{sact1}) 
this is equal to $E_i S_\la$.
\end{proof}

\begin{Corollary}\label{sbase}
The classes $\{[S(\la)]\:|\:\la \in \RPar_\alpha\}$
give a basis for $\Rep{R^\La_\alpha}$ as a free
$\Laurent$-module.
\end{Corollary}

\begin{proof}
This follows from Theorems~\ref{ids} and \ref{tri}(1).
\end{proof}

\begin{Corollary}
For $\la \in \Par$ and $i \in I$, the following hold in $[\Rep{R^\La}]$:
\begin{align*}
E_i [S(\la)]&=\sum_A q^{d_A(\la)} [S(\la_A)],
&F_i [S(\la)]&=\sum_B q^{-d^B(\la)}[S(\la^B)],
\end{align*}
where the first sum is over all removable $i$-nodes $A$ for $\la$,
and the second sum is over all addable $i$-nodes $B$ for $\la$.
\end{Corollary}

\begin{proof}
This follows from (\ref{sact1}), 
Theorem~\ref{cat}(2) and Theorem~\ref{ids}.
\end{proof}

\subsection{Ungraded decomposition numbers in characteristic zero}
Combining Theorem~\ref{ids} with Theorem~\ref{withlabels} (which we
recall was a reformulation of the
geometric Theorem~\ref{maina}),
we recover 
the following result
which computes 
decomposition numbers of Specht modules in characteristic zero.
These decomposition numbers 
were computed originally by Ariki
in \cite{Ariki}
in his proof of
the Lascoux-Leclerc-Thibon conjecture from \cite{LLT}
(generalized to higher levels).

\begin{Theorem}[Ariki]\label{ungradedllt}
Assume that $\cha F = 0$.
For any $\mu \in \Par_\alpha$ we have that
$$
[\underline{S}(\mu)] = \sum_{\la \in \RPar_\alpha} d_{\la,\mu}(1)
[\underline{D}(\la)]
$$
in the Grothendieck group $[\rep{H^\La_\alpha}]$,
where $d_{\la,\mu}(1)$ denotes the polynomial
from (\ref{qdec}) evaluated at $q=1$.
In other words, for $\mu \in\Par_\alpha$ and $\la \in \RPar_\alpha$,
we have that
$$
[\underline{S}(\mu):\underline{D}(\la)] = d_{\la,\mu}(1),
$$
\end{Theorem}

\begin{proof}
Let $\underline{S}_\mu := 1 \otimes S_\la \in V(\La)_\Z^*$
denote the standard monomial from (\ref{stmn}) specialized
at $q=1$.
By (\ref{qdec}) at $q=1$, we have that
$$
\underline{S}_\mu = \sum_{\la \in \RPar_\alpha}
d_{\la,\mu}(1) \underline{D}_\la.
$$
Now apply $\underline{\dd}^{-1}$
and use Theorems~\ref{ids} and \ref{withlabels}.
\end{proof}

In \cite{Ariki}, Ariki formulated his results in different terms,
involving lifting projectives from $H^\La_d$
to the semisimple Hecke algebras
whose irreducible representations were classified in \cite{AK}.
The Specht modules in our formulation of the above theorem
were not introduced in full generality 
until \cite{DJM} (after the time of \cite{Ariki}).
They are canonical ``modular reductions'' of the irreducible representations
of the aforementioned semisimple Hecke algebras. 
Invoking a form of Brauer reciprocity,
 Ariki's results can be reformulated equivalently in terms
of decomposition numbers of Specht modules, as we have done above.

Putting this technical difference aside,
Theorem~\ref{ungradedllt} is still not strictly
the same as Ariki's original theorem from \cite{Ariki}, 
since we are using the
parametrization of irreducible modules coming from the crystal
graph, whereas Ariki was implictly using a parametrization coming from the
triangularity properties of the decomposition matrices of Specht modules.
We discuss this subtle labelling issue in the next subsection.

\subsection{Another classification of irreducible representations}\label{sanother}
The general theory of cellular algebras leads
to alternative way to classify the irreducible 
$H^\La_\alpha$-modules, which was worked out originally
by Ariki in \cite{Aclass}.
In the following theorem, we reprove the main points of this
alternative classification, keeping track of gradings as we go.

\begin{Theorem}\label{altclass}
For $\la \in \RPar_\alpha$,
the graded Specht module $S(\la)$ has irreducible head denoted
$\dot{D}(\la)$.
Moreover:
\begin{itemize}
\item[(1)]
The modules
$\{\dot{D}(\la)\langle m\rangle\:|\:\la \in \RPar_\alpha, m \in \Z\}$
give a complete set of pairwise non-isomorphic irreducible graded
$R^\La_\alpha$-modules.
\item[(2)] 
For $\la \in \Par_\al$, we have in $[\Rep{R^\La_\alpha}]$ that
$$
[S(\la)] = \left\{
\begin{array}{ll}
[\dot{D}(\la)] + (*)&\text{if $\la$ is restricted,}\\
(*)&\text{otherwise,}
\end{array}\right.
$$
where $(*)$ denotes a 
$\Z[q,q^{-1}]$-linear combination of $[\dot{D}(\mu)]$'s for
$\mu \lhd \la$.
\item[(3)]
We have that
$\dot{D}(\la)^\circledast \cong \dot{D}(\la)$ 
for each $\la \in \RPar_\alpha$.
\end{itemize}
\end{Theorem}

\begin{proof}
Recall from Hypothesis~\ref{assump}(2) (which was verified in $\S$\ref{tzero}
in the case $e=0$ and $\S$\ref{construct} in the case $e > 0$) that
$\lexeq$ is a total order on $\Par$ refining 
the partial order $\preceq$.
So, applying 
$\dd^{-1}$ to Theorem~\ref{tri}(1)--(3) and using
Theorems~\ref{ids} and \ref{cat}(3), we deduce:
\begin{itemize}
\item[(a)] The classes $\{[S(\la)]\:|\:\la \in \RPar_\alpha\}$
are linearly independent.
\item[(b)] For $\la \in \Par_\alpha \setminus \RPar_\alpha$, 
we can express $[S(\la)]$ as a $q\Z[q]$-linear combination of
$[S(\mu)]$'s for $\mu \in \RPar_\alpha$ with $\mu \lex \la$.
\item[(c)] For $\la \in \RPar_\alpha$,
$[S(\la)] - [S(\la)^\circledast]$ is a
$\Z[q,q^{-1}]$-linear combination of $[S(\mu)]$'s
for $\mu \in \RPar_\alpha$ with $\mu \lex \la$.
\end{itemize}

Now we forget gradings for a moment. Recall that $H^\La_\alpha$ is a cellular
algebra with weight poset $(\Par_\alpha, \unlhd)$, and the Specht modules
are its 
cell modules. 
We claim that the cell modules $\underline{S}(\la)$
for $\la \in \RPar_\alpha$ have irreducible
head denoted $\underline{\dot{D}}(\la)$,
the modules
$\{\underline{\dot{D}}(\la)\:|\:\la \in \RPar_\alpha\}$
give a complete set of pairwise non-isomorphic irreducible 
$H^\La_\alpha$-modules, and 
$$
[\underline{S}(\la)] = \left\{
\begin{array}{ll}
[\underline{\dot{D}}(\la)] + (*)&\text{if $\la$ is restricted,}\\
(*)&\text{otherwise,}
\end{array}\right.
$$
for any $\la \in \Par_\alpha$,
where $(*)$ denotes a 
linear combination of $[\underline{\dot{D}}(\mu)]$'s for
$\mu \lhd \la$.
To prove the claim, recall by the general theory
of cellular algebras from \cite{GL} that certain of the
cell modules are distinguished, the distinguished cell modules
have irreducible heads which give a complete set
of non-isomorphic irreducible modules, and finally every composition
factor of an arbitrary cell module $\underline{S}(\la)$
is isomorphic to the irreducible head of a distinguished cell
module $\underline{S}(\mu)$ for $\mu \unlhd \la$.
Therefore to prove the claim it suffices to show that the distinguished
cell modules are the $\underline{S}(\la)$'s indexed by 
$\la \in \RPar_\al$.
Proceed by induction on the
total order $\lexeq$ that refines $\unlhd$.
For the induction step, 
consider $\underline{S}(\la)$ for $\la \in \Par_\alpha$.
If $\la$ is not restricted then $[\underline{S}(\la)]$
is a sum of earlier $[\underline{S}(\mu)]$'s by (b), so 
$\underline{S}(\la)$ cannot be distinguished.
If $\la$ is restricted then $[\underline{S}(\la)]$
is not a sum of earlier $[\underline{S}(\mu)]$'s by (a),
so $\underline{S}(\la)$ must be distinguished. The claim follows.

Re-introducing the grading using Lemmas~\ref{no1}--\ref{no3},
it follows from the claim that $S(\la)$ has irreducible head
$\dot{D}(\la)$ for each $\la \in \RPar_\al$ such that
$\underline{\dot{D}(\la)} \cong 
\underline{\dot{D}}(\la)$.
Moreover, (1) and (2) hold. It remains to deduce (3).
We certainly have that $\dot{D}(\la)^\circledast
\cong \dot{D}(\la)\langle m\rangle$ for some $m \in \Z$.
Now (2) gives us that
$[S(\la)] - [S(\la)^\circledast]$
is equal to $(1-q^m) [\dot D(\la)] + (*)$
where $(*)$ is a $\Z[q,q^{-1}]$-linear combination of
$[\dot D(\mu)]$'s for $\mu \in \RPar_\alpha$
with $\mu \lex \la$.
On the other hand (c) gives that
$[S(\la)] - [S(\la)^\circledast]$
is a $\Z[q,q^{-1}]$-linear combination just
of the $[\dot D(\mu)]$'s.
Hence $1-q^m = 0$, so $m=0$ as required.
\end{proof}

\begin{Corollary}\label{adj}
For each $\mu \in \RPar_\alpha$, we have that
$$
\dd^{-1}(D_\mu) = \sum_{\la \in \RPar_\alpha} a_{\la,\mu}(q) [\dot{D}(\la)]
$$
for some unique
bar-invariant Laurent polynomials $a_{\la,\mu}(q) \in \Z[q,q^{-1}]$
such that $a_{\mu,\mu}(q) = 1$ and $a_{\la,\mu}(q) = 0$ unless $\la \lexeq \mu$.
\end{Corollary}

\begin{proof}
Expand $D_\mu$ in terms of
$S_\nu$'s using Theorem~\ref{tri} (recalling Hypothesis~\ref{assump}(2)).
Then apply $\dd^{-1}$ and use Theorem~\ref{ids} to
get a linear combination of $[S(\nu)]$'s.
Finally replace each $[S(\nu)]$ with $[\dot{D}(\la)]$'s
using Theorem~\ref{altclass}(2).
This yields an expression of the form
$\sum_{\la \in \RPar_\alpha}
a_{\la,\mu}(q) [\dot{D}(\la)]$ such that $a_{\mu,\mu}(q) = 1$
and $a_{\la,\mu}(q) = 0$ unless $\la \lexeq \mu$.
As $D_\mu$ is bar-invariant, this expression is too,
so all the $a_{\la,\mu}(q)$'s are bar-invariant
thanks to Theorem~\ref{altclass}(3).
\end{proof}

\begin{Remark}
In Theorem~\ref{adjt} below we will show further that
$a_{\la,\mu}(q) = 0$ unless $\la \unlhd \mu$, and that
all the coefficients of $a_{\la,\mu}(q)$ are non-negative integers.
\end{Remark}

Now that we have two different parametrizations of irreducible
representations, one from Theorem~\ref{clas},
the other from Theorem~\ref{altclass}, we 
must address the problem of 
identifying the two labellings; eventually
it will emerge that 
\begin{equation}\label{idprob}
D(\la) \cong \dot{D}(\la)
\end{equation}
for each $\la$.
In level one, this fact has an elementary proof 
by-passing the geometric Theorem~\ref{maina}; see \cite{KBrIII, BBr}.
However, in higher levels, this identification turns out to be surprisingly
subtle and the only known proofs when $e > 0$ rely ultimately on geometry.
The labelling problem for higher levels 
was solved originally
by Ariki in \cite{Abranch}. 
The first author was already aware 
at that time of a slightly different argument to solve the same 
problem (see \cite[footnote 8]{Abranch}),
which we present below.
We begin in this subsection by solving the identification problem
in characteristic zero; see Theorem~\ref{idb} for the general case.

\begin{Theorem}\label{ida}
Assume that $\cha F = 0$.
For every $\la \in \RPar_\al$, we have that 
$D(\la) \cong \dot{D}(\la)$,
where $D(\la)$ is as in Theorem~\ref{clas} and
$\dot{D}(\la)$ is as in Theorem~\ref{altclass}.
\end{Theorem}

\begin{proof}
In view of Theorems \ref{clas}(1) and \ref{altclass}(3),
it suffices to prove this in the ungraded setting, i.e.
we need to show that $\underline{D}(\la)
\cong \underline{\dot{D}}(\la)$ for
each $\la \in \RPar_\al$.
By Theorems~\ref{altclass}(2) and \ref{ungradedllt}, using
also Theorem~\ref{tri} and Hypothesis~\ref{assump}(2) for the triangularity
in the second case, we have 
that
\begin{align*}
[\underline{S}(\la)] 
&= [\underline{\dot{D}}(\la)] + (\text{a $\Z$-linear
combination of $\underline{\dot{D}}(\mu)$'s for $\mu \lex \la$}),\\
[\underline{S}(\la)] 
&= [\underline{D}(\la)] + (\text{a $\Z$-linear
combination of $\underline{D}(\mu)$'s for $\mu \lex \la$}),
\end{align*}
for each $\la \in \RPar_\al$.
By induction on the lexicographic ordering, we deduce from this that
$[\underline{\dot{D}}(\la)] = [\underline{D}(\la)]$,
and the corollary follows.
\end{proof}

\subsection{Graded decomposition numbers in characteristic zero}\label{sgdn}
We can now prove the graded versions of Theorems~\ref{withlabels}
and \ref{ungradedllt}. 
These statements should be viewed as a graded version of 
the Lascoux-Leclerc-Thibon conjecture
(generalized to higher levels).

\begin{Theorem}\label{late}
Assume that $\cha F = 0$.
For each $\la \in \RPar_\alpha$, we have that
$\db(Y_\la) = [Y(\la)]$ and $\dd([D(\la)]) = D_\la$,
where $\db$ and $\dd$ are the maps from Theorem~\ref{cat}.
\end{Theorem}

\begin{proof}
It suffices to prove the second statement, 
since the first follows from it by dualizing as in the proof of Theorem~\ref{withlabels}.
By Corollary~\ref{adj} and Theorem~\ref{ida},
we have that
\begin{equation}\label{aha}
\dd^{-1} (D_\la)
=
[D(\la)] + \sum_{\mu \in \RPar_\al,\ \mu\lex \la} a_{\mu,\la}(q)[D(\mu)]
\end{equation}
for some bar-invariant Laurent polynomials $a_{\mu,\la}(q) \in
\Z[q,q^{-1}]$.
Moreover we know from Theorem~\ref{withlabels}
that $a_{\mu,\la}(1) = 0$.

Now we proceed to show that by induction on the lexicographic
ordering that $\dd([D(\la)]) = D_\la$
for all $\la \in \RPar_\alpha$.
When $\la$ is minimal, this is immediate from (\ref{aha}).
In general, we have that by (\ref{qdec}), (\ref{aha}),
Theorems~\ref{ids} and \ref{tri},
and the induction hypothesis that
\begin{align*}
[S(\la)] &= \dd^{-1} (S_\la)
= 
\dd^{-1}\bigg(
D_\la+ \sum_{\mu \in \RPar_\al,\ \mu \lex \la}
d_{\mu,\la}(q) D_\mu\bigg)\\
&=
[D(\la)] + \sum_{\mu\in\RPar_\al,\ \mu\lex\la}
(d_{\mu,\la}(q)+a_{\mu,\la}(q))  [D(\mu)]
\end{align*}
for every $\la \in \RPar_\alpha$.
Now consider the coefficient 
$d_{\mu,\la}(q)+a_{\mu,\la}(q)$
in this expression
for any
$\mu \in \RPar_\al$ with $\mu \lex \la$.
As this is the decomposition of a module
in the Grothendieck group, 
all coefficients of $d_{\mu,\la}(q)+a_{\mu,\la}(q)$
are non-negative integers.
As $d_{\mu,\la}(q) \in q \Z[q]$,
we deduce that the $q^0, q^{-1}, q^{-2},\dots$
coefficients of $a_{\mu,\la}(q)$ are non-negative
integers. As $a_{\mu,\la}(q)$ is bar-invariant
it follows that all its coefficients are non-negative.
Finally as $a_{\mu,\la}(1) = 0$ we get that
$a_{\mu,\la}(q) = 0$ too.
This holds for all $\mu$, so (\ref{aha}) implies that
$\dd^{-1} (D_\la) = [D(\la)]$, as required.
\end{proof}

\begin{Corollary}\label{maindecthm}
Assume that $\cha F = 0$.
For $\mu \in \Par_\al$, we have that
$$
[S(\mu)] = \sum_{\la \in \RPar_\alpha}
d_{\la,\mu}(q) [D(\la)].
$$
In other words, for $\mu \in\Par_\alpha$ and $\la \in \RPar_\alpha$,
we have that
$$
[S(\mu):D(\la)]_q = d_{\la,\mu}(q).
$$
Moreover, for all such $\la,\mu$, we have that
$d_{\la,\mu}(q) = 0$ unless $\la \unlhd \mu$.
\end{Corollary}

\begin{proof}
The first two statements follow from Theorems~\ref{late} and \ref{ids},
combined with the definition (\ref{qdec}).
The final statement 
follows from Theorems~\ref{ida} and \ref{altclass}(2).
\end{proof}

\subsection{Graded adjustment matrices}\label{sgam}
In this subsection we complete
the proof of (\ref{idprob}) for fields $F$ of positive characteristic.
Recall we have already established this in the case $\cha F = 0$
in Theorem~\ref{ida}.
We will deduce the result in general from the characteristic zero case
by a base change argument.

So assume now that $F$ is of characteristic $p > 0$, keeping all other notation
as at the beginning of section 4.
Assume we are given $\alpha \in Q_+$ with $\height(\alpha) = d$.
Let $\widehat\xi \in \C^\times$ be a primitive $e$th root of unity 
(or any non-zero element that is not a root of unity if $e=0$).
As well as the algebra $R^\La_\alpha$ over $F$,
we consider the corresponding algebra
defined from the parameter $\widehat\xi$ over the ground field $\C$.
To avoid confusion we denote it by $\widehat{R}^\La_\alpha$,
and denote the graded Specht and irreducible modules
for $\widehat{R}^\La_\alpha$ by 
$\widehat{S}(\la)$ and $\widehat{D}(\la)$, respectively.

In \cite[$\S$6]{BKyoung}, we explained a general procedure to
reduce the 
irreducible $\widehat{R}^\La_\alpha$-module $\widehat{D}(\la)$
modulo $p$ to obtain an
$R^\La_\alpha$-module with the same $q$-character,
for each $\la \in \RPar_\al$.
There is some freedom in this procedure related to choosing 
a lattice in $\widehat{D}(\la)$. We can make an essentially
unique choice as follows. Let
$v_\la$ denote the image under
some surjection $\widehat{S}(\la) \twoheadrightarrow \widehat{D}(\la)$
of the homogeneous basis vector 
$v_\T \in \widehat{S}(\la)$, where $\T$ is the ``initial'' standard $\la$-tableau
obtained by writing the numbers $1,2,\dots,d$ in order along rows
starting with the top row.
By \cite[$\S$6.2]{BKW},
$\widehat{S}(\la)$ is generated as an $\widehat{R}^\La_\alpha$-module
by this vector $v_\T$, hence $v_\la \in \widehat{D}(\la)$ is non-zero.
Now let
\begin{equation}\label{jam}
J(\la) := F \otimes_{\Z} L
\end{equation}
where $L \subset \widehat{D}(\la)$ denotes the $\Z$-span of the
vectors
$\psi_{r_1} \cdots \psi_{r_m} 
y_1^{n_1} \cdots y_d^{n_d} v_\la$
for all $m \geq 0, 1 \leq r_1, \dots, r_m < d$ and $n_1,\dots,n_d \geq 0$.
By \cite[Theorem 6.5]{BKyoung}, $L$ is a lattice in $\widehat{D}(\la)$,
and $J(\la)$
is a well-defined graded $R^\La_\alpha$-module
with
$y_r \in R^\La_\alpha$ acting as $1 \otimes y_r$,
$\psi_r$ acting as $1 \otimes \psi_r$, and
$e(\bi)$ acting as $1 \otimes e(\bi)$.

\begin{Lemma}\label{J}
For each $\la \in \RPar_\al$, $J(\la)$ has the same $q$-character
as $\widehat{D}(\la)$.
Hence, $\dd([J(\la)]) = D_\la$.
\end{Lemma}

\begin{proof}
The fact that $J(\la)$ has the same $q$-character
as $\widehat{D}(\la)$ is immediate from the construction because,
for $L$ as in (\ref{jam}), each
$e(\bi) L$ is a graded lattice in $e(\bi) \widehat{D}(\la)$.
To show that 
$\dd([J(\la)]) = D_\la$, it suffices by Theorem~\ref{late}
to show that $\dd([M]) = \dd([N])$ in $V(\La)_\Laurent^*$ whenever we are given 
$M \in \Rep{R^\La_\alpha}$
and $N \in \Rep{\widehat{R}^\La_\alpha}$ 
with the same $q$-characters.
To see this, use Theorem~\ref{ch}, Corollary~\ref{sbase}
and the observation from (\ref{chspecht}) 
that graded Specht modules over $R^\La_\alpha$
and $\widehat{R}^\La_\alpha$ have the same $q$-characters
to reduce to checking 
the statement in the special case that $M = S(\la)$ 
and $N = \widehat{S}(\la)$
for some $\la \in \RPar_\al$.
Then apply Theorem~\ref{ids}.
\end{proof}

\begin{Theorem}\label{adjt}
The bar-invariant Laurent polynomials $a_{\la,\mu}(q)$ from Corollary~\ref{adj}
have the property that 
$a_{\la,\mu}(q) = 0$ unless $\la \unlhd \mu$. Moreover,
for each $\mu \in \RPar_\alpha$, we have that
$$
[J(\mu)] = [\dot{D}(\mu)] + \sum_{\la \in \RPar_\alpha,\ \la \lhd \mu} a_{\la,\mu}(q) [\dot{D}(\la)].
$$
Hence, all coefficients of
$a_{\la,\mu}(q)$ are non-negative integers,
and
 we have that
$$
[S(\mu):\dot D(\la)]_q
= 
\sum_{\nu \in \RPar_\al} a_{\la,\nu}(q) d_{\nu,\mu}(q)
$$
for any $\la \in \RPar_\alpha$ and $\mu \in \Par_\alpha$.
\end{Theorem}

\begin{proof}
The first statement follows by repeating the proof of Corollary~\ref{adj},
using the stronger result established in Corollary~\ref{maindecthm}
that $d_{\la,\mu}(q) = 0$ unless $\la \unlhd \mu$
to replace the total order $\lexeq$ by the partial order $\unlhd$.
By Lemma~\ref{J}, we have that
$\dd([J(\mu)]) = D_\mu$.
Using this, the next statement of 
the theorem follows from 
Corollary~\ref{adj}.
For the final statement, we have that by Theorem~\ref{ids} and
(\ref{qdec}) that
$$
[S(\mu)] = \sum_{\nu \in \RPar_\al} d_{\nu,\mu}(q) \dd^{-1}(D_\nu).
$$
Now expand each $\dd^{-1}(D_\nu)$ using
the formula from 
Corollary~\ref{adj}.
\end{proof}

We refer to the matrix $(a_{\la,\mu}(q))_{\la,\mu \in \RPar_\alpha}$
as the {\em graded adjustment matrix}.
For level one and 
$\xi = 1$, our graded adjustment matrix specializes at $q=1$
to the adjustment matrix defined originally by James in the modular
representation theory of symmetric groups.
Curiously we did not yet find an example in which $a_{\la,\mu}(q) \notin \Z$;
this is related to a question raised by
Turner in the introduction of \cite{T}.
Now we can complete the identification of the two labellings of irreducible
representations in positive characteristic. 

\begin{Theorem}\label{idb}
Assume that $\cha F > 0$.
For every $\la \in \RPar_\alpha$,
we have that $D(\la) \cong \dot{D}(\la)$, where $D(\la)$ is as in 
Theorem~\ref{clas} and $\dot{D}(\la)$ is as in Theorem~\ref{altclass}.
\end{Theorem}

\begin{proof}
We first claim 
for any $\la \in \RPar_\alpha$ that $[J(\la):D(\la)]_q = 1$.
To see this, let $\bi \in I^\alpha$ be an extremal sequence for
$J(\la)$ in the sense of $\S$\ref{sex}.
As $J(\la)$ has the same $q$-character as the irreducible
$\widehat{R}^\La_\alpha$-module $\widehat{D}(\la)$,
$\bi$ must also be an extremal sequence for $\widehat{D}(\la)$. 
Now apply Lemma~\ref{C241201} twice, once for $R^\La_\alpha$ and once
for $\widehat{R}^\La_\alpha$, to get that
$\la = \tilde f_{i_d} \cdots \tilde f_{i_1} \varnothing$ and
$[J(\la):D(\la)] = 1$.

Using the claim and Theorem~\ref{adjt},
it is now an easy exercise to show that $[\dot{D}(\mu)]
= [D(\mu)]$, proceeding by induction on the dominance ordering.
The theorem follows.
\end{proof}

\subsection{\boldmath The Khovanov-Lauda conjecture in type $A$}\label{slast}
Theorem~\ref{late} 
combined with Lemma~\ref{party}
proves for all type $A$ quivers (finite or affine) a conjecture of 
Khovanov and Lauda formulated in \cite[$\S$3.4]{KL1};
see also \cite{BS} for an elementary proof in a very special case.
In this subsection we record one consequence
which is implicit in \cite{KL1}.
Apart from the case $e=2$, the main 
result of this subsection is also proved 
in \cite{VV3} by a more direct method (which includes 
all other simply-laced types, not just type $A$). 

Let $\mathbf B = \dot\bigcup_{\alpha \in Q_+} \mathbf B_\alpha$ 
be the canonical basis for $\mathbf f = \bigoplus_{\alpha \in Q_+}
\mathbf f_\alpha$
as in 
\cite[$\S$14.4]{Lubook}.
Let $U_q(\g)^-$ be the subalgebra of $U_q(\g)$ generated
by the $F_i$'s.
There is an isomorphism
\begin{equation}
\mathbf f \stackrel{\sim}{\rightarrow} U_q(\g)^-, \qquad x \mapsto x^-
\end{equation}
such that $\theta_i^- := F_i$.
In view of the results of \cite[$\S$14.4]{Lubook},
$\mathbf B$ is the unique weight
basis for $\mathbf f$ such that
the following holds for every  $x \in \mathbf B$
and every dominant integral weight $\La$: the vector
$x^- v_\La$ is either zero or it is an element of the canonical
basis of $V(\La)$.

We have already observed that every
irreducible graded $R_\alpha$-module can be shifted in degree
so that it is self-dual with respect to the duality $\circledast$;
see \cite[$\S$3.2]{KL1}.
In view of Corollary~\ref{dualpims}, it follows that
every indecomposable projective
graded $R_\alpha$-module $P$ can be shifted in degree so that it is
self-dual with respect to the duality $\#$.
We say simply that $P$ is a {\em self-dual 
projective} if that is the case. So the head of an indecomposable
self-dual projective is a self-dual irreducible.

\begin{Theorem}\label{klcon}
Assume that $\cha F = 0$.
For every $\alpha \in Q_+$,
the isomorphism $\ga:\mathbf f_\alpha \rightarrow [\Proj{R_\alpha}]$ from Theorem~\ref{klthm}
maps 
$\mathbf B_\alpha$
to the basis of
$[\Proj{R_\alpha}]$ 
arising
from the isomorphism classes of the self-dual indecomposable projective
graded $R_\alpha$-modules.
\end{Theorem}

\begin{proof}
Let $\sigma:\mathbf f \rightarrow \mathbf f$ be the linear
anti-automorphism with $\sigma(\theta_i) = \theta_i$ for all $i$.
It is well known that $\sigma$ maps the canonical basis of $\mathbf f$
to itself.
Let $P$ be a self-dual indecomposable projective graded $R_\alpha$-module.
Let $x \in \mathbf f$ be its pre-image under $\gamma$.
To prove that $x \in \mathbf B$, we show equivalently
that $\sigma(x) \in \mathbf B$.
By the characterization of $\mathbf B$
recalled before the statement of the theorem,
this follows if we can show for any $\La$ that
$\sigma(x)^- v_\La$ is either zero or an element of the canonical
basis of $V(\La)$.
Since $P$ is self-dual, we get from Theorem~\ref{klthm}(3)
that $x$, hence $\sigma(x)$, is bar-invariant. Therefore it is enough just
to show that $\sigma(x)^- v_\La$ is either zero or an element
of the canonical basis of $V(\La)$ up to scaling by a
power of $q$.
Recalling the map $\flat$ from
(\ref{flatmap}), note that $\sigma(x)^- v_\La$
 is equal to
$x^\flat v_\La$ up to scaling by a power of $q$.
So applying Lemma~\ref{party},
we are reduced to showing that $x^\flat v_\La$
is either zero or an element of the quasi-canonical basis
of $V(\La)$ up to scaling by a power of $q$.
Finally, by Proposition~\ref{id1}, we have that
$x^\flat v_\La = \de^{-1}(\pr\, P)$.
Clearly $\pr\, P$ is either zero or a projective indecomposable
$R^\La_\alpha$-module. Moreover when it is non-zero,
Theorem~\ref{late}
gives that $\de^{-1} (\pr\, P)$
is an element of the quasi-canonical basis
of $V(\La)$ up to scaling by a power of $q$.
This completes the proof.
\end{proof}

\end{document}